\documentclass[10pt, reqno]{amsart}
\numberwithin{equation}{section}
\usepackage{amssymb,latexsym,amsfonts}
\usepackage{bm}

\usepackage{hyperref}

\DeclareFontFamily{OT1}{pzc}{}
\DeclareFontShape{OT1}{pzc}{m}{it}{<-> s*[1.200]pzcmi7t}{}
\DeclareMathAlphabet{\mathscr}{OT1}{pzc}{m}{it}

\newcommand{\R}{\mathbb{R}}
\newcommand{\C}{\mathbb{C}}
\newcommand{\wh}[1]{\widehat{#1}}

\newcommand{\eps}{\varepsilon}

\newcommand{\supp}{\text{supp}}

\newcommand{\HHH}{\mathcal{H}}

\newcommand{\EEE}{\mathbf{E}}

\newcommand{\Z}{\mathbb{Z}}
\newcommand{\qq}{\mathbf{q}}

\newcommand{\indic}{1\!\!1}

\newcommand{\zzz}{\mathbf{z}}
\newcommand{\vvv}{\mathbf{v}}
\renewcommand{\O}{\mathcal{O}}
\newcommand{\mmu}{\bm{\mu}}
\newcommand{\TT}{\mathbf{T}}

\newtheorem{theorem}{Theorem}[section]

\newtheorem{prop}[theorem]{Proposition}
\newtheorem*{prop_sans_no}{Main Proposition}
\newtheorem{lemma}[theorem]{Lemma}
\newtheorem{coro}[theorem]{Corollary}

\newtheorem{proposition}[theorem]{Proposition}

\theoremstyle{definition}
\newtheorem{definition}[theorem]{Definition}
\newtheorem{remark}[theorem]{Remark}

\theoremstyle{remark}

\begin{document}

\title[Endline null form estimates]{Sharp null form estimates on endline geometric conditions of the cone}
\author[J. Yang]{Jianwei Urbain Yang${}^\dag$}
\email{jw-urbain.yang@bit.edu.cn}
\address{Jianwei Urbain Yang, 
Department of Mathematics, Beijing Institute of Technology, Beijing 100081, P. R. China}

\thanks{$^\dag$Department of Mathematics, Beijing Institute of Technology. The author is  supported by NSFC grant No. 11901032 and Research fund program for young scholars of Beijing Institute of Technology.}

\date{}



\begin{abstract}
	We prove $\HHH^{\alpha_1}\times\HHH^{\alpha_2}\to L^q_tL^r_x$ null form estimates for solutions to homogeneous wave equations with  $(q,r)$ on the endline of the condition concerning geometry of the cone, except the critical index. 
	This extends the previous endpoint result of
	\emph{Tao},  Math. Z. 238,  no. 2, 215--268, (2001)  in symmetric-norms  to mixed-norms  and	improves the local-in-time result of \emph{Tataru}, MR1979953, to be global in the  setting of  constant variable coefficient equations, as well as the sharp off-endline 
    estimates established by \emph{Lee} and \emph{Vargas},
    Amer. J. Math
    130 (2008), no. 5, 1279--1326, to the borderline  with respect to the cone condition. 
	Our proof is based on the
     multiplier theory in mixed-norms,
    which ultimately reduces the question to 	
    a \emph{uniform} endline bilinear  restriction estimates
	including high-low frequency interactions for a family of conic type surfaces depending on a  parameter $\sigma$, which
	converges to the oblique cone as $\sigma\to 0$.
	We prove this uniform estimate by using 
	the enhanced version of the induction-on-scale method. 
\end{abstract}

\maketitle

\section{Introduction}\label{S:introduction}
\subsection{Motivation and background}
Let $n\ge 2$ be a fixed integer and 
$\phi$, $\psi$ be solutions of the
homogeneous wave equation in $\R^{1+n}=\R_t\times\R^n_x$
\begin{equation}
\label{eq:wq}
\Box \phi=0,\;\quad\Box \psi=0,\;\quad (\Box=-\partial_t^2+\Delta_x,\;t\in\R,\;x\in\R^n)
\end{equation}
subject to the initial conditions at $t=0$
$$
\phi[0]:=\bigl(\phi(0,\cdot),\partial_t\phi(0,\cdot)\bigr)
=(\phi_0,\phi_1),\quad
\psi[0]:=\bigl(\psi(0,\cdot),\partial_t\psi(0,\cdot)\bigr)
=(\psi_0,\psi_1)\,.
$$
For $\alpha_1,\alpha_2\in\R$,  define
$$
\bigl\|\phi[0]\bigr\|_{\HHH^{\alpha_1}}
=\Bigl(\|\phi_0\|_{\dot{H}^{\alpha_1}}^2+\|\phi_1\|^2_{\dot{H}^{\alpha_1-1}}\Bigr)^{\frac{1}{2}},\;
\bigl\|\psi[0]\bigr\|_{\HHH^{\alpha_2}}
=\Bigl(\|\psi_0\|_{\dot{H}^{\alpha_2}}^2+\|\psi_1\|^2_{\dot{H}^{\alpha_2-1}}\Bigr)^{\frac{1}{2}},
$$
where $\dot{H}^\alpha=\dot{H}^\alpha(\R^n)$ is the homogeneous $L^2-$based Sobolev space of order $\alpha$.\medskip

Let $D_0, D_+,D_-$ be the Fourier multipliers defined by
$$
\widetilde{D_0 f}(\tau,\xi)
=|\xi|\, \widetilde{f}(\tau,\xi),
$$
$$
\widetilde{D_+ f}(\tau,\xi)
=\bigl( |\tau|+|\xi|\bigr)\,\widetilde{f}(\tau,\xi),
$$
$$
\widetilde{D_- f}(\tau,\xi)
=\bigl| |\tau|-|\xi|\bigr|\,\widetilde{f}(\tau,\xi),
$$
where $\,\widetilde{\cdot}\,$ denotes the space-time 
Fourier transform in  sense of distributions on $\R^{1+n}$ and 
 $(\tau,\xi)$ are the Fourier variables corresponding to $(t,x)$.\medskip
 
 We are interested in the validity of the null form estimates 
 for $\phi$ and $\psi$
 \begin{equation}
 \label{eq:NF}
 \bigl\| D_0^{\beta_0} D_+^{\beta_+} D_-^{\beta_-} (\phi\psi)\bigr\|_{L^q_t(\R;L^r_x(\R^n))}
 \le C\,\bigl\|\phi[0]\bigr\|_{\HHH^{\alpha_1}}\, \bigl\|\psi[0]\bigr\|_{\HHH^{\alpha_2}},
 \end{equation} 
 for some finite constant $C$,
 with $1\le q,r\le \infty$ such that 
 the following conditions on $(q,r,\alpha_1,\alpha_2,\beta_0,\beta_+,\beta_-)$ 
 are imposed:\medskip
 
$\bullet$ Scaling invariance
\begin{equation}
\label{eq:LV-2}
\beta_0+\beta_++\beta_-=\alpha_1+\alpha_2+\frac{1}{q}-n\Bigl(1-\frac{1}{r}\Bigr),
\end{equation}

$\bullet$ Geometry of the cone
\begin{equation}
\label{eq:LV-3}
\frac{1}{q}\le \frac{n+1}{2}\Bigl(1-\frac{1}{r}\Bigr),\;\;\;\;
\frac{1}{q}\le\frac{n+1}{4},
\end{equation}

$\bullet$ Concentration along null directions
\begin{equation}
\label{eq:LV-4}
\beta_-\ge \frac{1}{q}-\frac{n-1}{2}\Bigl(1-\frac{1}{r}\Bigr),
\end{equation}

$\bullet$ Low frequencies in $(++)$ interaction
\begin{equation}
\label{eq:LV-5}
\beta_0\ge \frac{1}{q}-n\Bigl(1-\frac{1}{r}\Bigr),
\end{equation}
\begin{equation}
\label{eq:LV-6}
\beta_0\ge \frac{2}{q}-(n+1)\Bigl(1-\frac{1}{r}\Bigr),
\end{equation}
\begin{equation}
\label{eq:LV-15}
\beta_0\ge \frac{2}{q}-n\Bigl(1-\frac{1}{r}\Bigr)-\frac{1}{2},
\end{equation}

$\bullet$ Low frequencies in $(+-)$ interaction
\begin{equation}
\label{eq:LV-7}
\alpha_1+\alpha_2\ge \frac{1}{q},
\end{equation}
\begin{equation}
\label{eq:LV-8}
\alpha_1+\alpha_2\ge \frac{3}{q}-n\Bigl(1-\frac{1}{r}\Bigr),
\end{equation}

$\bullet$ Interaction between high and low frequencies 
\begin{equation}
\label{eq:LV-9}
\alpha_1,\alpha_2\le \beta_-+\frac{n}{2},
\end{equation}
\begin{equation}
\label{eq:LV-10}
\alpha_1,\alpha_2\le \beta_-+\frac{n}{2}-\frac{1}{q}+\frac{n-1}{2}\Bigl(\frac{1}{2}-\frac{1}{r}\Bigr),
\end{equation}
\begin{equation}
\label{eq:LV-11}
\alpha_1,\alpha_2\le \beta_-+\frac{n}{2}-\frac{1}{q}+n\Bigl(\frac{1}{2}-\frac{1}{r}\Bigr),
\end{equation}
\begin{equation}
\label{eq:LV-12}
\alpha_1,\alpha_2\le \beta_-+\frac{n}{2}-\frac{1}{q}+n\Bigl(\frac{1}{2}-\frac{1}{r}\Bigr)+\Bigl(\frac{1}{2}-\frac{1}{q}\Bigr),
\end{equation}
\begin{equation}
\label{eq:LV-13}
\alpha_1,\alpha_2\le \beta_-+\frac{n+1}{2}-\frac{1}{q},
\end{equation}
\begin{equation}
\label{eq:LV-14}
\alpha_1,\alpha_2\le \beta_-+\frac{n}{2}-\frac{1}{q}+n\Bigl(\frac{1}{2}-\frac{1}{r}\Bigr)+\frac{1}{r}-\frac{1}{q},
\end{equation}

$\bullet$ Extra conditions for low frequencies of $(+-)$ interactions in low dimensions
\begin{equation}
\label{eq:LRV-1.16}
\alpha_1+\alpha_2\ge \frac{3}{q}+\frac{1}{r}-2,\quad n=3\,,
\end{equation}
\begin{equation}
\label{eq:LRV-1.17}
\alpha_1+\alpha_2\ge \frac{3}{q}+\frac{1}{2r}-\frac{5}{4},\quad n=2\,.
\end{equation}

That the conditions \eqref{eq:LV-2}-\eqref{eq:LV-6} 
and \eqref{eq:LV-7}-\eqref{eq:LV-12} are necessary 
for \eqref{eq:NF} to hold was described first 
by Foschi and Klainerman \cite{FoKl}, and 
it is conjectured  by the authors \cite{FoKl} that 
for $n\ge 2$ and $1\le q,r\le\infty$, 
these conditions should also be sufficient\footnote{Conjecture 14.16 in \cite{FoKl}.}.
The condition \eqref{eq:LV-4} is 
related to Lorentz invariance of the cone \cite{FoKl,TV-2}. 
In \cite{LeeVargas}, Lee and Vargas observed that 
in order for \eqref{eq:NF} to be true, \eqref{eq:LV-13}\eqref{eq:LV-14} and \eqref{eq:LV-15} 
are needed as well. 
Moreover, by testing a more refined example 
in the low frequency interactions for the $(+-)$ case, 
Lee, Rogers and Vargas \cite{LeeRogersVargas} proved that
when $n=3,2$, \eqref{eq:NF} requires an additional conditions
\eqref{eq:LRV-1.16} and \eqref{eq:LRV-1.17} respectively.\smallskip

It is convenient to classify the conditions \eqref{eq:LV-2}-\eqref{eq:LRV-1.17} into three categories. 
We reserve the same terminology for  \eqref{eq:LV-2} and \eqref{eq:LV-3}, to which we shall refer as 
the first and second categories.  
We shall call \eqref{eq:LV-3} 
the \emph{cone condition} below for brevity. 
The conditions \eqref{eq:LV-4}-\eqref{eq:LRV-1.17} 
are packed  up as the third category, 
to which we refer as the \emph{multiplier conditions}.
\smallskip

 In the case when $q=r=2$, Foschi and Klainerman \cite[Theorem 1.1]{FoKl} has determined the exact conditions on $(\alpha_1,\alpha_2,\beta_0,\beta_+,\beta_-)$ for estimates \eqref{eq:NF} to be true in the $L^2_{t,x}(\R^{1+n})$ norm for all $n\ge 2$.

For  cases  beyond 
 bilinear-$L^2_{t,x}$ setting, 
the null form estimate 
encircles in some sense 
the bilinear Fourier restriction estimates 
for surfaces of disjoint conic type. 
For frequency localized waves
without Fourier multipliers,  bilinear estimates  with $q=r<2$ 
in the two dimensional case 
was first established by Bourgain  \cite{Bo95}  
under  separateness
of the frequency variables. 
This is a genuine bilinear restriction estimate on the cone,
 a remarkable  progress towards 
the conjecture of Klainerman and Machedon  \cite{FoKl,Wolff99}. 
The result was improved later by 
Tao, Vargas and Vega \cite{TVV} and 
Tao-Vargas \cite{TV-1,TV-2}, 
which shapes  nowadays the 
standard bilinear method to the restriction and 
Kakeya conjectures, see \cite{DeBook} for more investigations.
The first sharp result is obtained by Wolff \cite{Wolff99} 
on the cone, except the endpoint case, 
which is  settled by Tao \cite{TaoMZ}, 
by building up the wave-table theory  and exploiting energy concentrations in the physical space along
 null directions of the cone.  
Extensions of this result to the mixed-norms  on the endline $\frac{2}{q}=(n+1)\Bigl(1-\frac{1}{r}\Bigr)$
in  \eqref{eq:LV-3} with the exception of the critical index  $(q_c,r_c)$, which by definition reads
\begin{equation}
\label{eq:cri-ind}
q_c=\max\Bigl(1,\frac{4}{n+1}\Bigr),\quad r_c=\min\Bigl(2,\frac{n+1}{n-1}\Bigr),
\end{equation}
is obtained by Temur \cite{Temur}, 
improving the previous off-endline result 
due to Lee and Vargas \cite{LeeVargas}
to the borderline case. 
Note that the critical index is out of reach 
by the current method\footnote{We conjecture the two dimensional case should be available.}. 
Indeed, it shares a level of the same difficulty 
with the endpoint multilinear restriction conjecture of Bonnett-Carbery-Tao \cite{BCT}, 
a very difficult open question. 
Even as a weaker result, 
the endpoint multilinear Kakeya inequality 
can only be established through the intricate 
algebraic topological method by Guth \cite{guth}. 
A much simpler version of the proof is given by 
Carbery and Valdimarsson \cite{CaV}.
We find the endpoint Fourier restriction theorems become more interesting recently. We refer to \cite{GiMe} for its connexion with the  unrectifiability of measures and \cite{HKL,FPV,Hab} for their role in the Calder\'on problem.
\smallskip

The null form estimates \eqref{eq:NF} 
in $q=r\ge 2$ with nontrivial multipliers is first established by  Klainerman and Tataru \cite{KT}, while
the  $q=r<2$ case first  by Tao and Vargas \cite{TV-2}. 
More refined $L^p$-null form results in symmetric norms 
with optimal conditions were established by Tao \cite{TaoMZ}, based
on the endpoint  bilinear restriction theorem on the light cone
including the interaction between high-low frequencies.  
In \cite{LeeVargas}, Lee and Vargas proved 
the sharp null form estimates under the scaling
condition \eqref{eq:LV-2} and  that
\eqref{eq:LV-3}-\eqref{eq:LV-14}  
are all valid with strict inequalities for all $n\ge 4$, 
and for $n=2,3$ with a gap  
$\frac{4}{n+1}\le q\le \frac{4}{n}$ when $2<r\le \infty$. 
This gap is supplied  by Lee, Rogers and Vargas \cite{LeeRogersVargas} when $n=3$ 
incorporated with the new  condition \eqref{eq:LRV-1.16}, 
with strict inequality, 
whereas only a partial result 
was obtained in \cite{LeeRogersVargas} 
for the two dimensional case, 
under an apparently stronger condition $\alpha_1+\alpha_2>\frac{3}{q}-1$ when $r>2$.
This is related to an open problem concerning the endpoint bilinear Strichartz estimate for one-dimensional Schr\"odinger equations.\medskip

Note that all of the above results 
concerning \eqref{eq:NF} are established  
in the spacetime norms $L^q_tL^r_x$ 
with $(q,r)$ being \emph{off} the endline 
of the cone condition \eqref{eq:LV-3} and that 
all the multiplier conditions are assumed to hold 
with strict inequalities, only
except that in \cite[Theorem 17.3]{TaoMZ}, 
an endpoint result was obtained for symmetric norms $q=r\ge\frac{n+3}{n+1}$, 
and in this case, the low frequency 
$(++)$ interaction condition is allowed to be taken as an equality. 
This result was extended to mixed-norms  
on the endline by Tataru \cite{Tataru} for  
second order hyperbolic operators 
with rough coefficients, where the estimates are 
obtained locally in time. 
One of the difficulties for the global result in mixed-norms  on the endline
is the obstacle of using  Lorentz invariance. \medskip

The purpose of this paper is to prove 
the null form estimates \eqref{eq:NF} for all $(q,r)$ 
on the endline of the cone condition \eqref{eq:LV-3} except the critical index $(q_c,r_c)$, 
with  the low  frequency 
$(++)$ interaction condition  allowed to take  equality, and all the other relevant sharp multiplier conditions 
hold  with strict inequalities. 

\begin{theorem}
	\label{thm:-main}
	Let $n\ge 2$ and $1\le q,r\le\infty$.
	Assume the following conditions hold
	\begin{equation}
	\label{eq:LV-2-el}
	\beta_0+\beta_++\beta_-=
	\alpha_1+\alpha_2-q^{-1}\Bigl(\frac{n-1}{n+1}\Bigr),
	\end{equation}
	\begin{equation}
	\label{eq:LV-3-el}
	\frac{1}{q}= \frac{n+1}{2}\Bigl(1-\frac{1}{r}\Bigr),\;\;\;\;
	\frac{1}{q}<\min\Bigl(1,\frac{n+1}{4}\Bigr),
	\end{equation}
	\begin{equation}
	\label{eq:LV-6-el}
	\beta_0\ge 0,
	\end{equation}
	\begin{equation}
	\label{eq:LV-4-el}
	\beta_-> \frac{2}{q(n+1)},
	\end{equation}
	\begin{equation}
	\label{eq:LV-8-el}
	\alpha_1+\alpha_2> q^{-1}\Bigl(\frac{n+3}{n+1}\Bigr),
	\end{equation}
	\begin{equation}
	\label{eq:LV-11-el}
	\alpha_1,\alpha_2< \beta_-+q^{-1}\Bigl(\frac{n-1}{n+1}\Bigr),
	\end{equation}
	\begin{equation}
	\label{eq:LV-12-el}
	\alpha_1,\alpha_2< \beta_-+\frac{1}{2}-\frac{2}{q(n+1)}.
	\end{equation}
	Then, there is a constant $C$ depending possibly on $(q,r,\alpha_1,\alpha_2,\beta_0,\beta_+,\beta_-)$ such that  \eqref{eq:NF} holds for all 
	$(\phi[0],\psi[0])\in
	\HHH^{\alpha_1}(\R^n)
	\times\HHH^{\alpha_2}(\R^n)$.
\end{theorem}

	On the endline of  \eqref{eq:LV-3}, 
	only a partial components of the necessary conditions 
	in \eqref{eq:LV-2}-\eqref{eq:LRV-1.17} that matter. 
	We selected them for sake of reading. The strict inequalities \eqref{eq:LV-4-el}--\eqref{eq:LV-12-el} are responsible for the summability on various of dyadic geometric series.
	In certain off-endline cases 
	such as $(q,r)=(2,2)$ and $(\infty,2)$, 
	the condition associated to  interactions of  low frequencies in the $(++)$ case  being strict inequality is  also
	\emph{necessary} \cite{FoKl}. 
	Moreover, in one of these two cases, 
	certain special choices of $(\alpha_1,\alpha_2,\beta_-)$ 
	on the equalities of some of the multiplier conditions 
	are known to be \emph{inadmissible}.
\smallskip

	The strict inequalities in \eqref{eq:LV-4-el} and \eqref{eq:LV-8-el} are used in the proof  to tackle the fake
	singularity  by angular compression, which is related to
    the cone property (c.f. Example 14.3 and Example 14.9 of \cite{FoKl}). It is employed as a stepstone for
  the bilinear Whitney  type dyadic decomposition along the angular direction incorporated with a rescaling argument.
  This strategy is of course rather crude.
	To push   these conditions to equalities,   
   one may have to prove certain bilinear square-function estimates with respect to 
	angular decompositions.
\smallskip

Theorem \ref{thm:-main} extends the  null form estimates 
at the endpoint in the symmetric norm  of Tao \cite{TaoMZ} 
to the whole endline apart from the critical index $(q_c,r_c)$, 
with the low-frequency interaction (++) condition \eqref{eq:LV-6-el}
being able to attain the equality. 
It also extends the local-in-time estimates of 
Tataru \cite{Tataru} to be global-in-time 
for the constant coefficient wave equations. 
In particular, it can be added to the picture of
sharp null form estimates 
off  the   endline, attributed to Lee-Vargas \cite{LeeVargas} and Lee-Rogers-Vargas \cite{LeeRogersVargas}, where the condition on $\beta_0$ was required to  hold with strict inequality.

\subsection{Endline bilinear estimates for a family of conic type surfaces} 
To prove Theorem \ref{thm:-main}, 
we adopt a hybrid  of the  
dyadic decomposition in the frequency space 
introduced  in \cite{TV-2} and \cite{TaoMZ,LeeVargas}, 
which allows to  reduce \eqref{eq:NF}  
 to a \emph{uniform} endline bilinear  restriction
estimates in Theorem \ref{pp:sigma-0}  below for a family of conic type surfaces depending on a small parameter $\sigma>0$.

This {uniform} bilinear estimate
is  interesting in its own right. 
As a by-product,
we investigate by the end of this paper
 how one can combine
its proof  with the method of descent to get the endpoint  bilinear restriction estimates on the unit sphere, an open question raised by Foschi and Klainerman \cite[Conjecture 17.2
]{FoKl}, where the sharp non-endpoint 
case can be included into the elliptic type surfaces and has already been settled by Tao \cite{TaoGFA}.
Unlike the question on  paraboloids  in \cite{Y22}, one has to control the effects of certain perturbations on the phase functions.
\smallskip

To state our uniform bilinear estimates, we first introduce some notations.
 For each small number $\sigma>0$, let  $\mathscr{E}_{\sigma}$ be the   class of $\R-$valued functions  $\Phi_\sigma$ defined on  $\R^n$ satisfying the following conditions:
\begin{itemize}
	\item[-] $\Phi_\sigma\in C^\infty(\R^n\setminus\{0\})$  is homogeneous of order one,
	\item[-] $\Phi_\sigma (\mathbf{0},1)=0$, $\;\nabla_{\xi'}\Phi_\sigma(\mathbf{0},1)=\mathbf{0}\in\R^{n-1},\; \nabla_{\xi'}^2\Phi_\sigma(\mathbf{0},1)={\rm Id}_{\R^{n-1}}$,
	\item[-] rank $\nabla_{\xi'}\Phi_\sigma(\xi',\xi_n)=n-1$, for all $\xi\in\R^n\setminus\{0\}$
\end{itemize}
 with $\xi:=(\xi',\xi_n)\in\R^{n-1}\times\R$,
such that on  $\mathcal{ N}=\{\xi\in\R^n\,:\,|\xi'|\le3 |\xi_n|\}$  and for any multi-indices $\alpha=(\alpha_1,\ldots,\alpha_n)$
with  $ \alpha_j\in\Z\cap [0,\infty)$ $\forall j$, there exists a finite constant $C=C_{\alpha}>0$ depending only on $n$ and $\alpha$ for which we have
\begin{equation}
\label{eq:cond-error-ctrl}
\Bigl|\partial^\alpha_\xi\Bigl(\Phi_\sigma(\xi)-\frac{|\xi'|^2}{2\xi_n}\Bigr)\Bigr|\le C\, \sigma^{2}\,|\xi|^{1-|\alpha|},\;\forall\; \xi\in\mathcal{ N}\cap (\R^n\setminus \{0\}),
\end{equation}
where  $|\alpha|:=\alpha_1+\cdots+\alpha_n$.\smallskip

Let $\Omega \subset\mathcal{ N}$ be a compact set away from the origin and
define 
$$
S_{\Omega}^{\Phi_\sigma}:f\mapsto \int_\Omega e^{i(x\cdot\xi+t\Phi_\sigma(\xi))}\wh{f}(\xi)\,d\xi\,,
$$
initially on Schwartz functions $\mathcal{S}(\R^n)$.
Then, $S^{\Phi_\sigma}_{\R^n}=e^{it\Phi_\sigma(D)}$
with $D=-i\nabla_x$.
This corresponds to the adjoint Fourier restriction operator\footnote{The Stein operator.} on
 the conic surface $$\mathcal{ C}^{\Phi_\sigma}_{\Omega}:=\{(\Phi_\sigma(\xi),\xi):\xi\in\Omega\}\subset\R^{1+n},$$ a Monge patch given
 by the graph of the function $\Phi_\sigma$ over $\Omega$. 
 For brevity, we shall refer to the given $\Phi_\sigma$ as a \emph{Monge function}.

\begin{theorem}
	\label{pp:sigma-0}
	Let $n\ge 2$ and $\varSigma_1,\varSigma_2$ be defined as 
	$$
	\varSigma_j=\{\xi\in\R^n\,:\,1\le (-1)^{j-1} \xi_{n-1},\,
	\xi_n\le 2,\, |\xi''|\le 1\},\quad j=1,2,
	$$
	where $\xi=(\xi'',\xi_{n-1},\xi_n)\in\R^{n-2}\times\R\times\R$.
	Then, there exists  $\sigma_0>0$, such that for any $\eps>0$, and  $(q,r)$  satisfying  
	$$
	\frac{1}{q}<\min\Bigl(1,\frac{n+1}{4}\Bigr),\quad
	\frac{1}{q}=\frac{n+1}{2}\Bigl(1-\frac{1}{r}\Bigr)\,,
	$$
    with	$1\le q,r\le\infty$, 
	there exists $C=C_{\varSigma_1,\varSigma_2,q,r,\sigma_0,\eps}$ such that 
	\begin{equation}
	\label{eq:B-sig}
	\bigl\|\bigl(S^{\Phi_\sigma}_{\varSigma_1} f\bigr)\bigl(S^{\Phi_\sigma}_{\mu\varSigma_2} g\bigr)\bigr\|_{L^q_t(\R;L^r_x(\R^n))}
	\le C\, \mu^{\max\bigl(\frac{1}{q}-\frac{1}{2},0\bigr)+\eps}\,
	\|f\|_{L^2}\|g\|_{L^2},
	\end{equation}
	holds
	for all $\mu\ge 1$,  $f, g\in \mathcal{ S}(\R^n)$
	and all $\Phi_\sigma\in\mathscr{E}_\sigma$, $\forall\,\sigma\in(0,\sigma_0]$.
\end{theorem}

Theorem \ref{pp:sigma-0}  is a \emph{stability} result of the endline bilinear estimate for 
the oblique cone  $\mathcal{ C}_{\mathtt{obl}}:=\bigl\{(\frac{|\xi'|^2}{2\xi_n},\xi',\xi_n): 1\le \xi_n\le 2, |\xi'|\le 3\bigr\}$  under small perturbations as the conic type surfaces given by the graph of Monge functions in $\cup_{\sigma\in (0,\sigma_0]}\mathscr{E}_\sigma$ provided $\sigma_0$ is small enough. \smallskip

	The $\mu^\eps$-loss arises from the iterative argument for the essential concentration of the waves on conic regions. 
	It is responsible for the requirement 
	on   the  conditions \eqref{eq:LV-11-el}\eqref{eq:LV-12-el} 
    with strict inequalities in Theorem \ref{thm:-main}.   
    To push them  to equalities,
    one  has to 
    remove the $\mu^\eps$ loss first. 
   Once this having been done,  
    it is also necessary
    to establish certain vector-valued bilinear adjoint restriction estimates, as pointed out in \cite{TV-2,TaoMZ}, to replace the crude  argument in Section \ref{sec:2}.

	  The $\mu^\varepsilon$-loss is irrelevant to the condition $\beta_0\ge 0$, since \eqref{eq:LV-6-el}
	 will only be used in the low-frequency $(++)$ case where $\mu\sim 1$. 
	 This inequality is required to be strict  in the off-endline case \cite{LeeVargas} and it is allowed to be
	 taken as equality in the endpoint result  \cite{TaoMZ}. To get  the endline estimates \eqref{eq:NF} with all $\beta_0\ge 0$ for the $r>1$ case,
	 we shall combine an argument going back to Klainerman-Tataru \cite{KT}, also used in \cite{TV-2,LeeVargas}, with a Mihlin-H\"ormander multiplier estimate  to remove the   strict inequality  postulated in \eqref{eq:LV-6}.
	 The $L^\infty_tL^1_x-$case will be treated separately by  a classical $L^2\times L^2\to 
	L^1$ paraproduct estimate, easily deduced from the  Coifman-Meyer argument. Note that
	 this strict inequality condition on $\beta_0$ does not cause any
	 wastage for the sharp  results of \cite{LeeVargas} in the off-endline case. \medskip

Theorem \ref{pp:sigma-0} is
a \emph{uniform} endline bilinear restriction estimates 
for a family of surfaces of disjoint conic type 
satisfying the axioms of Tao and Vargas in \cite{TV-1,TV-2}, 
including high-low frequency interactions as  \cite{TaoMZ}. 
We imposed more restrictive conditions
here in order to simplify  the argument.
It is plausible that one may relax the $\mathscr{E}_\sigma$ conditions.
Notice that  even for the limiting case when $\sigma=0$,  \eqref{eq:B-sig} is highly nontrivial.
Indeed, \eqref{eq:B-sig} in symmetric norms  with $q=r=\frac{n+3}{n+1}$ on $\mathcal{ C}_{\mathtt{obl}}$ can be easily deduced  from Theorem 1.1 of \cite{TaoMZ} via  changing variables
$$
(x',x_n,t)\to \Bigl(x',\frac{x_n+t}{\sqrt{2}},\frac{x_n-t}{\sqrt{2}}\Bigr),
$$
with $
x=(x',x_n)\in\R^{n-1}\times\R$ and 
$
(\xi',\xi_n)\to\bigl(\xi',\sqrt{2}\,\psi(\xi',\xi_n)\bigr)
$
with
$$
\psi(\xi',\xi_n)=\frac{|\xi'|^2}{2\xi_n}-\xi_n,\quad \xi_n\ne 0.
$$
This  argument clearly  fails when $q\neq r$, 
and we have to give a direct proof of \eqref{eq:B-sig}  by adapting the induction method in \cite{TaoMZ}. 
\smallskip

The reason of  resorting 
to the  uniform bilinear estimates of Theorem \ref{pp:sigma-0} is to 
get over the obstruction in using the Minkowski-conformal invariance  of the cone under   $Conf^+(\R^{n+1})$ in mixed-norms.

This invariance employed in the symmetric-norms does not work in the mixed-norms when $q\ne r$, since  transferring 
between Euclidean and the null coordinates
is impossible, due to  intertwining of
the temporal and spatial variables 
simultaneously  via linear transformations.
To overcome this, Lee and Vargas \cite[$\S$4.2,  P.1313]{LeeVargas}
adopted a different way of rescaling,
leading naturally to  a family of operators, smoothly depending on the level of
angular separateness, and proved a uniform bilinear estimate  for these operators for all $(q,r)$ satisfying the cone condition \eqref{eq:LV-3}  with \emph{strict} inequalities.

Our proof is a spontaneity of this  approach  and 
Theorem \ref{pp:sigma-0} slightly
generalizes  the result for specific family of Fourier extention  operators in spirit of Lee-Vargas to Monge functions in the class $\mathscr{E}_\sigma$, not only with an enlarged class of surfaces but also including
the endline case  except the critical index $(q_c,r_c)$.

\subsection{Outline of the proof}
We briefly explain our proof of Theorem \ref{thm:-main}. 
By symmetry, we first reduce the null form estimates for solutions of wave equations to the $(++)$ and $(+-)$ cases associated to one-sided half waves in the 
canonical way as \cite{FoKl,TV-2,TaoMZ,LeeVargas,LeeRogersVargas}.

For the $L^\infty_tL^1_x-$case,  following  \cite{FoKl}, we  write  the products of the two waves into a bilinear operator and use  the translation-modulation invariance along with the Plancherel theorem  to reduce the estimate to a  bilinear 
$L^2(\R^n)\times L^2(\R^n)\to L^1(\R^n)$ multiplier estimate which is known from the Coifman-Meyer theory.

For each $(q,r)$ on the endline with $1<r<r_c$,
we shall adopt the standard Littlewood-Paley  decomposition, reducing  to the following three cases: 
$$\mathit{(a)}.\;\text{low-low } (++), \;\,\mathit{(b)}.\;\text{low-low } (+-), \; \, 
\mathit{(c)}.\;\text{high-low } (+\pm)  \,,$$
where the `low-low' is short for the low-low frequency interactions and the  `high-low' stands for the interactions of high-low frequencies.

 All of these cases can be handled under the conditions of Theorem \ref{thm:-main} by using the Main Proposition in Section \ref{sec:2} below. To handle Case $(a)$, we consider the interactions of the two waves creating relatively low and very low frenquencies. For the first one, we use the bilinear Whitney decomposition along the angular directions to deal with  the angular compressions where the transversality condition depends on the level of colinearity, to which we use a rescaling to the Main Proposition. For the second, we use  an argument of Klainerman-Tataru \cite{KT} to find  enough amount of transversality.   In both cases,  the condition $\beta_0\ge 0$ allows us to use a Mihlin-H\"ormander type multiplier estimate in the mixed norms \cite{AI}. To handle  Case $(b)$, besides of the bilinear Whitney in the angular directions,  we shall also make use of the decomposition along  null directions, due  to Lee and Vargas \cite{LeeVargas}. The  case $(c)$ includes the high-low frequency interactions, which can be handled in a unified way. We shall use the condition \eqref{eq:LV-11-el}  for $q\ge 2$ and \eqref{eq:LV-12-el} for $q<2$.
 To complete the proof, we  reduce the Main Proposition to Theorem \ref{pp:sigma-0} by using the non-isotropic rescaling of Lee-Vargas  \cite{LeeVargas}.
\medskip

The major part of the paper is devoted to Theorem \ref{pp:sigma-0}, for which we shall use the 
enhanced induction-on-scale method of Tao \cite{TaoMZ}. Instead of  dealing with a \emph{fixed} lightcone there, we get
a uniform estimate for a family of surfaces, fairly well-behaved thanks to the $\mathscr{E}_\sigma$ conditions for all $\sigma\in(0,\sigma_0]$, provided $\sigma_0$ fixed small.

Recall that  in \cite{TaoMZ}, there are two key ingredients
in proving the endpoint 
theorems for the lightcone.

The first one is the  propagation 
of waves  along null directions, or more precisely  Huygens' principle for wave equations, which allows to explore spatial localizations by means of an appropriate truncation operator which preserves the wave structure on the frequency space.
As in \cite{TaoMZ}, we shall use a microlocal version of this property.
Consequences of this property are two aspects.  One is the energy estimate on the conic regions of opposite colours, where the first use of  the transversality condition takes place. A crucial fact is the Kakeya null property, that the normal vector at any point of the cone are always in the lightray directions. In particular, the normal field  is a submanifold of the unit sphere $\mathbb{S}^n$ of codimension one.  This is used to tackle the case when the energy of the  waves simultaneously  concentrate in a small disk.  The other one is the essential concentration of waves in the physical spacetime along conic sets which allows to wrest  a universal constant strictly smaller than one.   

The second  ingredient is the wave table theory  based on a careful wave-packet decomposition and  is more technically involved. The use of wave tables draws on three novelties compared with  
 \cite{TaoGFA,Wolff99}.
The first one is  equipping an auxiliary small constant
 to the previous versions of \cite{TaoGFA} by an averaging argument, such that the small constant can be
flexibly chosen
to deal with the energy-concentrated case by using iteration.
The second novelty is that in the construction of wave tables, one takes in more geometric information of interactions between two waves of opposite colours in the physical spacetime, especially that the wave table for a wave is constructed with respect to the other wave of the opposite colour, carefully modifying the initial data by using a linear transformation,  depending on the opposite colour wave, of the plain wave-packets in the tube summations.  This along with separated supports of quilts allows 
to get a pigeonhole-free version of the bilinear $L^2-$Kakeya type estimate, removing the otherwise logarithmic loss, which is acceptable in the non-endpoint case  incorporated with the $\eps-$removal lemma, however would block the approach to the endpoint estimates.
The third one is concerning the products of the two waves, of
which one is localized to high frequency and the wave table for the low frequency wave can be localized in the physical spacetime to quilts  at a depth reciprocal to the high frequency,  consistent with the uncertainty principle.

Based mainly on the above two innovations, the induction on scale method of Wolff \cite{Wolff99} is  enhanced by Tao  \cite{TaoMZ} to get the endpoint bilinear restriction estimate and our task is to verify this strategy for a familiy of conic type surfaces. To this end, we rely on the exploitation of the gain
in non-concentration of energy analogous to the analysis 
for Proposition 3.6 in \cite{TaoMZ} in order to bring in a universal small constant strictly less than one. To get the 
uniform estimate for the family of conic surfaces in question, 
we find that in this step, which only plays an auxiliary role,  it suffices to apply the  argument as in \cite{Y22} for any fixed $\Phi_\sigma$.
\medskip

The paper is organized as follows. In Section \ref{sec:2}, we  prove Theorem \ref{thm:-main} by reducing it  to Theorem \ref{pp:sigma-0} via Littlewood-Paley theory involving spacetime multipliers in mixed-norms,
The proof of Theorem \ref{pp:sigma-0} will be carried out through Section 3 to Section 5.  In Section 6, we  invest a possible way of combining the method 
of descent with the argument for Theorem \ref{pp:sigma-0} to get  the endpoint case of the Foschi-Klainerman conjecture on the bilinear Fourier restriction estimate on the unit sphere.

\subsection*{Acknowledgements} 
The author is supported by NSFC grant No. 11901032 and Research fund program for young scholars of Beijing Institute of Technology.

\section{Proof of Theorem \ref{thm:-main}  via the Main Proposition }\label{sec:2}
In this section, 
we first reduce Theorem \ref{thm:-main} to 
the main proposition below
by using a hybrid  of the  dyadic decomposition in \cite{TV-2} and \cite{TaoMZ,LeeVargas}, which allows us to take the equality in \eqref{eq:LV-6-el}.
The main proposition is then proved
by using   a non-isotropic rescaling argument and Theorem \ref{pp:sigma-0}.

\subsection{Reduction to the main proposition}
We start with fixing some notations. For each $n\ge 2$, define the forward/backward cone as $\mathcal{C}^\pm=\{(\pm |\xi|,\xi)\,:\, \xi\in\R^n\}$.
Given $\Omega\subset\R^n$, we use
$\mathcal{C}^\pm(\Omega)=\{(\pm|\xi|,\xi):\xi\in\Omega\}$
to denote the \emph{lift} of $\Omega$ to $\mathcal{C}^\pm$.

A function $u:\R^{1+n}\to \mathbb{C}$ is said to be a $(+)$ \emph{wave} (resp. a $(-)$ \emph{wave}) if  $\widetilde{u}$ is an $L^2$-measure supported on $\mathcal{C}^+$ (resp. on  $\mathcal{C}^-$).

Let $u$ be a $(+)$ wave and $v$ be a $(-)$ wave.
The energy of $u$ and $v$ are defined as 
$$
\EEE(u)=\|u(t,\cdot)\|_2^2,\; \EEE(v)=\|v(t,\cdot)\|_2^2.
$$
By Plancherel's theorem, we see that
$\EEE(u)$, $\EEE(v)$ are independent of $t$.\smallskip

Fix  $\sigma\in (0,1)$ and write $\xi=(\xi'',\xi_{n-1},\xi_n)$ with $\xi''=(\xi_1,\ldots,\xi_{n-2})$. Define
$$
\Gamma^\pm_\sigma=\bigl\{(\tau,\xi'',\xi_{n-1},\xi_n)\,:\;\tau=\pm |\xi|, \,1\le \pm\xi_n\le 2,\; \sigma\le \xi_{n-1}<2\sigma,\,|\xi''|\le\sigma \bigr\}.
$$

In this subsection, we prove Theorem \ref{thm:-main}  
assuming the following result.

\begin{prop_sans_no}
	Let $n\ge 2$.
	For any $\eps>0$ and $1\le q,r\le\infty$ such that 
	$$
	\frac{1}{q}<\min\Bigl(1,\frac{n+1}{4}\Bigr),\quad
	\frac{1}{q}=\frac{n+1}{2}\Bigl(1-\frac{1}{r}\Bigr),
	$$
	there exists a finite constant $C=C_{\eps,q,r,n}$
	such that we have 
	\begin{equation}
	\label{eq:BL}
		\|uv\|_{L^q_t(\R;L^r_x(\R^n))}
	\le C \mu^{\max\bigl(\frac{1}{q}-\frac{1}{2},\,0\bigr)+\eps}\sigma^{-\frac{4}{q(n+1)}}\EEE(u)^{1/2} \EEE(v)^{1/2},
	\end{equation}
     for all $u$ and $v$ being either  $(+)$ and $(-)$ waves such that 
	\begin{equation}
	\label{eq:sup-1}
	\supp\;\widetilde{u}\subset \Gamma^+_\sigma,\;
	\supp\;\widetilde{v}\subset\,\mu\,\Gamma^-_\sigma,
	\end{equation}
	or being both $(+)$ waves such that
	\begin{equation}
	\label{eq:sup-2}
	\supp\;\widetilde{u}\subset \Gamma^+_\sigma,\;
	\supp\;\widetilde{v}\subset\,-\mu\,
	\Gamma^-_\sigma,
	\end{equation}
	for all dyadic number $\mu \ge 1$ and all $ \sigma\in (0,\,1)$.
\end{prop_sans_no}

\begin{remark}
	 When $q=\infty$ and $r=1$, we may take $\varepsilon=0$ in \eqref{eq:BL} by using the Cauchy-Schwarz inequality and energy estimates. In fact, the $\eps$-loss might be removed for $q\ge 2$ by using the same method for $(q,r)=(2,\frac{n+1}{n})$ and interpolate with $(q,r)=(\infty,1)$. We will not elaborate this.
\end{remark}

\begin{remark}
	\label{rk++}
	The method of the  proof for the case 
	\eqref{eq:sup-2} can be modified easily to yield the 
	same bilinear estimate \eqref{eq:BL} for $u$, $v$ being both $(+)$ waves with the  frequency supports condition that 
	$
	\supp\,\widetilde{u}\subset\Gamma^+_{\sigma}
	$
	and $\supp\,\widetilde{v}$ is contained in 
	$$
	\mu\bigl\{(\tau,\xi'',\xi_{n-1},\xi_n)\,:\;\tau= |\xi|, \,1\le -\xi_n\le 2,\; \sigma\le \xi_{n-1}<2\sigma,\,|\xi''|\le\sigma \bigr\}.
	$$
	In this case,  the factor  $\sigma^{-\frac{4}{q(n+1)}}$ on the right side of \eqref{eq:BL} may be  discarded.
\end{remark}

Assuming  the  main proposition,
we  prove  Theorem \ref{thm:-main}.
Let $\phi$ and $\psi$ be the solutions to  
\eqref{eq:wq} with initial data $\phi[0]$ and $\psi [0]$ respectively.
Denote 
\begin{equation}
\label{eq:phipsi}
\phi^\pm(t)=e^{\pm i tD_0}\phi_0^\pm,\;
\psi^\pm(t)=e^{\pm i tD_0}\psi_0^\pm,
\end{equation}
as the one-sided solutions (half-waves)
with 
$$
\phi_0^\pm=\frac{1}{2}\Bigl(\phi_0\pm (iD_0)^{-1}\phi_1\Bigr),\;\,
\psi_0^\pm=\frac{1}{2}\Bigl(\psi_0\pm (iD_0)^{-1}\psi_1\Bigr)\,.
$$
Then, we have
$\phi=\phi^++\phi^-$, $\psi=\psi^++\psi^-$ and
\begin{align*}
\bigl\|\phi[0]\bigr\|_{\HHH^{\alpha_1}}=&\,\sqrt{2}\,\Bigl(\| D_0^{\alpha_1}\phi_0^+\|_2^2+\| D_0^{\alpha_1}\phi_0^-\|_2^2\Bigr)^\frac{1}{2},\\
\bigl\|\psi[0]\bigr\|_{\HHH^{\alpha_2}}=&\,\sqrt{2}\,\Bigl(\| D_0^{\alpha_2}\psi_0^+\|_2^2+\| D_0^{\alpha_2}\psi_0^-\|_2^2\Bigr)^\frac{1}{2}.
\end{align*}
Writing $$\phi\psi=\phi^+\psi^++\phi^+\psi^-+\phi^-\psi^++\phi^{-}\psi^{-}$$
and  using the triangle inequality,
it is enough, after taking conjugation,
to prove the estimate \eqref{eq:NF}
with $\phi \psi$ being equal to $\phi^+\psi^+$ and $\phi^+\psi^-$,
to which we refer respectively as the $(++)$ and the $(+-)$ case.  
We are thus reduced to 
\begin{equation}
\label{eq:1st-redction}
\| D_0^{\beta_0} D_+^{\beta_+} D_-^{\beta_-}(\phi^+\psi^{\pm})\|_{L^q_t(\R;L^r_x(\R^n))}\le C
\| D_0^{\alpha_1}\phi_0^+\|_{L^2(\R^n)}\|D_0^{\alpha_2} \psi_0^{\pm}\|_{L^2(\R^n)},
\end{equation}
for  some constant $C=C_{\varTheta}$, which may depend on ${\varTheta}:=(q,r,\beta_0,\beta_+,\beta_-,\alpha_1,\alpha_2)$,  satisfying the conditions in Theorem \ref{thm:-main}.\smallskip

As \cite{TV-2}, it is convenient to recast \eqref{eq:1st-redction} 
into   $C^{\pm}_{\alpha_1,\alpha_2}(Q^1,Q^2,D)$ below.
\begin{definition}
	\label{def:2.2}
	Let $\alpha_1,\alpha_2\in\mathbb{R}$, 
	$Q^1,Q^2\subset \mathbb{R}^n$ and 
	$D$ be the Fourier multiplier on $\mathbb{R}^{1+n}$ 
	associated to a symbol $m(\tau,\xi)$. 
	For  $(q,r)\in[1,\infty]^2$ satisfying \eqref{eq:LV-3-el}, 
	define $C^\pm_{\alpha_1,\alpha_2}(Q^1,Q^2,D)$ 
	to be the best constant $C$ such that
	\begin{equation}
	\label{eq:defrr}
		\| D(\phi^+\psi^\pm)\|_{L^q_t(\R;L^r_x(\R^n))}\le C \|\phi_0^+\|_{\dot{H}^{\alpha_1}(\R^n)}
	\|\psi_0^\pm\|_{\dot{H}^{\alpha_2}(\R^n)}
	\end{equation}
	holds for all $\phi_0^+$,$\psi_0^\pm$ 
	whose Fourier transforms 
	are supported in $Q^1,Q^2$ respectively, where $\phi^+,\psi^\pm$ are given by \eqref{eq:phipsi}.
\end{definition}

The optimal constant $C^\pm_{\alpha_1,\alpha_2}$ 
depends  on the given exponents $(q,r)$. 
We suppress the  expression for this dependence  
in the notation for brevity.

With this terminology, 
\eqref{eq:1st-redction} can be written as 
\begin{equation}
\label{eq:2nd-reduction}
C^\pm_{\alpha_1,\alpha_2}(\R^n,\R^n, D_0^{\beta_0}D_+^{\beta_+}D_-^{\beta_-})
\lesssim_{\;{\varTheta}} 1\,.
\end{equation}

\begin{remark}
	As in \cite{TV-2}, 
	the bilinear restriction theorem on the cone \cite{Wolff99,TaoMZ,Temur} 
	can be reformulated as
	$C^\pm_{0,0}(Q^1,Q^2,1)\lesssim 1$ 
	for compact subsets $Q^1,Q^2\subset\R^n$ 
	being appropriately separated.
	Thus, it is reasonable to integrate the bilinear restriction theorem into the more general null form theory.
\end{remark}

\begin{remark}
	If we let  $^{\flat}{C}^\pm_{\;\alpha_1,\alpha_2}
	(\R^n,\R^n, D_0^{\beta_0}D_+^{\beta_+}D_-^{\beta_-})$ 
	have the same meaning
	 in
     Definition \ref{def:2.2} as the optimal constant of the estimate \eqref{eq:defrr}
    for $q=r=\frac{n+2}{n}$ with 
	$\alpha_1+\alpha_2=\frac{n}{n+2}$, 
	which satisfies  \eqref{eq:LV-3} in the one-dimensinally reduced endline condition,
    then
	the Machedon-Klainerman conjecture for the 
	Schr\"odinger equation in the $(n-1) $ dimension 
	is implied by the following null form estimate 
	\begin{equation}
	\label{eq:mfd}
	^\flat{C}^\pm_{\;\alpha_1,\alpha_2}(\R^n,\R^n, D_0^{\beta_0}D_+^{\beta_+}D_-^{\beta_-})\lesssim 1\,,
	\end{equation}
	for some suitable  $\beta_0,\beta_+,\beta_-$
	satisfying \eqref{eq:LV-2}. This 
    is called \emph{the method of descent}, introduced in   \cite[Proposition 17.5]{TaoMZ}.
\end{remark}
\begin{remark}
	The estimate \eqref{eq:mfd} remains open, while for the $\alpha_1+\alpha_2>\frac{n}{n+2}$ case, it is included  in a general result proved in \cite{LeeVargas}.
	Although a proof for   \eqref{eq:mfd} is not known,
	an adaptation of 
	the  method  in \cite{TaoMZ}
    yields an intermediate result weaker than \eqref{eq:mfd} but strong enough to imply  
	the endline estimates on paraboloids \cite{Y22}.
  For the endpoint estimate on the unit sphere, we need deal with perturbations on the phase functions as in Theorem \ref{pp:sigma-0}. 
\end{remark}

To show \eqref{eq:2nd-reduction}, we consider the case $q=\infty$ and $r=1$ first. Following \cite{FoKl}, let 
\begin{align*}
W_+(\eta,\zeta)=&\,
\frac{|\eta+\zeta|^{\beta_0}(|\eta|+|\zeta|)^{\beta_+}(|\eta|+|\zeta|-|\eta+\zeta|)^{\beta_-}}{|\eta|^{\alpha_1}|\zeta|^{\alpha_2}},\\
W_-(\eta,\zeta)=&\,
\frac{|\eta+\zeta|^{\beta_0+\beta_+}(|\eta+\zeta|-||\eta|-|\zeta||)^{\beta_-}}{|\eta|^{\alpha_1}|\zeta|^{\alpha_2}}\quad.
\end{align*}
Define the bilinear multiplier operators $B_{(++)}$ and $B_{(+-)}$ as
$$
B_{(+\pm)}(f,g)(x)=\iint_{\R^{2n}}e^{ix\cdot(\eta+\zeta)}W_\pm(\eta,\zeta) f(\eta )g(\zeta)\,d\eta d\zeta.
$$
By translation invariance and Plancherel's theorem, \eqref{eq:2nd-reduction} follows from 
\begin{equation}
\label{eq:bilinear}
\|B_{(+\pm)}\|_{L^2(\R^n)\times L^2(\R^n)\to L^1(\R^n)}\lesssim 1.
\end{equation}
Noting that $W_\pm$ is smooth on $\R^{2n}\setminus\{0\}$ and  homogeneous of order zero by \eqref{eq:LV-2-el},  the  classical  Coifman-Meyer method yields \eqref{eq:bilinear}. We omit the details.\medskip

Next, we consider the  $r>1$ case.
We shall write $\phi\psi$ 
standing for $ \phi^+\psi^+$ or $\phi^+\psi^-$, 
depending on being 
either in the $(++)$ or the $(+-)$ case.

For each $\lambda\in 2^{\Z}$, let $A_\lambda:=\{\xi\in\R^n:|\xi|\sim \lambda\}$. 
Using  Littlewood-Paley,
we write  
$$
\phi(t)=\sum_{\nu\in 2^\Z} \phi_\nu(t),\quad
\psi(t)=\sum_{\mu\in 2^\Z} \psi_\mu(t),
$$
where the partial Fourier transforms 
$\wh{\phi}_\nu(t)$, $\wh{\psi}_\mu(t)$  of $\phi_\nu(t),\psi_\mu(t)$ with respect to 
the spatial variables 
are supported respectively in $A_\nu, A_\mu$
for every $t$. 
For  \eqref{eq:2nd-reduction},
by using the Minkowski inequality, Schur's test and  symmetry, it suffices to show 
\begin{equation}
\label{eq:Rd-1}
C^\pm_{\alpha_1,\alpha_2}(A_\nu,A_\mu, D_0^{\beta_0}D_+^{\beta_+}D_-^{\beta_-})
\lesssim_{{\varTheta},\varepsilon} \Bigl(\frac{\mu}{\nu}\Bigr)^{-\varepsilon},
\quad\forall\,\mu\ge\nu\,,
\end{equation}
for some small $\eps>0$, 
with  ${\varTheta}$  in \eqref{eq:1st-redction}. 
Scaling by  \eqref{eq:LV-2-el},  we reduce \eqref{eq:Rd-1} to
\begin{equation}
\label{eq:Rd-2}
C^\pm_{\alpha_1,\alpha_2}(A_1, A_\mu, D_0^{\beta_0} D_+^{\beta_+}D_-^{\beta_-})\lesssim \mu^{-\varepsilon}\,,\quad \forall\,\mu\ge 1.
\end{equation}
Let $\mathbf{D}:=D_0^{\beta_0}D_+^{\beta_+}D_-^{\beta_-}$ 
and $\mathbf{m}=\mathbf{m}(\tau,\xi)$ be the symbol of the multiplier
$\mathbf{D}$, 
$$
\mathbf{m}(\tau,\xi)=|\xi|^{\beta_0}(|\tau|+|\xi|)^{\beta_+}
\bigl| |\tau|-|\xi|\bigr|^{\beta_-}.
$$
Then
$\widetilde{\mathbf{D}(\phi_1\,\psi_\mu)}(\tau,\xi)=\mathbf{m}(\tau,\xi)\;\bigl(\widetilde{\phi}_1*\widetilde{\psi}_\mu\bigr)(\tau,\xi)$  vanishes outside the set
\begin{equation}
\label{eq:set-m}
\supp (\mathbf{m})\cap\bigl\{(\,|\zeta|\pm |\eta|,\zeta+\eta\,):\zeta\in A_1,\eta \in A_\mu\bigr\}\,,
\end{equation}
depending on being in  the $(++)$ or the $(+-)$ case.
In the $(++)$ case, we may confine $\mathbf{m}$ 
inside the region $|\xi|\le \tau$ by triangle inequality, whereas in the $(+-)$ case, 
we may impose the condition $|\tau|\le |\xi|$.
\smallskip

Now, we introduce the $P_\pm$ operator as in \cite{TV-2} 
corresponding to the $(++)$ and $(+-)$ cases respectively. 

In the $(++)$ case, let $P_+$ be  a smooth multiplier with symbol  supported on $|\xi|\ll \tau$ and  write
$$
\mathbf{D}(\phi_1\psi_\mu)=
P_+ \,\mathbf{D}(\phi_1^+\psi_\mu^+)+(I-P_+)\, \mathbf{D}(\phi_1^+\psi_\mu^+),$$
where the first term on the right side vanishes unless $\mu\sim 1$. It is easy to see that  as long as the second term does not vanishes identically,  $(I-P_+)$ is frequency-localized  to the region $|\xi|\sim\tau$. 

In the $(+-)$ case, we have 
$|\tau|+|\xi|\sim |\xi|\lesssim \mu$. 
Letting $P_-$ be given by a smooth multiplier with frequency localized on $|\tau|+|\xi|\ll \mu$, we may  write
$$
\mathbf{D}(\phi_1\psi_\mu)=
P_- \,\mathbf{D}(\phi_1^+\psi_\mu^-)+(I-P_-)\, \mathbf{D}(\phi_1^+\psi_\mu^-),$$
where the first term on the right side vanishes unless $\mu\sim 1$, whereas the second one is truncated in the frequency space to the region $|\tau|+|\xi|\sim \mu$.

By  triangle inequalities, 
 \eqref{eq:Rd-2} is reduced  to 
\begin{align}
\label{eq:+++}
&C^+_{\alpha_1,\alpha_2}(A_1, A_\mu, P_+\,\mathbf{D})
\lesssim \mu^{-\varepsilon}\,,\\
\label{eq:---}
&C^-_{\alpha_1,\alpha_2}(A_1, A_\mu, P_-\,\mathbf{D})
\lesssim \mu^{-\varepsilon}\,,\\
\label{eq:-+-}
&C^\pm_{\alpha_1,\alpha_2}(A_1, A_\mu, (I-P_\pm)\,\mathbf{D})
\lesssim \mu^{-\varepsilon}\,,
\end{align}
uniformly for all $\,\mu\ge 1$.
The $(I-P_\pm)$ part corresponds to
 the high-low frequency interactions, which 
 can be handled in a unified way. \medskip

We prove 
\eqref{eq:+++}\eqref{eq:---} and \eqref{eq:-+-} using  the main proposition.\medskip

\noindent{\bf $\bullet$ Proof of \eqref{eq:+++}:} We have $\mu\sim 1$ otherwise $P_+\mathbf{D}=0$.  This corresponds to the low frequency interactions in the $(++)$ case. 
\medskip

Using $1<q,r<\infty$ and the  Mihlin-H\"ormander 
theorem for mixed-norms, the operator $D_0^{\beta_0} D_+^{-\beta_0}$
is bounded on $L^q_t(\R,L^r_x(\R^n))$, see for example \cite[Theorem 7]{AI}. Here, we have used  $\beta_0
\ge 0$ and that the symbol of the multiplier $D_0^{\beta_0} D_+^{-\beta_0}$ is smooth on $\R^{1+n}_{\tau,\xi}\setminus\{0\}$ and is  homogeneous of order zero. 
Therefore,  it suffices to show  \eqref{eq:+++} with $\beta_0=0$.\medskip

Without loss of generality, we may take $\mu=1$  after suitably modifying the structure constant if necessary as in \cite{TV-2}.

Fix a small constant $\gamma_0\in 2^\Z$ with $0<\gamma_0\ll 1$
and decompose
$$\{(\tau,\xi):|\xi|\lesssim 1\}=\bigcup_{\substack{\gamma_0\le\gamma\le 1\\ \gamma\;\text{dyadic}}}\Omega_\gamma$$ 
where $\Omega_\gamma:=\{(\tau,\xi):|\xi|\sim\gamma\}$ for each $\gamma>\gamma_0$ and $\Omega_{\gamma_0}=\{(\tau,\xi):|\xi|\le \gamma_0\}$.

Accordingly, decompose $P_+ $ by the  inhomogeneous Littlewood-Paley partition of the identity operator
$$
P_+=\sum_{\gamma_0\le \gamma\,\le\, 1} \Delta_\gamma,
$$
where $\Delta_\gamma=P_+\Delta_\gamma$ is the multiplier adapted to $\Omega_\gamma$ and $\Delta_\gamma\sim 1$ if $\gamma>\gamma_0$.
Using  Minkowski's inequality, we may reduce \eqref{eq:+++} (with $\mu=1$) to showing that 
\begin{equation}
\label{eq:+++2}
C^+_{\alpha_1,\alpha_2}(A_1,A_1,\Delta_\gamma D_+^{\beta_+} D_-^{\beta_-})\lesssim 1\,,\forall\;  \gamma\in [\gamma_0,1]\,.
\end{equation}

For each $\gamma>\gamma_0$ and each dyadic $\sigma$ with $0<\sigma\le\,1$, we decompose $A_1$
into finitely overlapping projective sectors $\varGamma$ with angle $\sigma$. By using the bilinear Whitney type decomposition (c.f. \cite{TVV,TV-1,LeeVargas}) and Schur's test, we may reduce \eqref{eq:+++2} to showing that there exist some small $\varepsilon>0$ such that
\begin{equation}
\label{eq:+++4}
C^+_{\alpha_1,\alpha_2}(\varGamma,\varGamma',
\Delta_\gamma D_+^{\beta_+}D_-^{\beta_-})\lesssim \sigma^{\varepsilon},
\end{equation}
holds uniformly for all $\varGamma,\varGamma'$ with angle $\sigma$ such that 
$\sphericalangle (\varGamma
,\varGamma')\sim \sigma$. Here, $\sphericalangle (\varGamma
,\varGamma')$ denotes the angular distance between $\varGamma$ and $\varGamma'$.

Write $\varGamma=-\Gamma\cup \Gamma$ with $\Gamma$, $-\Gamma$ being the antipode
components of $\varGamma$ and likewise for $\varGamma'=-\Gamma'\cup\Gamma'$, where the angle between $\Gamma,\Gamma'$ belongs to $(0,\frac{\pi}{2})$.
 
By triangle's inequality and symmetry,  \eqref{eq:+++4} 
follows from
\begin{equation}
\label{eq:+++4-gd}
C^+_{\alpha_1,\alpha_2}(\Gamma,\pm\Gamma',
\Delta_\gamma D_+^{\beta_+}D_-^{\beta_-})\lesssim \sigma^{\varepsilon},
\end{equation}

To prove \eqref{eq:+++4-gd}, we first consider  the 
 cases  $\gamma>\gamma_0$.
Discarding the harmless operators $\Delta_\gamma D_+^{\beta_+-\beta_-}\sim 1$, we are reduced to 
\begin{equation}
\label{eq:+++5}
C^+_{\alpha_1,\alpha_2}(\Gamma,\pm\Gamma', |\Box|^{\beta_-})\lesssim \sigma^{\varepsilon}
\end{equation}
where $|\Box|:=D_+D_-$.

To proceed, we recall the following technical lemma on the multiplier estimates for $|\Box|$
 taken from \cite{LeeVargas}.

\begin{lemma}
	\label{lem:LV-2.6}
	Let $\beta_-\in\C$. Suppose either that $u$ and $v$ are $(+)$ and $(-)$ waves with \eqref{eq:sup-1} or that both $u$ and $v$ are $(+)$ waves such that \eqref{eq:sup-2} holds.
	Then, for all $1\le q,r\le\infty$, $\mu\ge 1$ and $\sigma\in (0,1)$, we have for any $N$, there is a finite constant $C_N$ independent of $\beta_-$ such that
	\begin{multline*}
	\quad	\||\Box|^{\beta_-}(uv)\|_{L^q_tL^r_x}\le C_N (1+|\beta_-|)^N(\mu\sigma^2)^{\mathbf{Re}(\beta_-)}\\
	\times	\sum_{(\Bbbk_1,\Bbbk_2)\in \Z^n\times\Z^n}(1+|\Bbbk_1|+|\Bbbk_2|)^{-N}\|u_{\Bbbk_1}v_{\Bbbk_2}\|_{L^q_tL^r_x},\quad
	\end{multline*}
	where for each $(\Bbbk_1,\Bbbk_2)\in\Z^n\times\Z^n$, $u_{\Bbbk_1}$ and $v_{\Bbbk_2}$ are $(+)$ and $(+)$/$(-)$ waves, having the same support property
	in the frequency space with $u,v$ and satisfying that
	$$
	\EEE(u_{\Bbbk_1})=\EEE(u),\quad \EEE(v_{\Bbbk_2})=\EEE(v).
	$$
\end{lemma}

By using Lemma \ref{lem:LV-2.6} and \eqref{eq:LV-4-el}, we may reduce \eqref{eq:+++5} with $\eps$ small enough to
\begin{equation}
\label{eq:+++6}
C^+_{\alpha_1,\alpha_2}(\Gamma,\pm\Gamma',1)\lesssim \sigma^{-\frac{4}{q(n+1)}},
\end{equation}
which, 
after a suitable spatial rotation,
follows from \eqref{eq:BL} of the Main Proposition
and Remark \ref{rk++}.
\smallskip

It remains to consider  the $\Delta_{\gamma_0}$ term.
In this case, instead of using Lemma \ref{lem:LV-2.6} and the bilinear Whitney, we shall use an argument as  in   \cite{KT,TV-2}.
Divide $A_1$ into essentially disjoint cubes $Q$
of sidelength $\gamma_0$. Let $m_{\gamma_0}$ be the symbol of $\Delta_{\gamma_0}$ and denote $B(0,R)$ as the ball of radius $R>0$ centered at the origin of $\R^n$. We write $Q\sim  Q'$ if there exists $\zeta\in Q, \eta\in Q'$ such that $m_{\gamma_0}(|\zeta|+|\eta|,\zeta+\eta)\ne0$. Then, there is a universal constant $C$ such that $Q\sim Q'$ only when $Q+Q'\subset B(0,C\gamma_0)$. By taking $\gamma_0$ smaller if necessary, $Q\sim Q'$ only when $Q$ and $Q'$ are almost opposite. 
Consequently,  we have
$$
\sup_{Q}\,{\rm Card}\{Q':Q\sim Q'\}+\sup_{Q'}\,{\rm Card}\{Q:Q\sim Q'\}\lesssim \,1.
$$
By using Schur's test, we may reduce \eqref{eq:+++2} to 
\begin{equation}
\label{eq:+++3}
C^+_{\alpha_1,\alpha_2}(Q,Q',\Delta_{\gamma_0}D_+^{\beta_+}D_-^{\beta_-})\lesssim \,1\,,
\end{equation}
uniformly for $\gamma_0$-cubes $Q\sim Q'$.

Noting that 
$D_+\sim D_-\sim 1$, we have by Bernstein's inequality
\begin{align*}
	C^+_{\alpha_1,\alpha_2}(Q,\,Q',\,\Delta_{\gamma_0}
	D_+^{\beta_+}D_-^{\beta_-})
	\lesssim C^+_{\alpha_1,\alpha_2} (Q,\, Q',\, 1).
\end{align*}
Thus, it suffices to show
\begin{equation}
\label{eq:+++3-gf}
C^+_{\alpha_1,\alpha_2}(Q,\,Q',\,1)\lesssim \,1,
\end{equation}
uniformly for $\gamma_0$-cubes $Q\sim Q'$.

By using rotation in the spatial variable, we may place the sets $Q$ and $Q'$ as  in Remark \ref{rk++} so that \eqref{eq:+++3-gf} follows.
 The proof of \eqref{eq:+++} is thus complete.
\smallskip

\noindent{\bf $\bullet$ Proof of \eqref{eq:---}:} We are in the low-low frequency interaction case with $\mu\sim1$.
Moreover, we have $D_0\sim D_+$. Hence by
the standard multiplier estimate (see \cite{AI}), we may reduce \eqref{eq:---} to
\begin{equation}
\label{eq:---1}
C^-_{\alpha_1,\alpha_2}(A_1,A_1,P_-D_+^{\beta_0+\beta_+}D_-^{\beta_-})\lesssim 1\,.
\end{equation}
%where we have taken $\mu=1$ without loss of generality.

Decompose
$\{(\tau,\xi):|\tau|+|\xi|\lesssim 1\}=\cup_{\gamma>0}\;\Omega_\gamma^+$, where for each  $\gamma\in(0,1]\cap 2^\Z$, 
$$
\Omega_\gamma^+:=\{(\tau,\xi):|\tau|+|\xi|\sim \gamma\}.
$$
For each $\gamma$, partition
$
\Omega_\gamma^+=\cup_{0<\lambda\le \gamma}\; \Omega^{-}_{\lambda,\gamma}
$ with $\lambda$ being dyadic, and 
$$
\Omega_{\lambda,\gamma}^-=\{(\tau,\xi)\in\Omega_\gamma^+: \bigl||\tau|-|\xi|\bigr|\sim \lambda\}.
$$
Let $\Delta_\gamma^+$ and $\Delta_{\lambda,\gamma}^-$
be the smooth multipliers adapted to $\Omega_\gamma^+$
and $\Omega_{\lambda,\gamma}^-$ respectively such that we may write
$$
P_-D_+^{\beta_0+\beta_+} D_-^{\beta_-}=\sum_{\substack{0<\gamma\,\le\, 1\\ \gamma\;\text{dyadic}}}\;
\sum_{\substack{0<\lambda\,\le\gamma\\ \lambda\;\text{dyadic}}}\gamma^{\beta_0+\beta_+}
\lambda^{\beta_-}\;\Delta_\gamma^+\Delta_{\lambda,\gamma}^-\,.
$$
By  Minkowski's inequality, \eqref{eq:---1} is reduced to showing for some $\varepsilon>0$ 
\begin{equation}
\label{eq:---2}
C_{\alpha_1,\alpha_2}^-(A_1, A_1, \Delta_\gamma^+\Delta_{\lambda,\gamma}^{-})
\lesssim \gamma^{-(\beta_0+\beta_++\beta_-)+\varepsilon}
\Bigl(\frac{\lambda}{\gamma}\Bigr)^{-\beta_-+\varepsilon}\,,
\end{equation}
holds uniformly for all $\lambda\le \gamma\le 1$.

Let $m^{+-}_{\lambda,\gamma}(\tau,\xi)$ be the symbol of 
$\Delta_\gamma^+\Delta_{\lambda,\gamma}^-$ such that $m^{+-}_{\lambda,\gamma}$ is a smooth bump function adapted to $\{(\tau,\xi):|\xi|+|\tau|\sim \gamma, ||\tau|-|\xi||\sim\lambda\}$. Partition $A_1$ into essentially disjoint sectors $\varGamma$ of length $\sim 1$ in the radial direction and of angular width $\sim\sqrt{\gamma\lambda}$. 
 Write $\varGamma'\sim\varGamma''$ if there exist $\zeta\in\varGamma'$, $\eta\in\varGamma''$ such that $m^{+-}_{\lambda,\gamma}(|\zeta|-|\eta|,\zeta+\eta)\ne 0$.

 Since the $\gamma\sim 1$ case is easier and can be handled by slightly adjusting the  argument, we only deal with the  $\gamma\ll 1$ case. 

Using the  elementary relation
$$
|\zeta+\eta|-||\zeta|-|\eta||\sim \frac{|\zeta|\cdot|\eta|}{|\zeta+\eta|}\;\theta(\zeta,-\eta)^2,
$$
where $\theta(\zeta,-\eta)$ denotes the angle formed between $\zeta$ and $-\eta$ (c.f. \cite{TV-2}), we have $\varGamma'\sim\varGamma''$ only when 
$
\sphericalangle(\varGamma',-\varGamma'')\lesssim \sqrt{\gamma\lambda}
$. By Schur's test,  \eqref{eq:---2} reduces to 
\begin{equation}
\label{eq:---3}
C^-_{\alpha_1,\alpha_2}(\varGamma',\varGamma'',\Delta_\gamma^+\Delta_{\lambda,\gamma}^-)\lesssim
\gamma^{-(\beta_0+\beta_++\beta_-)+\varepsilon}
\Bigl(\frac{\lambda}{\gamma}\Bigr)^{-\beta_-+\varepsilon}\,,
\end{equation}
uniformly for all $\varGamma'\sim\varGamma''$.\smallskip

For each dyadic $\sigma$ with $0<\sigma\lesssim\sqrt{\gamma\lambda}$, decompose $A_1$ as before
into finitely overlapping projective sectors $\{\Gamma\cup (-\Gamma)\}$ with angle $\sigma$.  Using the bilinear Whitney type decomposition and Schur's test, it suffices to show
\begin{equation}
\label{eq:---4}
C^-_{\alpha_1,\alpha_2}(\varGamma'\cap \Gamma',\varGamma''\cap(-\Gamma''),\Delta_\gamma^+\Delta_{\lambda,\gamma}^-)\lesssim
\sigma^{\varepsilon}
\gamma^{-(\beta_0+\beta_++\beta_-)+\varepsilon}
\Bigl(\frac{\lambda}{\gamma}\Bigr)^{-\beta_-+\varepsilon}\,,
\end{equation}
for $\sphericalangle(\Gamma',\Gamma'')\sim \sigma$.
To this end, we  decompose  $\Gamma'\times(-\Gamma'')$ as follows.\smallskip

For each $j\in\mathcal{J}_\sigma:=\{0,1,\ldots,\log_2\sigma^{-1}\}$ with $\sigma\in (0,1)\cap 2^\Z$, we write $\Gamma'=\cup_{\Lambda'_j}\Lambda'_j$, where each $\Lambda'_j$ is a segment on $\Gamma'$ of dimensions $\underbrace{\sigma\times\cdots\times\sigma}_{n-1\text{  times}}\times2^{j}\sigma$. For each $\Lambda'_j$, set
$$
\Lambda_j''=\{\xi\in(-\Gamma'')\;:\;\mathsf{dist}\bigl(\xi,-\Lambda'_j\bigr)\in [2^{j-1}\sigma,2^j\sigma)\},\quad \forall \;j\ge1,
$$ 
and $$\Lambda''_0=\{\xi\in(-\Gamma'')\;:\;\mathsf{dist}\bigl(\xi,-\Lambda_0\bigr)\le\sigma\},$$
so that we have 
$$
\Gamma'\times(-\Gamma'')=\bigcup_{j=0}^{\log_2\frac{1}{\sigma}}\bigcup_{\Lambda_j',\Lambda_j''}\Lambda'_j\times\Lambda_j''\,.
$$
By  Schur's test, \eqref{eq:---4}  reduces to 
\begin{equation}
\label{eq:---5}
C^-_{\alpha_1,\alpha_2}(\Lambda_j',\Lambda_j'',\Delta_\gamma^+\Delta_{\lambda,\gamma}^-)\lesssim
\sigma^{\varepsilon}
\gamma^{-(\beta_0+\beta_++\beta_-)+\varepsilon}
\Bigl(\frac{\lambda}{\gamma}\Bigr)^{-\beta_-+\varepsilon}\,,\quad \forall\;j\in\mathcal{J}_\sigma\,,
\end{equation}
where the logarithmic loss $\log\frac{1}{\sigma}$ is safely absorbed to $\sigma^\varepsilon$ by slightly adjusting  $\varepsilon$.

By using  $\Delta_\gamma^+\sim \gamma^{-(\beta_0+\beta_+)}D_+^{\beta_0+\beta_+}$
and $\Delta_{\lambda,\gamma}\sim \lambda^{-\beta_-}D_-^{\beta_-}$,  \eqref{eq:---5}  reduces to
\begin{equation}
\label{eq:---6}
C^-_{\alpha_1,\alpha_2}(\Lambda_j',\Lambda_j'',
D_+^{\beta_0+\beta_+}D_-^{\beta_-})\lesssim
\sigma^{3\varepsilon}
\,,\quad \forall\;j\in\mathcal{J}_\sigma\,,
\end{equation}
where we have used $\sigma\lesssim \sqrt{\gamma\lambda}$ and $\gamma\le 1$.

In view of the definition of $\Lambda_j',\Lambda_j''$ for $j\ge 1$ and the angular separation for the $j=0$ case, it is readily deduced that $ D_+\sim 2^j\sigma$ on $\mathcal{C}^+(\Lambda_j')+\mathcal{C}^-(\Lambda_j'')$. Therefore, by Bernstein's estimates, Lemma \ref{lem:LV-2.6} and \eqref{eq:BL}, we have
\begin{equation}
\label{eq:---7}
C^-_{\alpha_1,\alpha_2}(\Lambda_j',\Lambda_j'',D_+^{\beta_0+\beta_+}D_-^{\beta_-})\lesssim
2^{j(\beta_0+\beta_+-\beta_-)}\sigma^{\beta_0+\beta_++\beta_--\frac{4}{q(n+1)}}.
\end{equation}

If $\beta_0+\beta_+-\beta_-\ge0$,  by using  $2^j\sigma\lesssim 1$, the right side of \eqref{eq:---7}
is bounded by
\begin{align*} (2^j\sigma)^{(\beta_0+\beta_+-\beta_-)}\,\sigma^{2\beta_--\frac{4}{q(n+1)}}\lesssim \sigma^{3\varepsilon}
\end{align*}
for $\eps$ sufficiently small,
where we have used \eqref{eq:LV-4-el} with strict inequality.

If $\beta_0+\beta_+-\beta_-<0$,
discarding $2^{j(\beta_0+\beta_+-\beta_-)}\lesssim1$, and  using \eqref{eq:LV-2-el} \eqref{eq:LV-8-el} with strict inequality, the right side of \eqref{eq:---7}
is bounded by 
\begin{align*}
\,\sigma^{\alpha_1+\alpha_2-\frac{1}{q}\bigl(\frac{n-1}{n+1}+\frac{4}{n+1}\bigr)}\lesssim
\sigma^{\alpha_1+\alpha_2-\frac{1}{q}\bigl(\frac{n+3}{n+1}\bigr)}
\lesssim \sigma^{3\varepsilon}.
\end{align*}
Combining these two cases, we obtain \eqref{eq:---6}
and complete the proof of \eqref{eq:---}.
\medskip

\noindent{\bf $\bullet$ Proof of \eqref{eq:-+-}:}
We shall use \eqref{eq:LV-11-el} when $q\ge 2$ and use 
\eqref{eq:LV-12-el} when $q_c<q<2$. We may assume $\mu\gg 1$, whilst the $\mu\sim 1$ case  can be handled using the same method as for \eqref{eq:+++} and \eqref{eq:---}.

Using $\mu\gg 1$ and $(I-P_\pm)$, we have  $|\xi|\sim |\tau|+|\xi|\sim \mu$ for both the $(++)$ and the $(+-)$ cases. 
In fact, in the $(++)$ case, this follows from $|\xi|\sim\tau\sim\mu$, while in the $(+-)$ case, it follows from $|\tau|\le |\xi|$ and $|\tau|+|\xi|\sim \mu$.

By using the Bernstein estimate,
we may reduce \eqref{eq:-+-} to
\begin{equation}
\label{eq:-+-1}
C^\pm_{\alpha_1,\alpha_2}(A_1,A_\mu,(I-P_\pm)D_-^{\beta_-})\lesssim \mu^{-(\beta_0+\beta_+)-\varepsilon}.
\end{equation}
Noting that 
$$
\bigl|||\zeta|\pm|\eta||-|\zeta+\eta|\bigr|\lesssim \frac{|\zeta|\cdot|\eta|}{|\zeta|+|\eta|}\lesssim 1,
$$
for all $(\zeta,\eta)\in A_1\times A_\mu$,
we may decompose as before
$$
D_-^{\beta_-}=\sum_{\substack{\gamma\lesssim 1\\ \gamma\;\text{dyadic}}}\gamma^{\beta_-}\Delta_\gamma^-.
$$
By using the standard multiplier estimate, we are reduced to 
\begin{equation}
\label{eq:-+-2}
C^\pm_{\alpha_1,\alpha_2}(A_1,A_\mu,\Delta_\gamma^-)\lesssim \gamma^{-\beta_-+\varepsilon}\mu^{-(\beta_0+\beta_+)-\varepsilon}\,,
\end{equation}
for some small $\varepsilon>0$.

Partition $A_1=\cup_\varGamma \varGamma$, $A_\mu=\cup_{\varGamma_\mu}\varGamma_\mu$
where $\{\varGamma\}, \{\varGamma_\mu\}$ are collections of disjoint sectors of angular width $\sim \gamma^{\frac{1}{2}}$ and of radial width  $\sim 1$ and $\sim\mu$ respectively. Let $m_\gamma^-$ be the symbol of $\Delta_\gamma^-$. We write $\varGamma\sim\varGamma_\mu$ if there exists $\zeta\in\varGamma$ and $\eta\in\varGamma_\mu$ such that $m^-_\gamma(|\zeta|\pm|\eta|,\zeta+\eta)\ne 0$. Then, it is clear that $\varGamma\sim\varGamma_\mu$
only when $\sphericalangle(\varGamma,\pm\varGamma_\mu)\lesssim \gamma^{1/2}$, and
by using Schur's test once more,  we are reduced  to 
\begin{equation}
\label{eq:-+-3}
C^\pm_{\alpha_1,\alpha_2}(\varGamma,\varGamma_\mu,\Delta_\gamma^-)\lesssim \gamma^{-\beta_-+\varepsilon}\mu^{-(\beta_0+\beta_+)-\varepsilon},
\end{equation}
uniformly for all $\varGamma\sim\varGamma_\mu$.

Given $\varGamma,\,\varGamma_\mu$ such that $\varGamma\sim\varGamma_\mu$,
we may assume  that the operator $\Delta_\gamma^-$ in \eqref{eq:-+-3} is localized  on a neighborhood of the set $\{(|\zeta|\pm|\eta|,\zeta+\eta):(\zeta,\eta)\in\varGamma\times\varGamma_\mu\}$.

We shall  sketch the proof of \eqref{eq:-+-3} below. Let $\phi$ and $\psi_\mu$ be such that the spacetime Fourier transforms $\widetilde{\phi}$ and $\widetilde{\psi}_\mu$ are supported on $\mathcal{C}^+(\varGamma)$ and $\mathcal{C}^\pm(\varGamma_\mu)$ respectively.

For each $\mu\ge 1$, applying the standard bilinear Whitney type decomposition in the circular directions for $\phi\psi_\mu$ as above (see also \cite{TV-1,TaoMZ,LeeVargas}), we may write
$$
\phi\psi_\mu=\sum_{\substack{0<\sigma\lesssim\gamma^{1/2}\\ \sigma\;\text{dyadic}}}
\;\, \sum_{\Gamma,\Gamma':\sphericalangle(\Gamma,\pm\Gamma')\sim\sigma}\phi_{\Gamma}\,\psi_{\mu,\Gamma'}\,,
$$
where $\Gamma,\Gamma'$ are sectors of the forward cone or the forward and backward cones,  depending on being either in the $(++)$ or $(+-)$ case,
with  $\mathsf{Ang}(\Gamma)$, $\mathsf{Ang}(\Gamma')\sim \sigma\,,$
where $\mathsf{Ang}$ stands for the aperture angle of  $\Gamma,\Gamma'$. 
We slightly abused notations here by letting $\Gamma$ be on the cone rather than that 
in the proof for the other two cases  \eqref{eq:+++}\eqref{eq:---}, where instead of $\Gamma\subset A_1$, its lifts $\mathcal{C}^\pm(\Gamma)$ are subsets of the cone.

By using Minkowski's inequality, Cauchy-Schwarz and Plancherel's theorem, we may reduce  \eqref{eq:-+-3} further to 
\begin{equation}
\label{eq:rd-3}
\| D_0^{\beta_0} D_+^{\beta_+} D_-^{\beta_-} (\phi_\Gamma \psi_{\mu,\Gamma'})\|_{L^q_tL^r_x}
\le C\,\Bigl(\frac{\sigma}{\mu}\Bigr)^{2\eps}
\mu^{\alpha_2}\,
\EEE(\phi_{\Gamma})^{\frac{1}{2}}
\EEE(\psi_{\mu,\Gamma'})^{\frac{1}{2}},
\end{equation}
where we have used $D_0\sim D_+\sim\mu$ and $\Delta_\gamma^-\sim \gamma^{-\beta_-}D_-^{\beta_-}$.

Exploring the spatial rotation for $\phi_\Gamma$ and $\psi_{\mu,\Gamma'}$, we may assume that 
\begin{align*}
(++)\quad & (\,\Gamma,\,\Gamma'\,)=(\,\Gamma^+_\sigma,-\mu\Gamma^-_{\sigma}),\\
(+-)\quad &
(\,\Gamma,\,\Gamma'\,)=(\,\Gamma^+_\sigma,\,\;\mu\,\Gamma^-_{\sigma}\,)\,.
\end{align*}

Note that  $\widetilde{\phi_{\Gamma}}*\widetilde{\psi_{\mu,\Gamma'}}$
is supported in the region where we have $D_0\sim D_+\sim\mu$. To show \eqref{eq:rd-3}, it suffices to show
\begin{equation}
\label{eq:h-l-1}
\bigl\||\Box|^{\beta_-}(\phi_{\Gamma}\psi_{\mu,\Gamma'})\bigr\|_{L^q_tL^r_x}
\le C\Bigl(\frac{\sigma}{\mu}\Bigr)^{2\eps}
\mu^{-(\beta_0+\beta_+-\beta_-)+\alpha_2}
\EEE(\phi_{\Gamma})^\frac{1}{2}\EEE(\psi_{\mu,\Gamma'})^\frac{1}{2}
\end{equation}
with $(\Gamma,\Gamma')$ being equal  either to $(\Gamma^+_\sigma,-\mu\Gamma^-_\sigma)$ or to $(\Gamma^+_\sigma,\mu\Gamma^-_\sigma)$. 

In view of \eqref{eq:LV-2-el} and Lemma \ref{lem:LV-2.6}, we may reduce \eqref{eq:h-l-1}  to 
\begin{equation}
\label{eq:h-l-2}
\bigl\|(\phi_{\Gamma}\psi_{\mu,\Gamma'})
\bigr\|_{L^q_tL^r_x}
\le C\sigma^{-2\beta_-+2\eps}\mu^{\beta_--\alpha_1+\frac{n-1}{q(n+1)}-2\eps}
\EEE(\phi_{\Gamma})^\frac{1}{2}\EEE(\psi_{\mu,\Gamma'})^\frac{1}{2}.
\end{equation}

 By using the conditon  \eqref{eq:LV-4-el} with strict inequality and taking $\eps$ small enough depending on $\beta_-$ and $q$, we may reduce \eqref{eq:h-l-2} to 
\begin{equation}
\label{eq:h-l-3}
\bigl\|(\phi_{\Gamma}\psi_{\mu,\Gamma'})\bigr\|_{L^q_tL^r_x}
\le C\sigma^{-\frac{4}{q(n+1)}}
\mu^{\beta_--\alpha_1
	+\frac{n-1}{q(n+1)}-2\eps}
\EEE(\phi_{\Gamma})^\frac{1}{2}\EEE(\psi_{\mu,\Gamma'})^\frac{1}{2},
\end{equation}
with $(\Gamma,\Gamma')$ being equal to either $(\Gamma^+_\sigma,-\mu\Gamma^-_\sigma)$ or $(\Gamma^+_\sigma,\mu\Gamma^-_\sigma)$. \smallskip

 To show \eqref{eq:h-l-3}, we consider the following two cases:

$-$ Case A. $2\le q<\infty$\,;

$-$  Case B. $q_c<q<2$\,.\smallskip

In Case A, we have by using \eqref{eq:BL}
$$
\bigl\|(\phi_{\Gamma}\psi_{\mu,\Gamma'})\bigr\|_{L^q_tL^r_x}
\le C\sigma^{-\frac{4}{q(n+1)}}\mu^{\eps}\,
\EEE(\phi_{\Gamma})^\frac{1}{2}\EEE(\psi_{\mu,\Gamma'})^\frac{1}{2}\,.
$$
Using \eqref{eq:LV-11-el} with strict inequality, we have
$$
\beta_--\alpha_1+\frac{n-1}{q(n+1)}>3\eps
$$
by taking $\eps$ small enough. This yields \eqref{eq:h-l-3}.\medskip

In Case B, we have by using the Main Proposition
$$
\bigl\|(\phi_{\Gamma}\psi_{\mu,\Gamma'})\bigr\|_{L^q_tL^r_x}
\le C\sigma^{-\frac{4}{q(n+1)}}\mu^{\frac{1}{q}-\frac{1}{2}+\eps}
\EEE(\phi_{\Gamma})^\frac{1}{2}\EEE(\psi_{\mu,\Gamma'})^\frac{1}{2}\,.
$$
Using \eqref{eq:LV-12-el}  with strict inequality, we have
$$
\beta_--\alpha_1+\frac{n-1}{q(n+1)}-2\eps>\frac{1}{q}-\frac{1}{2}+\eps.
$$
by taking $\eps$ small enough. Thus, we have \eqref{eq:h-l-3} and hence \eqref{eq:-+-}. \medskip

Collecting \eqref{eq:+++}\eqref{eq:---} and \eqref{eq:-+-}, we have proved Theorem \ref{thm:-main}.
It remains to prove the Main Proposition.
\subsection{Reduction to Theorem \ref{pp:sigma-0}}
It suffices to prove \eqref{eq:BL}  for $(q,r)$ 
on the endline $\frac{2}{q}=(n+1)\Bigl(1-\frac{1}{r}\Bigr)$ 
apart from  the critical index $(q_c,r_c)$ defined in \eqref{eq:cri-ind}.

To handle the large angle case $\sigma\sim 1$, we shall use the following result, which is simply a restatement of the main proposition. We take $\sigma=\frac{1}{8}$ to illustrate the idea without loss of generality.
\begin{prop}
	\label{pp:sigma-1}
	Let $n\ge 2$.  For any $\eps>0$ and $1\le q,r\le\infty$ such that 
	$$
	\frac{1}{q}<\min\Bigl(1,\frac{n+1}{4}\Bigr),\quad
	\frac{1}{q}=\frac{n+1}{2}\Bigl(1-\frac{1}{r}\Bigr)\,,
	$$
	there exists a finite constant $C=C_{\eps,q,r,n}$ such that we have
	$$
	\|uv\|_{L^q_t(\R; L^r_x(\R^{n}))}\le \mu^{\max\bigl(\frac{1}{q}-\frac{1}{2},0\bigr)+\eps}
	\EEE(u)^{\frac{1}{2}} \EEE(v)^\frac{1}{2}	$$
	for all $\mu\ge 1$ and all $u,v$ being either $(+)$ and $(-)$ waves with 
	\begin{equation}
	\label{eq:spt-sigma-1-A}
	\supp\,\widetilde{u}\subset \,\Gamma^+_{\frac{1}{8}},\quad
	\supp\,\widetilde{v}\subset\,\mu\,\Gamma^-_{\frac{1}{8}},
	\end{equation}
	or being $(+)$ waves satisfying the conditions
	\begin{equation}
	\label{eq:spt-sigma-1-B}
	\supp\,\widetilde{u}\subset \,\Gamma^+_{\frac{1}{8}},\quad
	\supp\,\widetilde{v}\subset\,-\mu\,\Gamma^-_{\frac{1}{8}}.
	\end{equation}
\end{prop}
This is simply an extension of Theorem 1.1 in \cite{TaoMZ} to the mixed-norms with  $\mu=2^k$ for all $k\ge 1$ there. We remark that if $k\sim 1$, this result with $q\le 2$ has been established in \cite{Temur}, where the argument is technically simpler since one needs not deal with the high-low frequency interactions for $k\gg 1$.\medskip

The essential part of the proof for the main proposition is to handle the small angular case $\sigma\ll 1$. We shall reduce the question to Theorem \ref{pp:sigma-0}
by using the angular rescaling as in \cite{LeeVargas}.
Taking conjugate if necessary  $|uv|=|u\bar{v}|$, it suffices to consider the case when $u$ and $v$ are both $(+)$ waves satisfying \eqref{eq:sup-2}.

Let $\sigma_0>0$ be given by Theorem \ref{pp:sigma-0} 
and $$\theta(\xi')=\sqrt{1+|\xi'|^2}-1,\quad \xi'\in \R^{n-1}.$$ For each $\sigma\in (0,\sigma_0]$, define
 $$\Phi_\sigma(\xi)=\xi_n \sigma^{-2}\theta(\sigma\xi'/\xi_n).$$
It is easy to see that $\Phi_\sigma$ satisfies the conditions in $\mathscr{E}_\sigma$.
 \smallskip

Denote by $f(x)=u(0,x)$ and $g(x)=v(0,x)$ and let
$$f_\sigma(x)=\sigma^{-\frac{n-1}{2}}f(\sigma^{-1}x',x_n),\quad
g_\sigma(x)=\sigma^{-\frac{n-1}{2}}g(\sigma^{-1}x',x_n),$$
where $x=(x',x_n)$ with $x'\in\R^{n-1}$
and $x_n\in\R$, so that 
$$
\|f_\sigma\|_{L^2(\R^n)}=\EEE(u)^{1/2},\;
\|g_\sigma\|_{L^2(\R^n)}=\EEE(v)^{1/2}\,.
$$
Changing variables 
$$x_n\to x_n-t,\; (\xi',\xi_n)\to (\sigma\xi',\xi_n),\; (x',x_n,t)\to (\sigma^{-1}x',x_n,\sigma^{-2}t),$$ 
we have
$$
\|uv\|_{L^q_tL^r_x}\lesssim \sigma^{-\frac{2}{q}+(n-1)\bigl(1-\frac{1}{r}\bigr)}
\bigl\|\bigl(S^{\Phi_\sigma}_{\varSigma_1} f_\sigma\bigr)\bigl(S^{\Phi_\sigma}_{\mu\varSigma_2} g_\sigma\bigr)\bigr\|_{L^q_tL^r_x},
$$
where $\varSigma_1$ and $\varSigma_2$ fulfill the condition in Theorem \ref{pp:sigma-0}.
Applying the uniform bilinear estimate in Theorem \ref{pp:sigma-0} for all $\sigma\le \sigma_0\ll 1$, we obtain 
$$
\|uv\|_{L^q_tL^r_x}\le C_\eps \mu^{\max\bigl(\frac{1}{q}-\frac{1}{2},0\bigr)+\eps}\Bigl(\frac{1}{\sigma}\Bigr)^{\frac{2}{q}-\frac{n-1}{r'}} \EEE(u)^\frac{1}{2}\EEE(v)^{\frac{1}{2}},
$$
with $C_\eps$ independent of $\sigma$.
Therefore, to complete the proof of Theorem \ref{thm:-main}, it suffices to show Theorem \ref{pp:sigma-0} and Proposition \ref{pp:sigma-1}.\smallskip

In the rest part of the paper, we will only prove Theorem \ref{pp:sigma-0} with full details, and make remarks on the necessary modifications  to get Proposition \ref{pp:sigma-1}.
We shall assume $\mu=2^k$ for some integer $k\ge 1$.

\section{Fundamental properties of waves}
\label{sect:w-p-d}

For any $\sigma\in (0,\sigma_0]$ with $\sigma_0$ small to be fixed. Take $\Phi_\sigma\in \mathscr{E}_\sigma$ and let
\begin{equation}
\label{eq:rnb}
F_1^\sigma(x,t)=S_{2^m\varSigma_1}^{\Phi_\sigma}(t) f_1(x),\;
F_2^\sigma(x,t)=S_{2^k\varSigma_2}^{\Phi_\sigma}(t) f_2(x)\,,
\end{equation}
where $f_1,f_2\in\mathcal{ S}(\R^n)$.

We shall call  $F^\sigma_1$ and $F^\sigma_2$
the \emph{red} and \emph{blue} $\Phi_\sigma-$waves of frequency  $2^m$ and $2^k$ respectively. 
Let $\mathcal{C}^{\Phi_\sigma}=\{(\Phi_\sigma(\xi),\xi):\xi\in\R^n\}$ and denote
$$
\mathcal{C}_1^{\Phi_\sigma,2^m}=\{(\Phi_\sigma(\xi),\xi):\xi\in
2^m\varSigma_1\},\quad
\mathcal{C}^{\Phi_\sigma,2^k}_2=\{(\Phi_\sigma(\xi),\xi):\xi\in 2^k\varSigma_2\}\,,
$$
where $\varSigma_1$ and $\varSigma_2$ are defined  in Theorem \ref{pp:sigma-0}.
The spacetime Fourier transforms of $F^\sigma_1$ and $F^\sigma_{2}$ are supported respectively on $\mathcal{C}_1^{\Phi_\sigma,2^m}$ and $\mathcal{C}_2^{\Phi_\sigma,2^k}$.

When there is no need to distinguish the color, we shall also call $F^\sigma$ a $\Phi_\sigma-$\emph{wave} or simply a \emph{wave}. The \emph{energy} of a wave $F^\sigma(x,t)$ is defined as before by letting
$$
\EEE(F^\sigma):=\|F^\sigma(\cdot,t)\|_{L^2(\R_x^{n})}^2\,\,.
$$

We shall prove, through Section 3 to Section 5, that there exists a $\sigma_0>0$ such that $\forall\,\eps>0$ and $(q,r)$ satisfying the conditions in Theorem \ref{pp:sigma-0}, the bilinear estimate
\begin{equation}
\label{eq:bil-unif}
\|F^\sigma_1F^\sigma_2\|_{L^q_tL^r_x}\lesssim 2^{k\gamma(q,\eps)}
\EEE(F^\sigma_1)^{1/2}\EEE(F^\sigma_2)^{1/2},
\end{equation}
with $\gamma(q,\eps):=\max\Bigl(\frac{1}{q}-\frac{1}{2},0\Bigr)+\eps$,
 holds for all
 the red and blue $\Phi_\sigma$-waves $F^\sigma_1, F^\sigma_{2}$ of frequency $1$ and $2^k$, 
 uniformly with respect to 
  $\Phi_\sigma\in\mathscr{E}_\sigma$ and $\sigma\in(0, \sigma_0]$. The implicit constant
  in \eqref{eq:bil-unif} depends at most on $
 {\varSigma_1,\varSigma_2,\sigma_0,q,r,\eps}$.

\subsection{The law of propagation }
Let $\Xi^{\Phi_\sigma}_j=\bigl\{-\nabla\Phi_\sigma(\xi)\,:\, \xi\in \varSigma_j\bigr\}$ 	for  $j=1,2$. Then $\Xi_1^{\Phi_\sigma}, \Xi_2^{\Phi_\sigma}$ are disjoint
subsets of   an $(n-1)$ dimensional  hypersurface in $\R^n$.

Let ${\Xi}_j^{\Phi_\sigma,*}=\,\Xi_j^{\Phi_\sigma}$ be a small neighborhood of  $\Xi_j^{\Phi_\sigma}$  with the enlargement constant being independent of $\sigma_0$ such that if we  let
$$
\mathbf{ \Lambda}_j^{\Phi_\sigma}(\zzz
_0):=\zzz_0+\bigl\{t\bigl(\omega,1\bigr); \,t\in\R\,, \omega\in\,
{\Xi}^{\Phi_\sigma,*}_j\bigr\},
$$
with $\zzz_0=(x_0,t_0)$, then
$\bm{\Lambda}_1^{\Phi_\sigma}(\zzz_0)$
and $\bm{\Lambda}_2^{\Phi_\sigma}(\zzz_0)$
are  conic hypersurfaces in $\R^{n+1}$ meeting transversely for all $\Phi_\sigma\in\mathscr{E}_\sigma$ and $\sigma\le\sigma_0$ by taking $\sigma_0$  small enough.

\begin{prop}
	\label{pp:KKK}
 For each $j=1,2$, let $\varSigma_j^*$ be a small neighborhood of $\varSigma_j$ such that there exists a universal constant $c_0>0$ such that if we let $$\bigl(\Xi_j^{\Phi_\sigma,*}\bigr)^c:=\bigl\{-\nabla\Phi_\sigma(\xi):\xi\in\R^n,\,|\xi|=1\bigr\}\setminus \Xi_j^{\Phi_\sigma,*},$$
 and  $a_j\in C_0^\infty(\varSigma_j^*)$
 with $a_j(\xi)=1$ for all $\xi\in\varSigma_j$, then 
 $$
\min_{\iota=\pm} \angle \bigl(\iota(\omega,1), (-\nabla\Phi_\sigma(\xi),1)\bigr)\ge c_0
 $$
 for all $\omega\in \bigl(\Xi_j^{\Phi_\sigma,*}\bigr)^c$
 and 
 $\xi\in\supp\,a_j$.
Let	$$
	\mathcal{K}_j^{\Phi_\sigma}(x,t)=\int
	e^{ i (x\cdot \xi+t\Phi_\sigma(\xi))}
	a_j(\xi)\;d\xi\,.
	$$
	Then,
	\begin{equation}
	\label{eq:K-*}
	S_{\varSigma_j}^{\Phi_\sigma}(t)f(x)=\bigl[ \mathcal{K}_j^{\Phi_\sigma}(\cdot,t)*f\bigr](x),
	\end{equation}
	with
	\begin{equation}
	\label{eq:KKK}
	\bigl|\mathcal{K}_j^{\Phi_\sigma}(x,t)\bigr|\lesssim_{M}
	\bigl(1+\mathsf{dist}\bigl((x,t),\,{\mathbf{ \Lambda}_j^{\Phi_\sigma}}\;\bigr)\bigr)^{-M}
	\end{equation}
	for all $(x,t)\in\R^{n+1}$  and all integers $M\ge 1$.
	Here $\mathbf{ \Lambda}^{\Phi_\sigma}_j:=\mathbf{ \Lambda}^{\Phi_\sigma}_j(\zzz)_{\restriction_{\zzz=\mathbf{0}}}$ and the implicit constant in \eqref{eq:KKK} depends on $c_0$, but is uniform in $\sigma\in(0,\sigma_{ 0}]$.
\end{prop}

\begin{proof}
	By definition, we have \eqref{eq:K-*}. If $(x,t)$ is in a $O(1)-$neighbourhood of $\mathbf{ \Lambda}^{\Phi_\sigma}_j$, then  \eqref{eq:KKK} follows from  the Hausdorff-Young inequality. 
	Let $\mathsf{dist}((x,t),\mathbf{ \Lambda}^{\Phi_\sigma}_j)\ge C$ with   large $C$.
	Define
	$\displaystyle
	L=\frac{(x+t\nabla\Phi_\sigma(\xi))\cdot \nabla_\xi}{i\,|x+t\nabla\Phi_\sigma(\xi)|^2},
	$
	Then 
	$$
	L^M e^{ i (x\cdot \xi+{t}\Phi_\sigma(\xi))}
	=e^{i (x\cdot \xi+{t}\Phi_\sigma(\xi))},\quad \forall \; M\ge 1.
	$$
	Integrating by parts,
	we get \eqref{eq:KKK}.
	The proof is complete.
\end{proof}

\subsection{Energy estimates on conic regions of opposite colour}

\begin{lemma}
	\label{lem:opposite}
	Let $F^\sigma_1, F^\sigma_{2}$ be  red and blue waves   of frequency $2^m$ with $m\ge 0$ and let  $\mathbf{ \Lambda}_j^{\Phi_\sigma}(\zzz_0,r)$ be an $O(r)-$neighbourhood of $\mathbf{ \Lambda}_j^{\Phi_\sigma}(\zzz_0)$ with $r>0$ and  $\zzz_0=(x_0,t_0)$.
	 Then,  we have
	\begin{equation}
	\label{eq:opp}\| F^\sigma_j\|_{L^2(\mathbf{ \Lambda}^{\Phi_\sigma}_k(\zzz_0,r))}\lesssim  r^{1/2}
	\EEE(F_j^\sigma)^{1/2},\;\quad \forall\; j,k\in \{1,2\},\; j\ne k\,,
	\end{equation}
	for all $\zzz_0\in\R^{n+1}$ and  $r\gtrsim 2^{-m}$.
	The implicit constant in \eqref{eq:opp} might depend on $\sigma_0$ but is uniform with respect to $\sigma$ such that $0<\sigma\le \sigma_0$.
\end{lemma}

\begin{proof}
	By translation invariance, we may take $\zzz_0$ as the origin.
	We only show \eqref{eq:opp} for $(j,k)=(1,2)$ and by symmetry the other case follows. 
	
	Consider first $m=0$.
	Let
	$$
	\mathfrak{D}_2^{\sigma,t}=\bigl\{{x}\in\R^{n};\,\mathsf{dist}((x,t),\mathbf{ \Lambda}_2^{\Phi_\sigma})\lesssim r\bigr\}.
	$$
	Let $S^{\Phi_\sigma,\,*}_1$ be the adjoint of $S_1^{\Phi_\sigma}$.
	By  $TT^*$, it suffices to show
	$$
	\Bigl\| \int S_1^{\Phi_\sigma,\,*}(t)\bigl[ \indic_{\mathfrak{D}^{\sigma,t}_2}\, H(\cdot,t)\bigr]dt \Bigr\|_{L^2(\R^{n})}
	\lesssim  r^{1/2}\|H\|_{L^2(\R^{n+1})}
	$$
	for all $H\in L^2(\R^{n+1})$.
	Taking squares and multiplying out, we have
	\begin{multline}
	\Bigl\| \int S_1^{\Phi_\sigma,\,*}(t)\bigl[ \indic_{\mathfrak{D}^{\sigma,t}_2}\, H(\cdot,t)\bigr]dt \Bigr\|_{L^2(\R^{n})}^2\\
	\lesssim
	\iint
	\indic_{\mathfrak{D}^{\sigma,\tilde{t}}_2}(\tilde{x})\,
	\mathcal{K}_1^{\Phi_\sigma}(\tilde{x}-x,\tilde{t}-t)\;
	\indic_{\mathfrak{D}^{\sigma,t}_2}(x)\, H(x,t)\;\overline{H(\tilde{x},\tilde{t})}\;dx d\tilde{x}dtd\tilde{t},\label{eq:HHH}
	\end{multline}
	where $\mathcal{K}_1^{\Phi_\sigma}$ is given by Proposition \ref{pp:KKK} with $a_1$ replaced by $|a_1|^2$ there.
	Split the integral \eqref{eq:HHH} over time variables to the 
	$|t-\tilde{t}|\lesssim  r$ and $|t-\tilde{t}|\gg  r$ part.
	
	By  Cauchy-Schwarz and the $L^2-$boundedness  of
	$$ \sup_{t,\tilde{t}}\;\bigl\|\indic_{\mathfrak{D}^{\sigma,\tilde{t}}_2}
	S^{\Phi_\sigma}_{ 1}(\tilde{t})\circ S^{\Phi_\sigma,\,*}_{ 1}(t)\indic_{\mathfrak{D}^{\sigma,t}_2}\bigr\|_{L^2(\R^{n})\to L^2(\R^{n})}={O}(1),$$  the  $|t-\tilde{t}|\lesssim  r$ part of the integral \eqref{eq:HHH} is  bounded by $ r\|H\|_2^2$.
	\smallskip

   Next, using  \eqref{eq:KKK} 
	$$
		|\mathcal{K}_1^{\Phi_\sigma}(\tilde{x}-x,\tilde{t}-t)|\lesssim_{M}\Bigl(1+\mathsf{dist}((\tilde{x}-x,\tilde{t}-t)),\mathbf{ \Lambda}^{\Phi_\sigma}_1\Bigr)^{-M}
	$$
	and the constraints
	for all $x\in\mathfrak{D}_2^{\sigma,t}$, $\tilde{x}\in\mathfrak{D}_2^{\sigma,\tilde{t}}$,
	one finds that  the $|t-\tilde{t}|\gg r$ part of \eqref{eq:HHH}
	can be bounded by 
	$$
   \iint (1+|t-\tilde{t}|/r)^{-M}\|H(\cdot,t)\|_2\|H(\cdot,\tilde{t})\|_2dtd\tilde{t},
	$$
	where, we have taken $\sigma_0$ small so that the amount of the transversality between the conic surfaces $\mathbf{ \Lambda}_1^{\Phi_\sigma} $ and $\mathbf{ \Lambda}_2^{\Phi_\sigma}$ depends only on $\varSigma_1,\varSigma_2,\sigma_0$ but independent of $\sigma\in(0,\sigma_0]$ and $\Phi_\sigma\in\mathscr{E}_\sigma$.
	 The result follows  by Schur's test.
	 
	 For the general case $m\ge 1$, we may apply the result to $F^{\Phi_\sigma}_1(2^{-m}x,2^{-m}t)$ and use rescaling. The proof is complete.
\end{proof}

\begin{coro}
	\label{coro:opp}
	There is a constant $C>0$ depending only on $\varSigma_1,\varSigma_2$ and $\sigma_0$ such that if $\sigma_0$ is sufficiently small, then the following property holds: 
	Let $F^\sigma_1$ be a red $\Phi_\sigma-$wave of frequency $1$ and 
	$F_2^\sigma$ be a blue $\Phi_\sigma-$wave of frequency $2^k$
	associated to some $\Phi_\sigma\in\mathscr{E}_\sigma$ with $\sigma\in(0,\sigma_0]$.
	Let $R\ge r\gg 1$, $\zzz_0=(x_0,t_0)\in\R^{n+1}$ and  set
	$$\mathbf{ \Lambda}^{\Phi_\sigma}_{\mathsf{purple}}(\zzz_0,r)=\mathbf{ \Lambda}_1^{\Phi_\sigma}(\zzz_0,r)\cup\mathbf{ \Lambda}_2^{\Phi_\sigma}(\zzz_0,r).$$
    Then, we have
	$$
	\|F_1^\sigma F_2^\sigma\|_{L^1(Q_R\cap\mathbf{ \Lambda}^{\Phi_\sigma}_{\mathsf{purple}}(\zzz_0,r))}
	\le C \sqrt{rR}\;\EEE(F_1^\sigma)^{1/2}\EEE(F_2^\sigma)^{1/2},
	$$
	for all $Q_R$  spacetime cube of length $R$. 
\end{coro}
\begin{proof}
	The result follows from using Lemma  \ref{lem:opposite}
	and the triangle inequality.
\end{proof}

\subsection{The careful  wavepacket decomposition }

	Let $N>1$ be a large integer depending only on the dimension $n$ and  $C_0=2^{\lfloor N/\eps_0 \rfloor^{10}}$ be a  larger integer where $\eps_0>0$ is made small when necessary but never tending zero.

	Let
$0<\varpi\le 2^{-C_0}$ and $\varrho=R^{1/2}$. Define
$$\mathcal{ L}=\varpi^{-2}\varrho\,\Z^{n},\quad \varGamma=\varrho^{-1}\Z^{n-1}. $$
Fix $\Phi_\sigma\in\mathscr{E}_\sigma$.
For each $(\vvv,\mmu)\in\mathcal{ L}\times\varGamma$,
define a $(\Phi_\sigma,\varrho)-$\emph{tube} $T=T^{\Phi_\sigma}_{\varrho}(\vvv,\mmu)$ as
\begin{equation}
\label{eq:def-tube}
T=\bigl\{(x,t)\in\R^{n+1}:|x-\vvv+t\,\nabla_{\xi}\Phi_\sigma(\mmu,1)|\le\varrho\bigr\}.
\end{equation}
Denote by $\TT_\sigma=\TT_{\Phi_\sigma}$ the collection of such tubes and write 
$(\vvv_T,\mmu_T)$ as the coordinate parametrizing a given tube $T\in \TT_\sigma$ in the sense of \eqref{eq:def-tube}.

For any $T$,  let
$$\psi_{T}(x,t)=\min\bigl\{1, \mathsf{dist}((x,t),T)^{-100N}\bigr\}.$$ 
be the bump function adapted to $T$.

A \emph{disk} is a subset $D$ resident in $\R^{n+1}_{x,t}$ of the form 
$$D=D(x_D,t_D;r_D)=\{(x,t_D):|x-x_D|\le r_D\},$$
for some $(x_D,t_D)\in \R^{n+1}$, which is called the \emph{center} of $D$ and $r_D> 0$, which is called the \emph{radius} of $D$.
We call $t_D$ the \emph{time coordinate} of  $D$.
The indicator function of $D$ is defined as $$\indic_D(x)=\begin{cases}
1\, ,\;(x,t_D)\in D\,,\\  0\;,\; (x,t_D)\not\in D\,.
\end{cases}$$
The bump function adapted to $D$ is defined as 
$$
w_D(x):=\Bigl(1+\frac{|x-x_D|}{r_D}\Bigr)^{-N^{100}}.
$$
For any $F\in C^\infty(\R^{n+1})$ and disk $D$,  we write
\begin{align*}
&\|F\|_{L^2(D)}:=\;\Bigl(\int_{|x-x_D|\le r_D}\bigl|F(x,t_D)\bigr|^2dx\Bigr)^\frac{1}{2}.
\end{align*}

We use   $Q=Q(x_Q,t_Q;r_Q)$ to denote a spacetime cube $Q$  of side-length $r_Q$ centered at $\zzz_Q:=(x_Q,t_Q)\in\R^{n+1}$.  For any $C>0$, we write $C Q=Q(x_Q,t_Q;C r_Q)$, and
we call $\mathsf{LS}(Q):=[t_{Q}- \frac{r_{Q}}{2},\,t_{Q}+\frac{r_{Q}}{2}]$ the \emph{lifespan} of $Q=Q(x_{Q},t_{Q};r_{Q})$.

Let $Q=Q_R\subset \R^{n+1}$ be a spacetime 
cube of length $R\ge C_0 2^{k}$. Let $j\ge0$ be an integer and partition $Q$ into $2^{(n+1)j}$ many subcubes of 
side-length $2^{-j}R$. Denote $\mathcal{ Q}_j(Q)$
as the collection of these cubes. Let $J\approx \log R$ such that each $\qq$ in $\mathcal{ Q}_J(Q)$ is a  cube of side-length $\varrho$.

For any $\Phi_\sigma$-wave $F^\sigma$ of frequency one whose spacetime Fourier transform $\widetilde{{ F}^\sigma}$ is supported on the surface ${ C}^{\Phi_\sigma}_{\varSigma}:=\{(\Phi_\sigma(\xi),\xi):\xi\in\varSigma\}$, we define the \emph{margin} of ${ F}^\sigma$ 
\begin{equation}
\label{eq:margin}
\mathsf{marg}({ F}^\sigma):=\mathsf{dist}\bigl(\supp(\widetilde{{ F^\sigma}}),\;\mathcal{ C}^{\Phi_\sigma}\setminus \mathcal{ C}^{\Phi_\sigma}_{\varSigma^*}\bigr),
\end{equation}
where $\varSigma^*$ is the enlargement of $\varSigma$ in Proposition \ref{pp:KKK} with $\varSigma$ being either $\varSigma_1$ or $\varSigma_2$.

If $F^\sigma$ is a $\Phi_\sigma-$wave of frequency $2^m$,  we define $\mathbb{D}_{2^m}$ by letting
$$\mathbb{D}_{2^m}:F^\sigma(x,t)\mapsto \;F^\sigma(2^{-m}x,2^{-m}t)\;,$$
so that $\mathbb{D}_{2^m}F^\sigma(x,t)=S^{\Phi_\sigma}_{\varSigma}(t)\bigl[\mathbb{D}_{2^m}F^\sigma(\cdot,0)\bigr](x)$
with $\varSigma$ being equal to $\varSigma_1$ or $\varSigma_2$,
is a wave of frequency one.
Define the margin  of $F^\sigma$ as 
$
\mathsf{marg}(F^\sigma)=\mathsf{marg}\bigl(\mathbb{D}_{2^m}F^\sigma\bigr).
$
Thus, by homogeneity,
$$
\mathsf{marg}(F^\sigma)=2^{-m}\mathsf{dist}\bigl(\supp(\widetilde{{ F}^\sigma}),2^m (\mathcal{ C}^{\Phi_\sigma}\setminus\mathcal{ C}^{\Phi_\sigma}_{\varSigma^*})\bigr).
$$

	Let $\mathsf{Q}=\underbrace{[-1/2,1/2)\times\cdots\times[-1/2,1/2)}_{(n-1)\,\text{times}}$
and $\indic_\mathsf{Q}$ be the characteristic function of the unit box $\mathsf{Q}$.	
For each $\mmu\in\varGamma$ and $\zeta\in\R^{n-1}$,  define 
$$
\mathbb{P}^{\mmu,\,\varrho}_{\zeta}:f\mapsto\int_{\varSigma^*} 
e^{i\langle x,\xi\rangle}
\indic_{Q+\zeta}\bigl(\varrho\bigl(\xi_n^{-1}{\xi'}-\mmu\bigr)\bigr)\wh{f}(\xi)\,d\xi'd\xi_n,
$$
where $\mathsf{Q}+\zeta:=\{\eta+\zeta:\eta\in\mathsf{Q}\}$.

\begin{lemma}
	\label{lem:wp-d}
    There exists  $C>0$ depending only on $n$, such that the following 
    statement holds:  
    
    For any spacetime cube $Q=Q_R$ of side-length $R$, for any $\varpi\in (0,2^{-C_0}]$, there is a linear map $F^\sigma\mapsto \{F^\sigma_T\}_{T\in\TT_\sigma}$ such that for any $\Phi_\sigma-$wave $F^\sigma$ of frequency one, each $F_T^\sigma$ is also a $\Phi_\sigma$-wave of the same colour  and frequency to $F^\sigma$ with the relaxed margin property
    \begin{equation}
    \label{eq:re-mar}
    \mathsf{marg}(F^\sigma_T)\ge\mathsf{marg}(F^\sigma)-C \varrho^{-1},
    \end{equation}
    and 
	\begin{equation}
	\label{eq:wpd}
	F^\sigma(x,t)=\sum_{T\in\TT_\sigma}F^\sigma_{T}(x,t),\quad\forall \,(x,t)\in\R^{n}\times\R\,.
	\end{equation}
	Moreover, we have:
	\begin{itemize}
		\item[-] For any   $t\in \mathsf{LS}(C_0Q)$ and every $\, T\in\TT_\sigma$
		\begin{equation}
		\label{eq:up}
		\EEE(F^\sigma_{T})\le C\,\varpi^{-C}\int_{\mathsf{Q}}\,\bigl\|\psi_{T}(\cdot,t)\,\mathbb{P}_{\zeta}^{\mmu_T,\varrho} F^\sigma_{T}(\cdot,t)\bigr\|^2_{L^2(\R^n)} \;d\zeta\,,
		\end{equation}
		\item[-] Concentration of $F^\sigma_T$ on $T$
		\begin{equation}
			\label{eq:rapdcay}
			\|F^\sigma_{T}\|_{L^\infty(Q)}\le C \Bigl(1+\mathsf{dist}\bigl(T,Q\bigr)\Bigr)^{-N}
			\EEE(F^\sigma)^{1/2}\,,
		\end{equation}
	\item[-] No local  accumulating of tube multiplicities
		\begin{equation}
		\label{eq:strongcon}
		\sum_{T\in\TT_\sigma}\sup_{\qq\in\mathcal{Q}_J(Q)}
		\|\psi_{T}^{-50}F^\sigma_{T}\|^2_{L^2(C\qq)}\le C\,\varpi^{-C}\varrho\, \EEE(F^\sigma)\,,
		\end{equation}
	\item[-] 	The Bessel type inequality 
	\begin{equation}
	\label{eq:bessel}
	\Biggl(\sum_{\varDelta}\EEE\,\Bigl(\sum_{T\in\TT_\sigma} m^{\varDelta,T} F^\sigma_T\Bigr)\Biggr)^{\frac{1}{2}}
	\le (1+C\,\varpi)\,\EEE(F^\sigma)^{1/2},
	\end{equation}
	holds  for all $m^{\varDelta,T}\ge 0$ such that $$\displaystyle\sup_{T\in\TT_\sigma}\sum_{\varDelta} m^{\varDelta,T}\le 1$$ where  $\sum_\varDelta$ is summing over a finite number of $\varDelta$'s.
	\end{itemize}
\end{lemma}
\begin{proof}
	By translation in the physical spacetime and the modulation in the frequency space, we may take $\zzz_Q$ to be  the origin of $\R^{n+1}$.\smallskip

	\noindent\emph{- Step 1, The decomposition map.}
	Let  $\mathbf{ \Upsilon}_0\in\mathcal{S}(\R^{n})$
	be a non-negative Schwartz function
	such that $\wh{\mathbf{ \Upsilon}}_0$ is supported in $B(0,1/10):=\{\xi\in\R^{n};\,|\xi|\le 1/10\}$ and  $\wh{\mathbf{ \Upsilon}}_0$ equals to one on $B(0,1/20)$.  Put
	$
	\mathbf{\Upsilon}_\vvv(x)=\mathbf{ \Upsilon}_0(\varpi^2\varrho^{-1}(x-\vvv)),\, \vvv\in\mathcal{ L}\,.
	$
	By  the Poisson summation formula, we have
	$\sum_{\vvv\in\mathcal{ L}}\mathbf{\Upsilon}_\vvv(x)=1$ for all $x\in \R^n$.\smallskip

	For each
	$(\vvv,\mmu)\in\mathcal{ L}\times\varGamma$, let
	for every $x$ and $\,\xi=(\xi',\xi_n)\in\varSigma^*$
	$$
	a_{\vvv,\mmu}(x,\xi)=\mathbf{\Upsilon}_\vvv(x)
	\bigl(\indic_\mathsf{Q}*\indic_\mathsf{Q}\bigr)(\varrho(\xi_n^{-1}\xi'-\mmu))\,.
	$$
	For any $f\in\mathcal{S}(\R^{n})$ such that
	$\supp  \,\wh{f} \subset\varSigma$,
	define
	$$
	f_{\vvv,\mmu}(x)=\int_{\varSigma^*}
	e^{ i \langle x,\,\xi\rangle}
	a_{\vvv,\mu}(x,\xi) \,\wh{f}(\xi)\;d\xi .
	$$
	Then, by Fubini's theorem, we have $$f(x)=\sum_{(\vvv,\mmu)\in\mathcal{ L}\times\varGamma} f_{\vvv,\mmu}(x),\,\quad\forall x\in\R^{n}.$$
	For each $\Phi_\sigma-$wave $F^\sigma$,
	apply this decomposition with $f(x)=F^\sigma(x,0)$.
	By using the  linearity of the operator $S^{\Phi_\sigma}_\varSigma(t)$, we have
	\begin{equation}
	\label{eq:wp-decomp}
	S^{\Phi_\sigma}_\varSigma(t)f(x)=\sum_{(\vvv,\mmu)\in\mathcal{ L}\times\varGamma}
	S^{\Phi_\sigma}_\varSigma(t)f_{\vvv,\mmu}(x)\,.
	\end{equation}
	Letting  $F^\sigma_{T}(x,t)=S^{\Phi_\sigma}_\varSigma(t) f_{\vvv,\mu}(x)$ with $(\vvv,\mmu)=(\vvv_T,\mmu_T)$, we get  \eqref{eq:wpd}. The relaxed margin property \eqref{eq:re-mar} holds with a fixed $C$ independent of $\Phi_\sigma$. \smallskip
	
	\noindent\emph{- Step 2. Proof of \eqref{eq:up}\eqref{eq:rapdcay}\eqref{eq:strongcon}.}
	Let  $\mathbb{B}:=\{\xi\in\R^{n}; |\xi|\le 50 n\}$ and let
	$\alpha\in C_c^\infty(\R^{n-1})$ be 
	equal to one on $\mathbb{B}$, vanishing outside  $O(1)-$neighborhood of $\mathbb{B}$. Let $\beta\in C_0^\infty([\frac{1}{100},100])$ be equal to one on $[\frac{1}{50},50]$. Let
	$p(\xi',\xi_n)=\alpha(\xi')\beta(\xi_n)$
	and
	$$
	K_{\mmu}^{\sigma,\varrho}(t,x):=\int_{\R^{n}}
	e^{ i\bigl(x'\cdot \xi'+x_{n}\xi_n+t\Phi_\sigma(\xi)\bigr)}
	p\bigl(\varrho (\xi_n^{-1}\xi'-\mmu),\xi_n\bigr)
	\,d\xi 'd\xi_n.
	$$

	We have $F_{T}^\sigma(t,\cdot)=K^{\sigma,\varrho}_{\mmu_T}(t-t_0,\cdot)*F^\sigma_{T}(t_0,\cdot)$ for every $T\in\mathbf{T}_\sigma$ and all $t,t_0\in \R$. 
	Changing variables, we get
	$$
	K_{\mmu}^{\sigma,\varrho}(t,x)={\varrho^{-(n-1)}}
	\int_{\R^{n}}
	e^{ i\bigl(\varrho^{-1}x'\cdot \xi'+(x'\cdot\mmu+x_{n})\xi_n+t\,\Phi_\sigma(\varrho^{-1}\xi'+\mmu\xi_n,\xi_n)\bigr)}
	p\Bigl(\frac{\xi'}{\xi_n},\,\xi_n\Bigr)\,d\xi' d\xi_n.
	$$
	Taylor
	expanding using homogeneity and the $\mathscr{E}_\sigma$ conditions along with $\xi_n\sim 1$
	\begin{multline*}
	\quad\quad\Phi_\sigma(\varrho^{-1}\xi'+\mmu\xi_n,\,\xi_n)=\xi_n\Phi_\sigma(\varrho^{-1}\xi_n^{-1}\xi'+\mmu,1)\\
	=\xi_n\Phi_\sigma(\mmu,1)+{\varrho^{-1}}\nabla_{\xi'}\Phi_\sigma(\mmu,1)\cdot {\xi'}+{R}^{-1}\Psi^R_{\sigma,\mmu}(\xi)\,,
	\end{multline*}	
	where $\Psi_{\sigma,\mmu}^{R}$ is homogeneous of order one and depends on the Hessian of $\Phi_\sigma$ with
	$$
	\sup_{0<\sigma\le \sigma_{ 0}}
	\sup_{\mmu\in\varGamma,\, |\mmu|\lesssim 1}
	\sup_{\xi\,:\,p(\xi)\ne 0}\bigl|\,\partial_\xi^{\gamma}\Psi_{\sigma,\mmu}^R(\xi)\,\bigr|\le C_\gamma
	$$
	for any multi-indices $\gamma=(\gamma_1,\cdots,\gamma_n)\in \Z_{\ge 0}^{n}$,
	 we have
	\begin{equation}
	\label{eq:exp-K}
	K_{\mmu}^{\sigma,\varrho}(t,x)=\varrho^{-(n-1)}
		\int_{\R^{n}}
		e^{ i\varPhi^{\sigma,\varrho}_{\mmu}(x,t;\xi)}
		p\Bigl(\frac{\xi'}{\xi_n},\xi_n\Bigr) 
	\,d\xi' d\xi_n,
	\end{equation}
	where 
	$$
	\varPhi_{\mmu}^{\sigma,\varrho}(x,t;\xi):=
	\varrho^{-1}(x'+t\nabla_{\xi'}\Phi_\sigma(\mmu,1))\cdot \xi'+(x_{n}+x'\cdot\mmu+t\Phi_\sigma(\mmu,1))\xi_n+\frac{t}{R}\Psi_{\sigma,\mmu}^{R}(\xi).
	$$
	Letting
	$$
	\mathscr{L}_{\,\mathtt{I}}=\frac{1+ i^{-1}\bigl(\varrho^{-1}(x'+t\nabla_{\xi'}\Phi_\sigma(\mmu,1))+\frac{t}{R}\nabla_{\xi'}\Psi_{\sigma,\mmu}^{R}(\xi)\bigr)\cdot\nabla_{\xi'}}{1+
		\bigl |\varrho^{-1}(x'+t\,\nabla_{\xi'}\Phi_\sigma(\mmu,1)+\frac{t}{R}\,\nabla_{\xi'}\Psi_{\sigma,\mmu}^{R}(\xi)\bigr|^2},
	$$
		$$
	\mathscr{L}_{\,\mathtt{II}}=\frac{1+i^{-1}(x_{n}+x'\cdot\mmu+t\Phi_\sigma(\mmu,1)+\frac{t}{R}\partial_{\xi_n}\Psi_{\sigma,\mmu}^R(\xi)
		)\;\partial_{\xi_n}}{1+
		\bigl|x_{n}+x'\cdot\mmu+t\Phi_\sigma(\mmu,1)+\frac{t}{R}\partial_{\xi_n}\Psi_{\sigma,\mmu}^R(\xi)\bigr|^2},
	$$
	such that  for any integer $M\ge 1$, we have $$\mathscr{L}_{\,\mathtt{I}}^Me^{ i\varPhi_{\mmu}^{\sigma,\varrho}(x,t;\xi)}=
	\mathscr{L}_{\,\mathtt{II}}^Me^{ i\varPhi_{\mmu}^{\sigma,\varrho}(x,t;\xi)}=e^{ i\varPhi_{\mmu}^{\sigma,\varrho}(x,t;\xi)}.$$
	Noting that
	\begin{multline*}
	1+
	\bigl |\varrho^{-1}(x'+t\nabla_{\xi'}\Phi_\sigma(\mmu,1))+\frac{t}{R}
	\nabla_{\xi'}\Psi_{\sigma,\mmu}^R(\xi)\bigr|+\bigl|x_{n}+x'\cdot\mmu+t\Phi_\sigma(\mmu,1)+\frac{t}{R}\,\partial_{\xi_n}\Psi_{\sigma,\mmu}^{R}(\xi)\bigr|\\
	\gtrsim\frac{1}{C_0}\Bigl(\
		\bigl |\varrho^{-1}(x'+t\nabla_{\xi'}\Phi_\sigma(\mmu,1))\bigr|+\bigr|+\bigl|x_{n}+x'\cdot\mmu+t\Phi_\sigma(\mmu,1)\bigr|\Bigr)
	\end{multline*}
	holds for all $\xi\in\supp \,p$
	and all $|t|\le C_0 R$,
	we have by $M$-fold integration by parts
	\begin{equation}
	\label{eq:K-v-mu-pp}
	\bigl| K^{\sigma,\varrho}_{\mmu}(t,x)\bigr|
	\lesssim_M C_0^M \;\varrho^{-(n-1)}\Bigl(1+\varrho^{-1}\bigl|x+t\,\nabla\Phi_\sigma(\mmu,1)\bigr|\Bigr)^{-M},
	\end{equation}
	where we have used the fact that $\Phi_\sigma(\mmu,1)=\mmu\cdot\nabla_{\xi'}\Phi_\sigma(\mmu,1)+\partial_{\xi_n}\Phi_\sigma(\mmu,1)$
	and  for each $t$, the $x'$ variable is essentially
	contained in the  ball centered at $-t\nabla_{\xi'}\Phi_\sigma(\mmu,1)$ of radius $\varrho$.
With \eqref{eq:K-v-mu-pp} and using the same argument of dyadic decomposition and summing up convergent geometric series exploring the decay properties of bump functions,  c.f. \cite[(103) P.262]{TaoMZ},
we get
\begin{equation}
\label{eq:103tao}
\|F^\sigma\|_{L^2(D(x_D,t_D;\varrho))}\lesssim_{C_0}\|w_{D(x_D-(t-t_D)\nabla\Phi_\sigma(\mmu,1),t;\varrho)}F^\sigma(\cdot,t)\|_{L^2(\R^n)}
\end{equation}
for all disk $D$ and $\Phi_\sigma-$wave $F^\sigma$ such that the spatial Fourier transform is supported on $\xi\in\varSigma$
with $|\xi_n^{-1}\xi'-\mmu|\le 50n \varrho^{-1}$.  Using \eqref{eq:103tao}, $$F^\sigma_T(x,0)=\mathbf{ \Upsilon}_{\vvv_T}(x)\int_\mathsf{Q} 
\mathbb{P}^{\mmu_T,\varrho}_\zeta [F^\sigma(\cdot,0)](x)d\zeta,$$
and that  $\mathbb{P}^{\mmu_T,\varrho}_\zeta F^\sigma(x,t)$ is a $\Phi_\sigma-$wave  with $\mmu=\mmu_T$, satisfying the conditions for the $F^\sigma$ in \eqref{eq:103tao}, we obtain \eqref{eq:up} by Cauchy-Schwarz.

With \eqref{eq:K-v-mu-pp} \eqref{eq:103tao},
the estimates \eqref{eq:rapdcay}  \eqref{eq:strongcon}  can be easily verified by using the same proof in \cite{TaoMZ}.
\smallskip

\noindent\emph{- Step 3, Proof of the Bessel type inequality.}
The argument is same to \cite{TaoMZ,Y22} and we only sketch it.	 By Plancherel's theorem and Minkowski's inequality, the left side of \eqref{eq:bessel} is less or equal to 
	\begin{equation}
	\label{eq:pf-bes}
	 \int_{\mathsf{Q}}\Biggl( \sum_{\varDelta}
	\Bigl\|\sum_{T\in\mathbf{T}_\sigma} m^{\varDelta,\,T} \mathbf{\Upsilon}_{\vvv_T}(\cdot)\, \mathbb{P}^{\mmu_T,\varrho}_{\zeta}f(\cdot)\Bigr\|_{2}^2
	\Biggr)^{\frac{1}{2}}d\zeta\,,
	\end{equation}
	with $f(x)=F^\sigma(x,0)$.
	For each  $\mmu \in\varGamma$, define
	$
	\mathbb{B}_{\varrho,\mmu}=\mmu+\frac{1}{\varrho}\mathsf{Q}
	$
	and let
	$$
	\mathscr{O}=\bigcup_{\mmu\in \varGamma}
	\Bigl\{\zeta\in\mathbb{B}_{\varrho,\mmu};\; \mathsf{dist}(\zeta, \R^{n-1}\setminus \mathbb{B}_{\varrho,\mmu})\ge \varpi^2\varrho^{-1}\Bigr\}.
	$$
	For any $\eta\in \varrho^{-1}\mathsf{Q}$, define
	$$
	\varPi_{\mathscr{O}+\eta}:\;f(x)\mapsto
	\int e^{ i(x\cdot\xi)}
	\indic_{\{\xi\,;\;\xi_n^{-1}\xi'\in\mathscr{O}+\eta\}}(\xi)
	\wh{f}(\xi)\;d\xi \,.
	$$
	Splitting $f=\bigl(\varPi_{\mathscr{O}+\eta} f\bigr)+\bigl(\mathrm{Id}-\varPi_{\mathscr{O}+\eta} \bigr)f$ and using the triangle inequality, we have that
   the right side of  \eqref{eq:pf-bes} $\le \mathbf{I}+\mathbf{II}$, where
	\begin{align}
	\label{eq:main}
	\mathbf{I}=\,&\varrho^{n-1}\int_{\varrho^{-1}\mathsf{Q}}\Biggl( \sum_{\varDelta}
	\Bigl\|\sum_{T\in\TT_\sigma} m^{\varDelta,T} \mathbf{\Upsilon}_{\vvv_T}(\cdot)\, \mathbb{P}^{\mmu_T,\varrho}_{\,\varrho\eta}\circ \varPi_{\mathscr{O}+\eta} f(\cdot)\Bigr\|_{2}^2
	\Biggr)^{\frac{1}{2}}d\eta,\\
	\label{eq:error}
	\mathbf{II}=\,&\varrho^{n-1}\int_{\varrho^{-1}\mathsf{Q}}\Biggl( \sum_{\varDelta}
	\Bigl\|\sum_{T\in\TT_\sigma} m^{\varDelta,\,T} \mathbf{\Upsilon}_{\vvv_T}(\cdot)\, \mathbb{P}^{\mmu_T,\varrho}_{\,\varrho\eta}\circ\bigl( \mathrm{Id}- \varPi_{\mathscr{O}+\eta}\bigr) f(\cdot)\Bigr\|_{2}^2
	\Biggr)^{\frac{1}{2}}d\eta.
	\end{align}
	For $\mathbf{I}$, we use the  Plancherel theorem and the strict orthogonality  from the pairwise $\varpi^{2}\varrho^{-1}$-separateness between the  connected components of $\mathscr{O}$, which allows a petite amplification in the frequency space due to convolution with  $\wh{\mathbf{\Upsilon}}_{\vvv_T}$.
	For $\mathbf{II}$, by using the Plancherel theorem and the almost orthogonality followed with Cauchy-Schwarz and  Fubini, we have
	$$
	\sup_{\xi\in\varSigma^*}\,
	\varrho^{n-1}\int_{\varrho^{-1}\mathsf{Q}}\Bigl(1-\indic_{\mathscr{O}+\eta}(\xi'/\xi_n)\Bigr)\;d\eta\lesssim_n\,\varpi^2\,.
	$$
 We refer to \cite{TaoMZ,Y22} for more details. The proof is complete.
\end{proof}

\begin{remark}
	For waves of high frequency, say $2^m$ with $m\ge 1$, we may use the $\mathbb{D}_{2^m}$ normalizing the frequency to be one and apply Lemma \ref{lem:wp-d}  then scaling back. 
\end{remark}

\begin{remark}
	In the above, we  have always assumed
	the initial data of waves are Schwartz functions.
	This assumption can be removed by standard density argument.
\end{remark}

\subsection{Spatial localizations using  the Huygens principle}
\label{sec:huy}

We  introduce the spatial localization
operators as in \cite{TaoMZ} 
and summarize the properties useful in this paper to capture the energy concentration of waves. Since for any fixed $\Phi_\sigma$, the proof are same to \cite{TaoMZ},
we shall omit most of technical parts.

For any $c>0$, let  $c D:=D(x_D,\,t_D; c\, r_D)$.
Define the  \emph{disk exterior} of $D$ as
\begin{equation}
\label{eq:ext}
D^{\mathsf{ext}}=D^{\mathsf{ext}}(x_D,t_D;r_D)=\{(x,t_D):\, |x-x_D|>r_D\}.
\end{equation}
For any $F\in C^\infty_{loc}(\R^{n+1})$ and disk $D$,  we write
\begin{align*}
\|F\|_{L^2(D^{\mathsf{ext}})}:=\;\Bigl(\int_{|x-x_D|> r_D}\bigl|F(x,t_D)\bigr|^2dx\Bigr)^\frac{1}{2}.
\end{align*}
\subsubsection{The localization operator $P_D$}
We introduce the localization operator $P_D$ as in \cite{TaoMZ}.
Let $\mathbf{\Upsilon}_0(x)\ge 0$ be the Schwartz function in the proof of Lemma \ref{lem:wp-d}.
For every $r>0$,  set $\mathbf{ \Upsilon}_r(x)=r^{-n}\mathbf{\Upsilon}_0(r^{-1}x)$.

\begin{definition}
	Let $F^\sigma(x,t)$ be a $\Phi_\sigma-$wave
	of frequency $1$ for some $\Phi_\sigma\in\mathscr{E}_\sigma$.
	For any disk $D=D(x_D,t_D;r_D)$, we define $P_D^{\Phi_\sigma} F^\sigma$ at time $t_D$ as
	$$
	P_D^{\Phi_\sigma} F^\sigma(t_D)=\Bigl(\indic_D*\mathbf{ \Upsilon}_{r_D^{1-\frac{1}{N}}}\Bigr)\,F^\sigma(t_D),
	$$
	and  $\forall \;t\in\R$
	$$\bigl(P_D^{\Phi_\sigma} F^\sigma\bigr)(t)=S^{\Phi_\sigma}(t-t_D)\bigl[\,P_D^{\Phi_\sigma} F^\sigma(t_D)\bigr].$$
	For any $\Phi_\sigma-$wave  $F^\sigma$ of frequency $2^m$, we define
	$P_{D}^{\Phi_\sigma,2^m}F^\sigma=\mathbb{D}^{-1}_{2^m}\circ P_D^{\Phi_\sigma}\circ \mathbb{D}_{2^m} F^\sigma$.
\end{definition}

We shall abuse notations below, writing $P_D^{\Phi_\sigma}$ instead of $P_D^{\Phi_\sigma,2^m}$ for short and the meaning is clear from context.

We show next that the operator $P_D^{\Phi_\sigma}$ localizes a $\Phi_\sigma-$wave to $D_+$ while $1-P_D^{\Phi_\sigma}$ localizes it to the exterior of $D_-$.
\begin{lemma}
	\label{lem:Tao10.2}
	Let $m\ge 0$ and $D$ be a disk with radius $r_D=r\ge C_02^{-m}$. Then, for any $\Phi_\sigma-$waves $F^\sigma$ with frequency $2^m$, $P_D^{\Phi_\sigma} F^\sigma$ is a $\Phi_\sigma-$wave of frequency $2^m$ and the same colour to $F^\sigma$, satisfies the margin estimate
	$$
	\mathsf{marg}(P_D^{\Phi_\sigma} F^\sigma)\ge
	\mathsf{marg}(F^\sigma)-C_0(2^m r)^{-1+\frac{1}{N}},
	$$
	and the following local energy estimates 
	\begin{align}
	\label{2.16}\| P_D^{\Phi_\sigma} F^\sigma\|_{L^2(D_+^\mathsf{ext})}\lesssim& \; (2^mr)^{-N}\EEE(F^\sigma)^{1/2}\\
	\label{2.17}	
	\|(1-P_D^{\Phi_\sigma})F^\sigma\|_{L^2(D_-)}\lesssim & \; (2^mr)^{-N}\EEE(F^\sigma)^{1/2}\\
	\label{2.18}
	\sup_{t}\|P_D^{\Phi_\sigma} F^\sigma(t)\|_{L^2(\R^{n})}^2\le&\; \|F^\sigma\|_{L^2(D_+)}^2+\O( (2^mr)^{-N} \EEE(F^\sigma))\\
	\label{2.19}\sup_{t} \|(1-P_D^{\Phi_\sigma})F^\sigma(t)\|_{L^2(\R^{n})}^2\le&\; \|F^\sigma\|_{L^2( D^{\mathsf{ext}}_-)}^2+\mathcal{O}((2^mr)^{-N}\EEE(F^\sigma))\\
	\label{2.20}\sup_{t} \|(1-P_D^{\Phi_\sigma} )F^\sigma(t)\|_{L^2(\R^n)}\le &\; \EEE(F^\sigma)^{\frac{1}{2}},\; \sup_{t}\|P_D^{\Phi_\sigma}  F^\sigma(t)\|_{L^2(\R^n)}\le \EEE(F^\sigma)^{\frac{1}{2}},
	\end{align}
	where $D_\pm^\mathsf{ext}$ is the exterior of $
	D_{\pm}:=D\bigl(x_D,t_D;r(1\pm (2^mr)^{-\frac{1}{2N}})\bigr).
	$ The implicit constants are all independent of $\sigma$.
\end{lemma}

\begin{proof}
	The argument is exactly same as \cite{TaoMZ} and we omit it.
	See also \cite{Y22}.	
\end{proof}

\subsubsection{Concentration of red and blue waves on conic sets of the same colour}

\begin{lemma}
	\label{lem:conctr}
		There exists a constant $C>0$ such that the following statement holds.
	Let $R\ge r\ge C_02^{\max\{m,k\}}$ with $k,m\ge 0$. Let $D= D(\zzz_D;\,r)$ with $\zzz_D=(x_D,\,t_D)$. For any red and blue $\Phi_\sigma-$waves $F^\sigma$ and $G^\sigma$
	of frequency $2^m$ and $2^k$ respectively,  we have for any $M\ge 1$
	\begin{align}
	\label{2.21}
	\|P_D^{\Phi_\sigma} F^\sigma\|_{L^\infty\bigl(\R^{n+1}_{x,t}\setminus \mathbf{ \Lambda}^{\Phi_\sigma}_1(\zzz_D,\,Cr+R^{\frac{1}{N}})\bigr)}\lesssim_M\,R^{-M}
	\EEE(F^\sigma)^{1/2},\\
	\label{2.22'}
	\|(1-P_D^{\Phi_\sigma})F^\sigma\|_{L^\infty(Q(\zzz_D,\,r/C))}\lesssim_M\, r^{-M} \,\EEE(F^\sigma)^{1/2},\,\\
	\label{2.21-b}
	\|P_D^{\Phi_\sigma} G^\sigma\|_{L^\infty\bigl(\R^{n+1}_{x,t}\setminus\mathbf{ \Lambda}^{\Phi_\sigma}_{2}(\zzz_D,\,Cr+R^{\frac{1}{N}})\bigr)}\lesssim_M\, R^{-M}
	\EEE(G^\sigma)^{1/2},\\
		\|(1-P_D^{\Phi_\sigma})G^\sigma\|_{L^\infty(Q(\zzz_D,\,r/C))}\lesssim_M\, r^{-M} \,\EEE(G^\sigma)^{1/2}\,	.
	\end{align}
	 All the implicit structure constants are independent of $\sigma\in(0,\sigma_{ 0}]$.
\end{lemma}
\begin{proof}
	We only prove  \eqref{2.21}\eqref{2.22'}. Consider $m=0$ first. The argument is the same as in \cite{Y22} by using the explicit expression for 
	the kernel function of $P_D^{\Phi_\sigma}$ and dividing the integration in question into the local and global parts, where for the global parts, we may use   \eqref{eq:KKK} to get the rapid decay while for the 
	local part, we combine it with the constraints outside the conic set in  \eqref{2.21} to use the rapid decay of $\mathbf{ \Upsilon}_0$.  The proof of \eqref{2.22'} is similar. We refer to Lemma 3.5 of \cite{Y22} for details.
	For the general case, we use the dilation operator $\mathbb{D}_{2^m}F^\sigma$ reducing the question to the above case with $R,r$ replaced by $2^mR,\,2^mr$.
	The proof is complete.
\end{proof}

\section{The red and blue wave tables}
\label{sec:wavetables}

\subsection{Notion of the wave tables on a spacetime  cube}We recall first  the  red and blue wave tables in \cite{TaoMZ}.

A \emph{wave table operation} $\mathfrak{F}_j$ on $Q$ with \emph{depth} $j$ is a linear map such that 
for any $\Phi_\sigma$-wave  $F^\sigma$ for some $\sigma$, we have $\mathfrak{F}_j:F^\sigma\mapsto\mathcal{ F}^{\sigma}_j$, with $\mathcal{ F}_j^\sigma$ being of the  matrix form
\begin{equation}
\label{eq:w-t}
\mathcal{ F}_j^\sigma=\bigl(\mathcal{ F}^{\sigma,(\qq)}\bigr)_{\qq\in\mathcal{ Q}_j(Q)},
\end{equation}
which is  called the \emph{wave table} of $F^\sigma$ on $Q$,
where each $\mathcal{ F}^{\sigma,(\qq)}$ is a $\Phi_\sigma-$wave (of the same colour to $F^\sigma$) as well.
We define $\mathfrak{F}_0$ to be the identity map sending $F^\sigma$
to itself, which is linked in some sense to the cube $Q$.
For a wave table of $F^\sigma$ with given depth $j$, we  call each wave $\mathcal{ F}^{\sigma,(\qq)}$ the \emph{entry} of $\mathcal{F}^\sigma_j$. The Fourier support of $\mathcal{ F}^\sigma_j$ is defined as 
$$\supp\, \widetilde{\mathcal{ F}^\sigma_j}=\bigcup_{\qq\in\mathcal{ Q}_j(Q)}\supp\, \widetilde{\mathcal{ F^{\sigma,(\qq)}}}$$
and the margin of $\mathcal{ F}_j^\sigma$ is defined as in
\eqref{eq:margin} with $F^\sigma$ replaced by $\mathcal{ F}^\sigma_j$.

For any integer $m\ge 0$, we may apply the map $\mathfrak{F}_m$ to each above $\mathcal{ F}^{\sigma,(\qq)}$ to get 
$$
\mathfrak{F}_m(\mathcal{ F}^{\sigma,(\qq)})=\Bigl(\bigl(\mathcal{ F}^{\sigma,(\qq)}\bigr)^{(\qq')}\Bigr)_{\qq'\in\mathcal{ Q}_m(\qq)}.
$$
The composition $\mathfrak{F}_m\circ\mathfrak{F}_j$  defined this way by applying $\mathfrak{F}_m$ to each entry of $\mathfrak{F}_j(F^\sigma)$ gives a new wave table operation $\mathfrak{F}_{m+j}=\mathfrak{F}_m\circ\mathfrak{F}_j$ with depth $m+j$. By doing so, it is possible to obtain a chain $\{\mathfrak{F}_j\}_{j\ge 0}$ of wave table operations.

This general formulation can be extended  to any given $Q$. Moreover, there is in no way a prescribed  paradigm of constructing wave tables. One special construction that matters for us in this paper is the one in \cite{TaoMZ} by iterating $\mathfrak{F}_{j+C_0}:=\mathfrak{F}_{C_0}\circ\mathfrak{F}_{j}$ for each  $j\in C_0\Z$. To this end, we need to give an explicit construction of  $\mathfrak{F}_{C_0}$ first, to which we now turn immediately.

For any $R\ge {C_0}2^{\max\{k,m\}}$, we define $\mathfrak{R}^{\Phi_\sigma,2^m}_R$ as the set of  red waves $F^\sigma$ with  $\widetilde{F^\sigma}$  being an $L^2$ measure on 
$\mathcal{C}_1^{\Phi_\sigma,2^m}$   satisfying the margin condition
$$\mathsf{marg}(F^\sigma)\ge (100)^{-1}-(2^{-m}R)^{-\frac{1}{N}}.$$
Likewise,
define  $\mathfrak{B}^{\Phi_\sigma,2^k}_R$ as the set of  blue waves   $G^\sigma$ such that $\supp\; \widetilde{G^\sigma}\subset \mathcal{C}_2^{\Phi_\sigma,2^k}$  satisfying the margin condition
$$\mathsf{marg}(G^\sigma)\ge (100)^{-1}-(2^{-k}R)^{-\frac{1}{N}}.$$

We construct the wave tables for waves of frequency one first and then generalize the construction to at least one of the waves with high frequency.

Given two waves of opposite colours,
we shall construct $\mathfrak{F}_{C_0}$ for one wave on any given cube  with respect to the other. 
Let $Q=Q_R$ be a spacetime cube of length $R$ and $F^\sigma\in\mathfrak{R}^{\Phi_\sigma,1}_R$ and  $G^\sigma\in\mathfrak{B}^{\Phi_\sigma,1}_R$ be red and blue waves.
We construct the wave table of $F^\sigma$ of depth $C_0$  on $Q$ with respect to $G^\sigma$ based on Lemma \ref{lem:wp-d}.

Denote 
$$F^\sigma=\sum_{T_1\in\TT_1^\sigma}F^\sigma_{T_1},\quad
G^\sigma=\sum_{T_2\in\TT_2^\sigma}G^\sigma_{T_2},
$$
as the decomposition into wave-packets
for the red and blue waves given by \eqref{eq:wpd}.
 For any  $\varDelta \in \mathcal{Q}_{C_0}(Q)$, let
$$\mathsf{K}_{Q}(\varDelta)=\bigl\{\qq\subset \varDelta;\,\qq\in\mathcal{Q}_{J}(Q)\bigr\},$$ with $J\approx \log R$ such that each $\qq$ in $\mathcal{ Q}_J(Q)$ is a  cube of side-length $\sqrt{R}$. Let $\chi\in \mathcal{S}(\R_{x,t}^{n+1})$ be such that the spacetime Fourier transform of $\chi$ is compactly supported in a small neighbourhood of the origin and $\chi\ge 1$ on double of the unit ball, with the structure constants  independent of $\sigma$. Let $\mathcal{A}_\qq$ be the affine transform sending the John ellipsoid inside  $\qq$ to the unit ball such that  if we let  $\chi_\qq=\chi\circ \mathcal{A}_\qq$, then we have $\chi_\qq\ge \indic_\qq$, where $\indic_\qq$ is the characteristic function of $\qq$\,.\medskip

The red wave table $\mathcal{ F}^\sigma$ of $F^\sigma$ \emph{with respect to} the blue wave $G^\sigma$ on $Q$ , with depth $C_0$ depending on $\varpi$ and $R$ is given by 
\begin{definition}
	\label{def:w-t}
	 For each $\varDelta\in\mathcal{Q}_{C_0}(Q)$ and $T_1\in \mathbf{T}_1^\sigma$, define
	$$
	m_{T_1}^{G^\sigma,\,\varDelta}=\sum_{\qq\in\mathsf{K}_{Q}(\varDelta)}\sum_{T_2\in\mathbf{T}_2^\sigma}\bigl\|\chi_\qq\, \psi_{T_1}\,\psi_{T_2}^{-50}\,G_{T_2}^\sigma\bigr\|^2_{L^2(\R_{x,t}^{n+1})},
	$$
	and set $$m_{T_1}^{G^\sigma}=\sum_{\varDelta\in\mathcal{Q}_{C_0}(Q)}m_{T_1}^{G^\sigma,\,\varDelta}\,.$$ 
	The $(\varpi,R)-$\emph{red wave table} $\mathcal{ F}^\sigma=\mathfrak{F}_{C_0}^{\sigma,\varpi}(F^\sigma,G^\sigma;Q)$  at depth $C_0$ for $F^\sigma$
	with respect to $G^\sigma$	over $Q$  is defined as 
	$$
	\mathcal{F}^\sigma
	=\Bigl(\mathcal{F}^{\sigma,\varpi\,(\varDelta)}\Bigr)_{\varDelta\in\mathcal{Q}_{C_0}(Q)},
	$$
	with
	$$
	\mathcal{F}^{\sigma,\varpi,(\varDelta)}(x,t):=\sum_{T_1\in\mathbf{T}^\sigma_1}
	\frac{m_{T_1}^{G^\sigma,\,\varDelta}}{m_{T_1}^{G^\sigma}}F_{T_1}^\sigma(x,t),
	$$
	where $\mathcal{ F}^\sigma$ depends on $\varpi$ through  \eqref{eq:wpd}.
\end{definition}

By the linearity of $S^{\Phi_\sigma}(t)$, each $\mathcal{ F}^{\sigma,\varpi,(\varDelta)}$ is a red wave and 
$$
F^\sigma=\sum_{\varDelta\in\mathcal{Q}_{C_0}(Q)} \mathcal{F}^{\sigma,\varpi,\,(\varDelta)}\,.$$
Define the energy of $\mathcal{ F}^\sigma$ as 
$$
	\EEE(\mathcal{F}^\sigma):=	\sum_{\varDelta\in\mathcal{Q}_{C_0}(Q)} \EEE\bigl(\mathcal{F}^{\sigma,\varpi,(\varDelta)}\bigr).
$$
From the Bessel type inequality \eqref{eq:bessel}, we have
\begin{lemma}
	\label{lem:bessel-w-t}
	There is a constant $C_n$ depending only on $n$ such that we have
	\begin{align}
	\EEE(\mathcal{F}^\sigma)^{1/2}\le \;(1+C_n\,\varpi) \EEE(F^\sigma)^{1/2}\;,
	\end{align}
	for any $(\varpi,R)$-wave tables $\mathcal{F}^\sigma$ with depth $C_0$ over a spacetime cube $Q$.
\end{lemma}

By symmetry,
we may also define the blue wave table $\mathcal{G}^\sigma=\mathfrak{S}^{\sigma,\varpi}_{C_0}(G^\sigma,F^\sigma;Q)$ of $G^\sigma$ \emph{with respect to} any red wave $F^\sigma$ on the cube $Q$ as above.

The above constructions naturally extends 
to the case when one of the waves is of high frequency.
For waves both with high frequencies, we may normalize the red wave to be frequency one by using $\mathbb{D}_{2^m}$ and apply the same construction as above, then undo the scaling.

\subsection{No waste bilinear $L^2$-Kakeya type estimate}

Let $\varpi\in [0, 2^{-C_0}]$ and $j\ge 0$ be an integer. 
For any spacetime cube $Q$,
define the $(\varpi,j)-$\emph{interior} of $Q$  as
$$\mathfrak{I}^{\varpi,j}(Q):=\bigcup_{\qq\in\mathcal{Q}_j(Q)}(1-\varpi)\qq\,.$$
\begin{prop}
	\label{pp:C-0 quilt}
	Let $F^\sigma\in\mathfrak{R}^{\Phi_\sigma,2^m}_R$ and $G^\sigma\in \mathfrak{B}^{\Phi_\sigma,2^k}_R$.
	For any $\varpi\in (0, 2^{-C_0}]$, let $Q=Q_R$ with $R\ge C_0 \max\Bigl(2^{-m},2^{m-2k}\Bigr)$,   there exists a  constant $C$, depending only on $\varSigma_1,\varSigma_2,\sigma_{ 0}$ and independent of $C_0$,
	such that if $F^\sigma$ and $G^\sigma$ are red and blue waves satisfying the margin conditon
	$$
\mathsf{marg}(F^\sigma)\ge(2^mR)^{-1/2},\;\mathsf{marg}(G^\sigma)\ge (2^kR)^{-1/2}
	$$
	 with $\EEE(F^\sigma)=\EEE(G^\sigma)=1$. Let $\mathcal{F}^\sigma,\mathcal{G}^\sigma$ be the $(\varpi,R)-$wave tables for  $F^\sigma$ and $G^\sigma$ with depth $C_0$ over $Q^*:=CQ$ respectively given by Definition \ref{def:w-t}. Then,  we have
		\begin{equation}
	\label{eq:bil-kakeya}
	\max_{\varDelta,\varDelta'\in\mathcal{Q}_{C_0}(Q^*)}
	\bigl\|\mathcal{F}^{\sigma,\varpi,(\varDelta)}\, \mathcal{G}^{\sigma,\varpi,(\varDelta')}\bigr\|_{L^2_{t,x}(\mathfrak{I}^{\varpi,C_0}(Q^*)\setminus\{\varDelta'\})}\le\,C\,\varpi^{-C} \Bigl(\frac{2^m}{R}\Bigr)^{\frac{n-1}{4}}\,.
	\end{equation}
\end{prop}

\begin{proof}
  The argument is exactly same with (54) of \cite{TaoMZ}.
  First, by scaling invariance, one may take $m=0$ so that \eqref{eq:bil-kakeya} is essentially  (73) of \cite{TaoMZ}.
  Then, by using the explicit construction formula in Definition \ref{def:w-t} with $R$ replaced by $2^mR$ there, and \eqref{eq:strongcon} of Lemma \ref{lem:wp-d} to conclude the proof.	There are two places where the transversality condition 
  is used for the family of conic surfaces $\mathcal{ C}^{\Phi_\sigma}_1$ and $\mathcal{ C}^{\Phi_\sigma}_2$. One is the bilinear $L^2-$estimate in low dispersions and the other one is the energy estimate on conic regions of opposite colour of Lemma \ref{lem:opposite}.
  For the first one, we refer to Lemma 14.2 and Lemma 14.3  in \cite{TaoMZ} for more details. We apply our analogue counterpart for $\Phi_\sigma$ surfaces with the high frequency wave at frequency $2^{k-j}$ of angular dispersion $O((2^jR)^{-1/2})$.
   It is easy to see that these two properties can be made uniform for all $\Phi_\sigma\in\mathscr{E}_\sigma$ and all $\sigma\in(0,\sigma_{ 0}]$ by taking $\sigma_{ 0}$ small enough. 
\end{proof}

\subsection{Persistence on the non-concentration of energy}

\begin{definition}
	\label{def:energy-con}
	For any $r>0$, a spacetime cube $Q$, a red wave $F^\sigma$ and a blue wave $G^\sigma$, we define 
	 the $r-$\emph{energy-concentration} of $F^\sigma, G^\sigma$
	 over $Q$ to be the quantity
	$$
	\EEE_{r,Q}(F^\sigma,G^\sigma)=\max\Bigl\{\frac{1}{2}\EEE(F^\sigma)^{1/2}\EEE(G^\sigma)^{1/2},\; \sup_D\|F^\sigma\|_{L^2(D)}\|G^\sigma\|_{L^2(D)}\Bigr\},
	$$
	where $D$ ranges over all disks with $r_D=r$ and $t_D\in \mathsf{LS}(Q)$
\end{definition} 

This quantity is very robust in taking advantage of the Huygens' principle of wave equations,
although this property does not hold anymore for Schr\"odinger equations. However, one can recover this property
by using the method of descent  introduced in \cite{TaoMZ}. From this perspective, there seems a principle of adding dimensions to treat other surfaces
as embeded submannifolds of conic surfaces, whenever this propoerty fails and then use  method of descent to get the expected result as a limiting case.

\begin{proposition}
	\label{pp:persist}
	Let $ R\ge C_0\max\{2^{-m},2^{m-2k}\}$ with $m,k\ge 0$  and $Q=Q_R$ be a  spacetime cube of side-length $R$. For each $r\ge R^{\frac{1}{2}+\frac{1}{N}}2^{-m(\frac{1}{2}-\frac{1}{N})}$, we define $$r^\#:=r(1-C_0(2^mr)^{-\frac{1}{3N}}).$$
	There exists a constant $C>0$, such that  if $F^\sigma\in\mathfrak{R}^{\Phi_\sigma,2^m}_R, G^\sigma\in\mathfrak{B}^{\Phi_\sigma,2^k}_R$
	with $\EEE(F^\sigma)=\EEE(G^\sigma)=1$
	and $\mathcal{F}^\sigma$ is
	a $(\varpi,R^{1/2})$-wave tables for $F^\sigma$ with respect to $ G^\sigma$ over $Q$ with depth $C_0$, then
	\begin{align*}
	\sup_{\varDelta\in\mathcal{Q}_{C_0}(Q)}
	\EEE_{r^\#,5Q}\bigl(\mathcal{F}^{\sigma,\varpi,(\varDelta)},
	{G}^{\sigma}\bigr)
	\le (1+C\varpi)\,\EEE_{r,5Q}(F^\sigma,G^\sigma)+\mathcal{O}\bigl(\varpi^{-C}R^{-N/2}\bigr)
	\end{align*}
	holds for all  $\varpi\in (0,2^{-C_0}]$ and all $Q$.
\end{proposition}
\begin{proof}
	The argument is the same to \cite{TaoMZ} and we only sketch it. By scaling, it suffices to consider the $m=0$ case.
	Let
	$$
	D=D(\zzz_0,r^\#),\;D'=D\Bigl(\zzz_0,r\bigl(1-\frac{C_0}{2}r^{-\frac{1}{3N}}\bigr)\Bigr),\;
	D''=D(\zzz_0,r)
	$$
	with $\zzz_0=(x_0,t_0)$, $t_0\in \mathsf{LS}(5Q)$, $D\subsetneqq D'_-\subsetneqq D'\subsetneqq D'_+\subsetneqq D''$. 
	
	Write $F^\sigma=P^{\Phi_\sigma}_{D'}F^\sigma+(I-P^{\Phi_\sigma}_{D'})F^\sigma:F_1+F_2$ and by the linearity of $F^\sigma\mapsto \mathcal{ F}^\sigma$, write the corresponding decomposition for each $\mathcal{ F}^{(\varDelta)}$ short for $\mathcal{ F}^{\sigma,\varpi,(\varDelta)}$ as 
	$\mathcal{ F}^{(\varDelta)}_1+\mathcal{ F}^{(\varDelta)}_2$.
	The $\mathcal{ F}^{(\varDelta)}_1$ part follows from 
	Lemma \ref{lem:bessel-w-t}	and \eqref{2.18}.
   The contribution of the $\mathcal{ F}_2^{(\varDelta)}$
   part is bounded by the error term. Indeed, the estimate follows from using the condition $r\ge R^{\frac{1}{2}+\frac{1}{N}}$ 
	by distinguishing the following two cases for the wave packets $F^\sigma_T$ for $T\in\TT_\sigma$ in the definition of $\mathcal{ F}^{(\varDelta)}$
	\begin{itemize}
		\item Case A. $\mathsf{dist}(T,D)\ge R^{\frac{1}{2}+\frac{1}{100 N}}$,
		\item Case B. $\mathsf{dist}(T,D)\le R^{\frac{1}{2}+\frac{1}{100 N}}$,
	\end{itemize}
	For Case A, we use \eqref{eq:rapdcay}\eqref{eq:strongcon} to see the contributions of $\qq$ that do intersect with $D$ is bounded by the error term.
	For Case B, we use \eqref{2.17} to get the result. We refer to \cite[P.30]{Y22} for more details.

\end{proof}

\subsection{Wrap-up}
To state and demonstrate the main result of this section, let us introduce more notations.\medskip

Given  a wave  $F^\sigma$ and a cube $Q=Q_R$ of side-length $R\ge C_0 2^k$, assume that $k\in C_0\Z$ large and that 
we have defined $\mathfrak{F}_j$ for each $j\in C_0\Z\cap [0,k]$, inductively by  $\mathfrak{F}_{j+C_0}=\mathfrak{F}_{C_0}\circ\mathfrak{F}_j$ so that 
$\mathcal{ F}^\sigma_j=\mathfrak{F}_j(F^\sigma)$
is a wave table of $F^\sigma$ over $Q$ with depth $j$. For each $j$, we have obtained inductively a chain of wave tables
$$
\mathcal{ F}^\sigma_0\prec \mathcal{ F}^\sigma_{C_0}\prec\cdots
\prec \mathcal{ F}^\sigma_{j'}\prec\cdots\prec\mathcal{ F}^\sigma_j.
$$
For any $j'<j$, we say $\mathcal{ F}^\sigma_j$ is a \emph{relative} wave table of $\mathcal{ F}^\sigma_{j'}$
with depth $j-j'$. For each entry $\mathcal{ F}^{\sigma,(\qq)}$
of $\mathcal{ F}^{\sigma}_j$ with $\qq \in\mathcal{ Q}_j(Q)$,
there exists a unique $\qq'\in \mathcal{ Q}_{j'}(Q)$ such that
$\qq\in\mathcal{ Q}_{j-j'}(\qq')$ and $\mathcal{ F}^{\sigma,(\qq)}=\bigl(\mathcal{ F}^{\sigma,(\qq')}\bigr)^{(\qq)}$, obtained as an entry of  $\mathfrak{F}_{j-j'}(\mathcal{ F}^{\sigma,(\qq')})$.
\medskip

For each $j$ and $0\le j'<j$,
define the $j'-$\emph{quilt} of $\mathcal{ F}^\sigma:=\mathcal{F}^\sigma_j$  on $Q$ as
$$
\bigl[\mathcal{F}^\sigma\bigr]_{j'}=\sum_{\qq'\in\mathcal{Q}_{j'}(Q)}\indic_{\qq'}\, \bigl|\mathcal{F}^{\sigma,\,(\qq')}\bigr|\,,
$$
where $\indic_{\qq'}$ is the indicator function of $\qq'$
and $\mathcal{ F}^{\sigma,(\qq')}$ is the  entry of $\mathcal{ F}^\sigma_{j'}$. 
Then
\begin{equation}
\label{eq:tao18}
\max_{\qq\in\mathcal{ Q}_j(Q)}\indic_{\qq}\,\bigl|\mathcal{ F}^{\sigma,(\qq)}\bigr|\le \bigl[\mathcal{ F}^\sigma\bigr]_j
\le\cdots\le
\bigl[\mathcal{ F}^\sigma\bigr]_{j'}\le \cdots \le
\bigl[\mathcal{ F}^\sigma\bigr]_{0}=|F^\sigma|\indic_{Q}.
\end{equation}
Let 
$$
X^{\varpi,k}(Q)=\bigcap_{j=C_0}^{k} \mathfrak{I}^{\varpi2^{-(k-j)/N},j}(Q)\,.
$$
Recall that for any  $\Phi_\sigma\in\mathscr{E}_\sigma$
and $\zzz_0=(x_0,t_0)\in \R^{n+1}$, 
$$\mathbf{\Lambda}_{\mathsf{purple}}^{\Phi_\sigma}(\zzz_0,r)=\mathbf{ \Lambda}_1^{\Phi_\sigma}(\zzz_0,r)\cup \mathbf{ \Lambda}_2^{\Phi_\sigma}(\zzz_0,r),$$
with $\mathbf{ \Lambda}_i^{\Phi_\sigma}(\zzz_0,r)$  given in Lemma \ref{lem:opposite}  for $i=1,2$.
Define
$$\mathcal{X}_{\Phi_\sigma,\zzz_0}^{\varpi,k,r}(Q)=X^{\varpi,k}(Q)\cap \,\mathbf{\Lambda}^{\Phi_\sigma}_{\mathsf{purple}}(\zzz_0,r)\,.$$

For any $u\in C^\infty(\R^{n}_{x}\times\R_t)$ and any measurable subset $\Omega\subset\R^{n+1}$, such that $\Omega=\bigcup_{t\in \mathcal{I}}\bigl(\varPi_{t}\times \{t\}\bigr)$ for some $\mathcal{I}\subset\R$, we let
$$
\|u\|_{Z(\Omega)}:=\Bigl(\int_{\mathcal{I}}\Bigl(\int_{\varPi_{t}}|u(x,t)|^{s}dx\Bigr)^\frac{q}{s}dt\Bigr)^{\frac{1}{q}},
$$
with  $(q,s)=\bigl(q_c^+,r_c^-\bigr)\in\mathbf{\Gamma}$,
where for any $\gamma\in\R$, we denote $\gamma^+$ ( resp. $\gamma^-$) as a real number greater (resp. smaller) than  but  sufficiently close to $\gamma$. Thus, the sense of $Z\bigl(X^{\varpi,k}(Q)\bigr)$
and $Z\bigl(\mathcal{X}_{\Phi_\sigma,\zzz_0}^{\varpi,k,r}(Q)\bigr)$ is clearly. The assumption on $k\in C_0\Z$ is only for a technical reason and can be removed by scaling.
\medskip

Based on the above preparations, we  arrive at 
\begin{prop}
	\label{pp:4.1-tao}
	There exists a constant $C>0$ depending only on $n$ such that the following statement holds.

	Let $R\ge C_0 2^k$ and $\varpi\in (0,2^{-C_0}]$.
	Let $F^\sigma$ and $G^\sigma$ be red and blue waves
	associated to $\Phi_\sigma\in\mathscr{E}_\sigma$ for some $\sigma\in (0,\sigma_0]$
	with frequency $1$ and $2^k$ respectively,
	which obey the energy normalization 
	$$
	\EEE(F^\sigma)=\EEE(G^\sigma)=1
	$$
	and the relaxed margin condition 
	\begin{equation}
	\label{eq:re-m}
\min\Bigl\{	\mathsf{marg}(F^\sigma),\mathsf{marg}(G^\sigma)\Bigr\}\ge {100}^{-1}-2(2^k R^{-1})^{\frac{1}{N}}.
	\end{equation}
	For any spacetime cube $Q$ of sidelength $CR$, there 
	exists a red wave table $\mathcal{F}^\sigma$
	on $Q$ with depth $k$ and frequency one, and 
	a blue wave table $\mathcal{G}^\sigma$ on $Q$ with depth $C_0$ and frequency $2^k$, both depending on $\varpi$
	such that
	\begin{itemize}
		\item[-] The margin condition holds
		\begin{equation}
		\label{eq:m-f-wt}
		\min\Bigl\{\mathsf{marg}(\mathcal{F}^\sigma),\,
		\mathsf{marg}(\mathcal{G}^\sigma)\Bigl\}
		\ge 100^{-1}-(2^{k+C_0}/R)^{-\frac{1}{N}}\,,
		\end{equation}
		\item[-] The energy estimate holds
		\begin{equation}
		\label{eq:en-est}
		\max\Bigl\{
		\EEE(\mathcal{F}^\sigma),\EEE(\mathcal{G}^\sigma)\Bigr\}\le 1+C\varpi\,,
		\end{equation}
		\item[-] Effective approximation via $(k,C_0)-$quilts product
		\begin{equation}
		\label{eq:appkC_0}
	\quad\;\;\,	\|F^\sigma G^\sigma\|_{Z(X^{\varpi,k}(Q))}\le
		\bigl\|\bigl[\mathcal{F}^\sigma\bigr]_k\bigl[\mathcal{G}^\sigma\bigr]_{C_0}\bigr\|_{Z(X^{\varpi,k}(Q))}+\varpi^{-C}2^{k\gamma(q)}\,,
		\end{equation}
		\item[-] Refined approximation on conic regions
		\begin{multline}
		\label{eq:refinedapp}
		\quad\quad\quad	\|F^\sigma G^\sigma\|_{Z(\mathcal{X}^{\varpi,k,r}_{\Phi_\sigma,\zzz_0}(Q))}\le
		\bigl\|\bigl[\mathcal{F}^\sigma\bigr]_k\bigl[\mathcal{G}^\sigma\bigr]_{C_0}\bigr\|_{Z(\mathcal{X}^{\varpi,k,r}_{\Phi_\sigma,\zzz_0}(Q))}\\
		+\varpi^{-C}2^{k\gamma(q)}\Bigl(1+\frac{R}{2^kr}\Bigr)^{-\frac{1}{N}}\,
		\end{multline}
		for all $\zzz_0\in\R^{n+1}$ and $r>0$,
		\item[-] Persistence of the non-concentration of energy
		\begin{equation}
		\label{eq:persist}
		\EEE_{r^\#,5Q} (\mathcal{F}^\sigma,\mathcal{G}^\sigma)\le
		\EEE_{r,5Q}(F^\sigma,G^\sigma)+C\varpi+\varpi^{-C}R^{-\frac{N}{10}}\,,
		\end{equation}
		for all $r\ge C R^{\frac{1}{2}+\frac{3}{N}}$, with $r^\#:=r(1-C_0r^{-\frac{1}{3N}})$,
	\end{itemize} 
	where
	 $\gamma(q)=\max\Bigl(\frac{1}{q}-\frac{1}{2},0\Bigr)$ and the energy concentration for $\mathcal{ F}^\sigma,\mathcal{ G}^\sigma$ is defined as 
	 $$
	 \EEE_{r,Q}(\mathcal{ F}^\sigma,\mathcal{ G}^\sigma):=\max_{\varDelta,\varDelta'\in\mathcal{Q}_{C_0}(Q)}\EEE_{r,Q}\bigl(\mathcal{ F}^{(\varDelta)},\mathcal{ G}^{(\varDelta')}\bigr)
	 $$
\end{prop}
\begin{proof}
	The argument is the same to \cite{TaoMZ} and we only sketch it below.
	We first construct
	 the red wave table with respect to the blue wave by letting $\mathfrak{F}_0(F^\sigma)$ be the restriction of $F^\sigma$ on $Q$, and inductively $\mathfrak{F}_{j+C_0}:=\mathfrak{F}_{C_0}^{\sigma,\varpi 2^{-(k-j)/{N}}}\circ\mathfrak{F}_j$ for every integer $j\in[0,k]$ being a multiple of $C_0$. 
	 Reverse the colour of the waves, we construct the blue wave table for $G^\sigma$ with respect to the red wave table. This is done by letting $\mathcal{ G}^\sigma$ be the table of entries given by $\mathfrak{S}^{\sigma,\varpi}_{C_0}(G^\sigma,\mathcal{ F}^{\sigma,(\qq)};Q)$ for each $\qq\in \mathcal{ Q}_k(Q)$. Telescoping the no waste bilinear $L^2$ estimate and using the same interpolation argument with the energy estimate and also Lemma \ref{lem:opposite} as in \cite{Y22} to get the approximation as well as its  refinement via $(k,C_0)-$quilts in the $Z-$norms.
	 Finally, the persistence of the non-concentration  of energy is readily deduced from Proposition \ref{pp:persist}.

\end{proof}

\begin{remark}
	For Proposition \ref{pp:sigma-1}, where the surface is the lightcone without using the Lee-Vargas rescaling in \cite{LeeVargas},
	the above construction of wave tables is readily obtained by modifying Proposition 4.1 of \cite{TaoMZ} in the step of using interpolations. 
\end{remark}

\section{Proof of Theorem \ref{pp:sigma-0}}
\label{sec:pf-sigma-0}
For any $\sigma\in(0,\sigma_0]$ and $R\ge C_02^k$, we denote
$$
\mathfrak{R}^\sigma_R\times \mathfrak{B}^{\sigma,k}_R:=\bigcup_{\Phi_\sigma\in\mathscr{E}_\sigma}\Bigl( \mathfrak{R}^{\Phi_\sigma,1}_R\times\mathfrak{B}^{\Phi_\sigma,2^k}_R\Bigr)\,.
$$
\begin{definition}
	\label{def-A}
	Fix $\sigma\in(0,\sigma_0]$.
	For any $R\ge C_02^k$, fix $Q_R\subset\R^{n+1}$ a  spacetime cube of side-length $R$. Let $A^\sigma(R)$ be the optimal constant $C$ such that \begin{equation}
	\label{eq:bilinar-wolff}
	\| F^\sigma G^\sigma\|_{Z(Q_R)}\le C	\,\EEE(F^\sigma)^{1/2}\EEE(G^\sigma)^{1/2},
	\end{equation}
	holds for all  $(F^\sigma, G^\sigma)\in \mathfrak{R}^\sigma_R\times \mathfrak{B}^{\sigma,2^k}_R$ and all $Q_R$\,.
\end{definition}
By the invariance of translation   in the physical spacetime and  modulating the frequency variables, $A^\sigma(R)$ is independent of the center  of $Q_R$. \smallskip

 Since we have
$$
\sup_{0<\sigma\le \sigma_0} A^\sigma(C_02^k)\lesssim 2^{k\bigl(\frac{1}{q}-\frac{1}{2}\bigr)},
$$
by using the same argument in \cite{LeeVargas},
We shall only consider $R\ge C_02^k$.\medskip

To show Theorem \ref{pp:sigma-0}, we shall prove that there is a fixed constant $C_*$ depending only on $n,\eps, \sigma_0$, $\varSigma_1, \varSigma_2$ and the $(q,s)$ exponent in $Z-$norm  taken sufficiently close to the critical index $(q_c,r_c)$, such that 
$$\sup_{0<\sigma\le \sigma_0}A^\sigma(R)\le C_* 2^{k\gamma(q,\eps)}\,,$$ holds
with $\gamma(q,\eps)=\frac{1}{q}-\frac{1}{2}+\eps$ for all $R\ge C_02^k$. 
We set  $$\displaystyle\overline{A}^{\sigma}(R)=\sup_{\sigma\le\sigma'\le\sigma_0}\;\sup_{{C_0}2^k\le R'\le R}A^{\sigma'}(R'),$$  
for dealing with the margin condition.
We may assume $\overline{A}^\sigma(R)\ge 2^{k\gamma(k,\eps)}$.
\smallskip

We consider  the small scale case $R\le 2^{C_0k}$ first.
In this case,
by using the non-endpoint argument leading to (41) in \cite[P. 1307]{LeeVargas}, we have 
$$
\overline{A}^\sigma(R)\lesssim_\eps\, 2^{k\gamma(q,\eps)} \,,\quad \forall\,\eps>0.
$$ 

Thus, it suffices to consider the large scale case $R\ge 2^{C_0k}$. To this end, we need an auxiliary quantity linked to the energy concentration as in \cite{TaoMZ}.
Here, slightly different from \cite{TaoMZ} where 
the proof deals with a fixed lightcone, we are working with a family of conic surfaces.  However, we find that it suffices to study the 
auxiliary quantity for every fixed conic surface. 

\begin{definition}
	\label{def:AAA}
	For any $\sigma\in(0,\sigma_0]$, let 
	$\Phi_\sigma\in\mathscr{E}_\sigma$, $R\ge 2^k{C_0}$ and  $r,r'>0$. We define
	$\mathscr{A}^{\Phi_\sigma}(R,r,r')$ to be the optimal constant $C$ such that
	$$
	\bigl\|F^\sigma G^\sigma\bigr\|_{Z(Q_R\cap\, \mathbf{\Lambda}^{\Phi_\sigma}_{\mathsf{purple}}(\zzz_0,r'))}
	\le C \bigl(\EEE(F^\sigma)\EEE(G^\sigma)\bigr)^{\frac{1}{2q}}\; \EEE_{r,C_0Q_R}(F^\sigma,G^\sigma)^{\frac{1}{q'}},
	$$
	holds for all  $(F^\sigma,G^\sigma)\in\mathfrak{R}^{\Phi_\sigma,1}_R\times\mathfrak{B}^{\Phi_\sigma,2^k}_R$  and all  $Q_R$ and all $\zzz_0=(x_0,t_0)\in\R^{n+1}$.
\end{definition}

\subsection{Control of $\mathscr{A}^{\Phi_\sigma}$ by $\overline{A}^\sigma$ in the large scale case $R\ge 2^{C_0k }$}
\begin{proposition}
	\label{pp_2}
	There is a constant $C>0$ depending only on $n$ and $q,s$ in the $Z$-norm, but not explicitly on $C_0$, such that   for any $\sigma\in(0,\sigma_0]$ and any $\Phi_\sigma\in\mathscr{E}_\sigma$, we have 
	\begin{equation}
	\label{eq:pp-2}
	\mathscr{A}^{\Phi_\sigma}({R},r,C_0(r+1))\le (1+C2^{-C_0})\overline{A}^\sigma(R)+2^{CC_0}2^{k\gamma(q)},
	\end{equation}
	for all $R\ge2^{C_0k}$ and all $r\ge R^{\frac{1}{2}+\frac4N}$.
\end{proposition}
The proof is achieved  in three steps.

\subsubsection{Step 1. The non-concentrated case $r \ge C_0R$}

Recall an orthogonality lemma:
\begin{lemma}
	\label{lem:Te}
	Let $F_1,F_2,\ldots,F_J$ be a finite number of functions on $\R^{n+1}$ such that  $\{F_j\}_j\subset L^q_tL^s_x(\R^{n+1})$ and that the supports of these functions are mutually disjoint. Then, we have
	$$
	\Bigl\|\sum_{j=1}^J F_j\Bigr\|_{L^q_tL^s_x}^q\le \sum_{j=1}^J\|F_j\|_{L^q_tL^s_x}^q\;,
	$$
	for all $(q,s)\in\mathbf{\Gamma}$ close to the critical index $(q_c,r_c)$ .
\end{lemma}
Please see Lemma 5.3 of \cite{Temur} for the proof.

\begin{proposition}
	\label{pp:the noncon}
	There is a constant $C>0$ such that for any $R\ge2^{C_0k}$, we have for all $r \ge C_0R$ and $r'>0$
	$$\mathscr{A}^{\Phi_\sigma}(R,r, r')\le (1+C \varpi)
	\overline{A}^{\sigma}(R)+\varpi^{-C}2^{k\gamma(q)},
	$$
	for all $0<\varpi\le 2^{-C_0}$ and all $\Phi_\sigma\in\mathscr{E}_\sigma$ with $\sigma\in(0,\sigma_0]$.
\end{proposition}
\begin{proof}
	For any $\sigma\in(0,\sigma_0]$ and $\Phi_\sigma\in\mathscr{E}_\sigma$.
	Let $F^\sigma\in\mathfrak{R}^{\Phi_\sigma,1}_R,G^\sigma\in\mathfrak{B}^{\Phi_\sigma,2^k}_R$ be red and blue waves  with normalized energy.
	For any $Q=Q_R$, let $\zzz_Q=(x_Q,t_Q)$ be the center of $Q$. Let $D=D(\zzz_Q,r/2)$ and write
	$$
	F^\sigma=P_D^{\Phi_\sigma} F^\sigma+(1-P^{\Phi_\sigma}_D)F^\sigma,\;
	G^\sigma=P_D^{\Phi_\sigma} G^\sigma+(1-P_D^{\Phi_\sigma})G^\sigma.$$
	Using Lemma \ref{lem:conctr}, we have
	$$
	\max\Bigl\{
	\|\bigl((1-P_D^{\Phi_\sigma})F^\sigma\bigr)\; G^\sigma\|_{Z(Q_R)},\;\|(P_D^{\Phi_\sigma} F^\sigma)(1-P_D^{\Phi_\sigma})G^\sigma\|_{Z(Q_R)}\Bigr\}\le \varpi^{-C}2^{k\gamma(q)}.
	$$
	We are reduced to
	\begin{equation}
	\label{UHU}
		\|(P^{\Phi_\sigma}_D F^\sigma )(P_D^{\Phi_\sigma} G^\sigma)\|_{Z(Q)}\le (1+C\varpi)\overline{A}^{\sigma}(R)\, \EEE_{r,C_0Q}(F^\sigma,G^\sigma)^{1/q'}+\varpi^{-C}2^{k\gamma(q)}.
	\end{equation}
	To see this is the case, let $\mathcal{F}^\sigma_D$ and $\mathcal{G}^\sigma_D$ be the wave tables  for the red and blue waves $P_D^{\Phi_\sigma}F^\sigma$ and $P_D^{\Phi_\sigma} G^\sigma$ on an appropriately enlarged\footnote{The enlargement is for the use of the averaging Lemma 4.2 in \cite{LeeVargas}, extension of  Lemma 6.1 in \cite{TaoMZ} to the mixed-norms.} cube $Q^*$ containing $Q$. Applying Proposition \ref{pp:C-0 quilt}, we have
	$$
	\|(P_D^{\Phi_\sigma} F^\sigma) (P_D^{\Phi_\sigma} G^\sigma)\|_{Z(Q_R)}\le (1+C\varpi)
	\bigl\|[\mathcal{F}_D^\sigma]_{k}[\mathcal{G}_D^\sigma]_{C_0}\bigr\|_{Z(X^{\varpi,k}(Q^*))}+\varpi^{-C}2^{k\gamma(q)}.
	$$
	Applying Lemma \ref{lem:Te} and the definition of $\overline{A}^\sigma(R)$, we get
	\begin{multline}
	\bigl\|[\mathcal{F}_D^\sigma]_{k}[\mathcal{G}_D^\sigma]_{C_0}\bigr\|_{Z(X^{\varpi,k}(Q^*))}\le\Bigl(\sum_{\varDelta\in \mathcal{Q}_{C_0}(Q^*)}
	\bigl\|\mathcal{F}_{D}^{\sigma,\varDelta}\,\mathcal{G}_{D}^{\sigma,\varDelta}\bigr\|^q_{Z(\varDelta)}\Bigr)^{1/q}\\
	\le \overline{A}^\sigma(2^{-C_0}R)\Bigl(\sum_{\triangle\in\mathcal{Q}_{C_0}(Q^*)}
	\EEE(\mathcal{F}_{D}^{\sigma,\varDelta})^{q/2} \EEE(\mathcal{G}_{D}^{\sigma,\varDelta})^{q/2}\Bigr)^{1/q},
	\end{multline}
	where we have used $\mathcal{F}^{\sigma,\varDelta}_D\in\mathfrak{R}^{\sigma}_{2^{-C_0}R},\, \mathcal{G}^{\sigma,\varDelta}_D\in\mathfrak{B}^{\sigma,k}_{2^{-C_0}R}.$
	Using Cauchy-Schwarz,
	$\EEE(\mathcal{F}_{D}^\sigma), \EEE(\mathcal{G}_D^\sigma)\le 1+C\varpi$ by  Lemma \ref{lem:bessel-w-t}, and \eqref{2.18} of Lemma \ref{lem:Tao10.2},
	we get \eqref{UHU}.
\end{proof}
\subsubsection{Step 2.  The energy concentrated case $R^{\frac{1}{2}+\frac{3}{N}}\le r\le C_0 R$}
\begin{proposition}
	\label{pp:the-con}
	There is $C>0$ and $\theta>0$ such that for any $R\ge 2^{C_0k}$, we have for all $r,r'>0$ with $R^{\frac{1}{2}+\frac{3}{N}}\le r\le C_0 R$.
	$$
	\mathscr{A}^{\Phi_\sigma}(R,r,r')\le (1+C\varpi) \mathscr{A}^{\Phi_\sigma}(R/C_0,r^\#,r')+\varpi^{-C}\Bigl(1+\frac{R}{2^kr'}\Bigr)^{-\theta}2^{k\gamma(q)}
	$$
	with $r^\#=r(1-C_0r^{-\frac{1}{3N}})$
	holds	for all $0<\varpi\le 2^{-C_0}$ and $\Phi_\sigma\in\mathscr{E}_\sigma$ with $\sigma\in (0,\sigma_0]$.
\end{proposition}
\begin{proof}
	Let $F^\sigma\in\mathfrak{R}^{\Phi_\sigma,1}_R,
	G^\sigma\in \mathfrak{B}^{\Phi_\sigma,2^k}_R$ be red and blue $\Phi_\sigma-$waves  with normalized energyb $\EEE(F^\sigma)=\EEE(G^\sigma)=1$. For any $Q=Q_R$, by using Proposition \ref{pp:C-0 quilt}, we have for all $\zzz_0$
	\begin{multline*}
	\bigl\|F^\sigma G^\sigma\bigr\|_{Z\bigl(\mathcal{X}^{\varpi,k,r}_{\Phi_\sigma,\zzz_0}(Q)\bigr)}\\
	\le (1+C\varpi)\Bigl\|[\mathcal{F}^{\sigma}]_{k}[\mathcal{G}^{\sigma}]_{C_0}\Bigr\|_{Z(\mathcal{ X}^{\varpi,k,r}_{\Phi_\sigma,\zzz_0}(Q^*)}
	+\varpi^{-C}\Bigl(1+\frac{R}{2^kr'}\Bigr)^{-\frac{1}{N}}2^{k\gamma(q)}.
	\end{multline*}
	We are reduced to showing
	\begin{multline}
	\label{OIU}
	\Bigl\|[\mathcal{F}^{\sigma}]_{k}[\mathcal{G}^{\sigma}]_{C_0}\Bigr\|_{Z(\mathcal{ X}^{\varpi,k,r}_{\Phi_\sigma,\zzz_0}(Q^*))}\\ \le (1+C\varpi) \mathscr{A}^{\Phi_\sigma}(R/C_0,r^\#,r')\;
	\EEE_{r,C_0Q}(F^\sigma,G^\sigma)^{1/q'}+\varpi^{-C} R^{-N/2} 2^{k\gamma(q)}.
	\end{multline}
	Using the definition of $\mathscr{A}^{\Phi_\sigma}(R,r,r')$, we have for all $\varDelta$
	\begin{multline*}
	\bigl\|\mathcal{F}^{\sigma,\varDelta}\mathcal{G}^{\sigma,\varDelta}\bigr\|_{Z(\mathcal{ X}^{\varpi,k,r}_{\Phi_\sigma,\zzz_0}(\varDelta))}\\
	\le \mathscr{A}^{\Phi_\sigma}(R/C_0, r^\#,r')\,\EEE_{r^\#,C_0\varDelta}(\mathcal{F}^{\sigma,\varDelta},\mathcal{G}^{\sigma,\varDelta})^{1/q'}(\EEE(\mathcal{F}^{\sigma,\varDelta})\EEE(\mathcal{G}^{\sigma,\varDelta}))^{1/(2q)}.
	\end{multline*}
	By  Lemma \ref{lem:Te}, Proposition \ref{pp:persist} with $2^{-C_0}CR\ll C_0^{-1}R$ so that $C_0\varDelta\subset 5Q$, we obtain by  Cauchy-Schwarz's inequality
	\begin{multline}
	\label{OIKU}
	\Bigl\|[\mathcal{F}^{\sigma}]_{k}[\mathcal{G}^{\sigma}]_{C_0}\Bigr\|^q_{Z(\mathcal{X}^{\varpi,k,r}_{\Phi_\sigma,\zzz_0}(Q^*))}
	\le
		\sum_{\varDelta\in\mathcal{Q}_{C_0}(Q^*)}
	\bigl\|\mathcal{F}^{\sigma,\varDelta}\mathcal{G}^{\sigma,\varDelta}\bigr\|^q_{Z(\mathcal{X}_{\Phi_\sigma,\zzz_0}^{\varpi,k,r}(Q^*))}\\ \le (1+C\varpi) \mathscr{A}^{\Phi_\sigma}(R/C_0,r^\#,r')^q\,\,
	\EEE_{r,C_0Q}(F^\sigma,G^\sigma)^{q/q'}+\varpi^{-C} R^{-qN/2}2^{k\gamma(q)},
	\end{multline}
	and \eqref{OIU} follows by adjusting the constant $C$.
\end{proof}
\subsubsection{Step 3. Proof of Proposition \ref{pp_2}}
With Proposition \ref{pp:the noncon} and \ref{pp:the-con}, we may complete the proof of  Proposition \ref{pp_2} by using the standard iteration argument. Please see \cite[P.225]{TaoMZ}
for details.
\qed

\subsection{Essential concentration along conic regions}\label{sect:pf-thm}
Fix a pair of red and blue waves $(\mathscr{F}^\sigma,\mathscr{G}^\sigma)\in\mathfrak{R}^\sigma_R\times\mathfrak{B}^{\sigma,2^k}_R$ with $\EEE(\mathscr{F}^\sigma)=\EEE(\mathscr{G}^\sigma)=1$. 
The complete proof relies crucially on  the Kakeya compression property below.

Before stating this result, let us start with a non-endpoint bilinear estimate.
For $0<r_1<r_2<+\infty$,  we define the cubical annulus as
$$
Q^{\mathsf{ann}}(x_Q,t_Q;r_1,r_2)=Q(x_Q,t_Q;r_2)\setminus Q(x_Q,t_Q;r_1).
$$
The following  non-endpoint bilinear estimate holds for localized blue or red waves, on a dyadic annulus corresponds to Lemma 11.1 of Tao \cite{TaoMZ}, and we omit the proof since it is rather standard using the wave-packets and a multiplicity estimate on the overlappedness of tubes. See also \cite{Y22} for more details.
\begin{lemma}
	\label{lem:non-edpt}
	Let $R\ge  2^{C_0k}$,
	$C_02^k\le r\le R^{\frac{1}{2}+\frac{4}{N}}$, and
	$D=D(\zzz_D,\,C_0^{1/2}r)$ with $\zzz_D=(x_D,t_D)$. Then, there exists $b>0$, depending only on $\sigma_0,\,Z$ and $n$, such that for any  red and blue $\Phi_\sigma-$waves $F^\sigma,G^\sigma$ of frequency $1$ and $2^k$ respectively, with  $\EEE(F^\sigma)=\EEE(G^\sigma)=1$, we have
	\begin{equation}
	\|(P_D^{\Phi_\sigma} F^\sigma)G^\sigma\|_{Z(Q^\mathsf{ann}(\zzz_D; R,2R))},	\| F^\sigma\, (P_D^{\Phi_\sigma}G^\sigma)\|_{Z(Q^\mathsf{ann}(\zzz_D; R,2R))}\lesssim R^{-b},
	\end{equation}
	where the implicit constants are uniform with respect to $\sigma\in (0,\sigma_0]$
\end{lemma}

The main result of this subsection reads
\begin{proposition}
	\label{pp_1}
	Let $\Phi_\sigma\in\mathscr{E}_\sigma$ and $(\mathscr{F}^\sigma, \mathscr{G}^\sigma)\in\mathfrak{R}^{\Phi_\sigma,1}_R\times \mathfrak{B}^{\Phi_\sigma,2^k}_R$  be the pair of red and blue waves fixed at the beginning of this section with $\EEE(\mathscr{F}^\sigma)=\EEE(\mathscr{G}^\sigma)=1$.
	There exists a constant $C>0$ depending only on $n,\eps,\sigma_0$ and $q,s,$ in the $Z-$norm such that for any $R\ge 2^{C_0k}$ and $\delta\in(0,1/2)$,
	if $Q_R$ satisfies
	\begin{equation}
	\label{eq:crit}
	\|\mathscr{F}^\sigma\,\mathscr{G}^\sigma\|_{Z(Q_R)}\ge
	\frac{1}{2} \,\overline{A}^\sigma(R),
	\end{equation}
	and we  let $r_\delta$ be the supremum of all radii $r\ge {C_0}2^k$ such that
	\begin{equation}
	\label{eq:en-cr}
	\EEE_{r,C_0Q_R}(\mathscr{F}^\sigma,\mathscr{G}^\sigma)\le 1-\delta
	\end{equation}
	holds and  let $r_\delta={C_0}2^k$ if no such radius exists,
	then  there exists a cube $\widetilde{Q}_{\overline{R}_\delta}$ 	of size  $\overline{R}_\delta\in[ 2^{C_0k},R]$ and $\zzz_\delta\in\R^{n+1}$  such that $\overline{R}_\delta^{\frac{1}{2}+\frac{4}{N}}\le r_\delta$ when $r_\delta\ge 2^{C_0k}$, for which we have
	\begin{equation}
	\label{eq:pp-1}
	\|\mathscr{F}^\sigma \mathscr{G}^\sigma\|_{Z(Q_R)}\le (1-C(\delta+C_0^{-C})^{q})^{-2/q}\|\mathscr{F}^\sigma \mathscr{G}^\sigma\|_{Z(\Omega_{\delta}^{\Phi_\sigma})}+\;2^{CC_0}2^{k\gamma(p,\eps)}\,,
	\end{equation}
	where
	$\Omega_{\delta}^{\Phi_\sigma}:=\widetilde{Q}_{\overline{R}_\delta}\cap \mathbf{ \Lambda}^{\Phi_\sigma}_\delta$ with  $\mathbf{\Lambda}^{\Phi_\sigma}_\delta=\mathbf{\Lambda}^{\Phi_\sigma}_{\mathsf{purple}}(\zzz_\delta,\,C_0(r_\delta+1))$.
\end{proposition}

\subsubsection{ The medium or low concentration case: $r_\delta\ge R^{1/2+4/N}$}
In this case, we show there is a constant $C$ such that for some $\zzz_\delta\in\R^{n+1}$, we have
\begin{equation}
\label{eq:low-med}
\bigl\|\mathscr{F}^\sigma \mathscr{G}^\sigma\bigr\|_{Z(Q_R)}
\le \bigl(1-C(\delta+C_0^{-C})^q\bigr)^{-1/q}\bigl\|\mathscr{F}^\sigma \mathscr{G}^\sigma\bigr\|_{Z(Q_R\cap \mathbf{ \Lambda}^{\Phi_\sigma}_\delta)}
\end{equation}
holds with $C$ independent of $Q_R$ and $\zzz_\delta$. In this case,  $\overline{R}_\delta=R$ and $\widetilde{Q}_{\overline{R}_\delta}=Q_R$.\medskip

By definition, there is $D_\delta=D(\zzz_\delta,r_\delta)$ with  $\zzz_\delta=(x_0,t_0)$ and $t_0\in\mathsf{LS}(C_0Q_R)$, such that we have
\begin{equation}
\label{eq:sup}
\min\Bigl(\|\mathscr{F}^\sigma\|_{L^2(D_\delta)}^2, \|\mathscr{G}^\sigma\|_{L^2(D_\delta)}^2\Bigr)\ge 1-2\delta\,.
\end{equation}
Let $D^\natural=C_0^{1/2}D_\delta=D(\zzz_\delta,\,C_0^{1/2}r_\delta)$ and write
$$
\mathscr{F}^\sigma=P^{\Phi_\sigma}_{D^\natural}\mathscr{F}^\sigma+(1-P_{D^\natural}^{\Phi_\sigma})\mathscr{F}^\sigma,\;\;
\mathscr{G}^\sigma=P_{D^\natural}^{\Phi_\sigma}\mathscr{G}^\sigma+(1-P_{D^\natural}^{\Phi_\sigma})\mathscr{G}^\sigma.
$$
By  Lemma \ref{lem:Te} and the  condition $2^{k\gamma(q,\eps)}\le \overline{A}^\sigma(R)\le 2\,\|\mathscr{F}^\sigma \mathscr{G}^\sigma\|_{Z(Q_R)}$, it suffices to show for some universal constant $C>0$, we have
\begin{equation}
\label{eq:extr-1}
\|(P_{D^\natural}^{\Phi_\sigma}\mathscr{F}^\sigma ) \,\mathscr{G}^\sigma\|_{Z(Q_R\setminus\mathbf{ \Lambda}^{\Phi_\sigma}_\delta)}\lesssim C_0^{-C},
\end{equation}
\begin{equation}
\label{eq:extr-2}
\|(1-P_{D^\natural}^{\Phi_\sigma})\mathscr{F}^\sigma\, P_{D^\natural}^{\Phi_\sigma}\mathscr{G}^\sigma\|_{Z(Q_R\setminus \mathbf{ \Lambda}^{\Phi_\sigma}_\delta)}\lesssim C_0^{-C},
\end{equation}
and
\begin{equation}
\label{eq:extr-3}
\|(1-P^{\Phi_\sigma}_{D^\natural})\mathscr{F}^\sigma \, (1-P^{\Phi_\sigma}_{D^\natural})\mathscr{G}^\sigma\|_{Z(Q_R)}\lesssim (\delta+C_0^{-C})\;\;\overline{A}^\sigma(R).
\end{equation}

To see  \eqref{eq:extr-1} and \eqref{eq:extr-2}, by using energy estimates, we are reduced to
\begin{equation}
\label{eq:PP-huygens}
\|P_{D^\natural}^{\Phi_\sigma}\mathscr{F}^\sigma\|_{L^\infty(Q_R\setminus \mathbf{ \Lambda}^{\Phi_\sigma}_{\delta})}\lesssim_N R^{-N/2},\, \|P_{D^\natural}^{\Phi_\sigma}\mathscr{G}^\sigma\|_{L^\infty(Q_R\setminus \mathbf{ \Lambda}^{\Phi_\sigma}_{\delta})}\lesssim_N R^{-N/2},
\end{equation}
which is obvious in view of Lemma \ref{lem:conctr} and  $r_\delta\ge R^{1/2+4/N}$.

To show \eqref{eq:extr-3}, by the induction argument and  \eqref{eq:en-cr}, \eqref{2.19}, \eqref{2.20} as well as the assumption on $r_\delta$, we have
$$
\EEE\bigl((1-P_{D^\natural}^{\Phi_\sigma})\mathscr{F}^\sigma\bigr)\lesssim \delta+R^{-N/2},\; \;\EEE\bigl((1-P^{\Phi_\sigma}_{D^\natural})\mathscr{G}^\sigma\bigr)\lesssim \delta+R^{-N/2}.
$$
It is clear that $$(1-P^{\Phi_\sigma}_{D^\natural})\mathscr{F}^\sigma\in\mathfrak{R}^{\Phi_\sigma,1}_{R'},\quad (1-P^{\Phi_\sigma}_{D^\natural})\mathscr{G}^\sigma\in\mathfrak{B}^{\Phi_\sigma,2^k}_{R'},$$ with $R'=\frac{R}{(1+o(1))}$. This yields \eqref{eq:extr-3} by finitely partitioning $Q_R$ and using the definition of $A^{\sigma}(R')$ and the monotonicity of $\overline{A}^\sigma(R)$.
The proof is complete.

\subsubsection{The high concentration case: $r_\delta\le  R^{1/2+4/N}$}
We turn to the case where the blue and red waves are highly concentrated.  Define $$\overline{R}_\delta=\max\Bigl(2^{C_0k},\;r_\delta^{1/(1/2+4/N)}\Bigr).$$

Consider the case $\overline{R}_\delta>2^{C_0k}$. In this case, we necessarily have $r_\delta>2^{C_0k/2}$ and there is $\zzz_\delta$ such that we have \eqref{eq:sup}.
Let $\widetilde{Q}=Q^{\zzz_\delta}_{\overline{R}_\delta}$ be the  cube of size $\overline{R}_\delta$ centered at $\zzz_\delta$.
By  splitting $Q_R=\bigl(Q_R\cap \widetilde{Q}\bigr)\cup\bigl(Q_R\setminus \widetilde{Q}\bigr)$ and using Lemma \ref{lem:Te}, we have
$$
\bigl\|\mathscr{F}^\sigma \mathscr{G}^\sigma\bigr\|_{Z(Q_R)}^q
\le \bigl\|\mathscr{F}^\sigma \mathscr{G}^\sigma\bigr\|_{Z(\widetilde{Q})}^q
+\bigl\|\mathscr{F}^\sigma \mathscr{G}^\sigma\bigr\|_{Z(Q_R\setminus \widetilde{Q})}^q.
$$

For the first term on $\widetilde{Q}$, the argument as in the medium or low concentration case leads to an estimate of the form \eqref{eq:low-med} with $Q_R$ there replaced by $\widetilde{Q}$.

For the second term,
write 
$$
\mathscr{F}^\sigma \mathscr{G}^\sigma=\underbrace{\bigl((P^{\Phi_\sigma}_{D^\natural}\mathscr{F}^\sigma)\mathscr{G}^\sigma\bigr)}_{:=I}
+\underbrace{\bigl((1-P^{\Phi_\sigma}_{D^\natural})\mathscr{F}^\sigma \; P^{\Phi_\sigma}_{D^\natural}\mathscr{G}^\sigma\bigr)}_{:=II}+
\underbrace{\bigl((1-P^{\Phi_\sigma}_{D^\natural})\mathscr{F}^\sigma\,(1-P^{\Phi_\sigma}_{D^\natural})\mathscr{G}^\sigma\bigr)}_{:=III}
$$
For $I$ and $II$, dyadic decomposing $Q_R\setminus \widetilde{Q}$
into annuli around $\zzz_\delta$ of the form $Q^{\mathsf{ann}}(\zzz_\delta; 2^{j},2^{j+1})$ with $2^j\gtrsim \overline{R}_\delta$. Taking $C_0$ large and applying Lemma \ref{lem:non-edpt} then summing over dyadic $2^{-jb}$, we are done.

It remains to handle the $III-$term.
Denote
$$
\mathring{\mathscr{F}}^\sigma=(1-P^{\Phi_\sigma}_{D^\natural})\mathscr{F}^\sigma,\quad \mathring{\mathscr{G}}^\sigma=(1-P^{\Phi_\sigma}_{D^\natural})\mathscr{G}^\sigma\,.
$$
Note that $\mathring{\mathscr{F}}^\sigma, \mathring{\mathscr{G}}^\sigma$ are red and blue waves without the relaxed margin conditions required in $\mathfrak{R}^{\Phi_\sigma}_R$ and $\mathfrak{B}^{\Phi_\sigma}_R$. Thus, we can not apply the inductive argument as in the case when $r_\delta\ge R^{1/2+4/N}$. To over come this obstacle, we may use  the same method in  \cite[Appendix B]{Y22}, to which we refer for details.

Switching back and
affording a fixed universal constant, we have
\begin{equation}
\label{eq:ext-en-ind}
\|III\|_{Z(Q_R)}\lesssim\;\delta  \overline{A}^\sigma(R)+2^{O_\eps(C_0)}2^{k\gamma(q,\eps)}
\end{equation}
Thus, by using \eqref{eq:crit} 
$$\|III\|_{Z(Q_R)}\lesssim\;\delta \|\mathscr{F}^\sigma \mathscr{G}^\sigma\|_{Z(Q_R)}+ 2^{O_\eps(C_0)}2^{k\gamma(q,\eps)},$$
Plugging this back we are done.
\medskip

It remains to consider the case $\overline{R}_\delta=2^{2C_0k}$. In this case, the energy is concentrated in a scale $\le 2^{C_0k}$, we use  the small scale estimate at the very beginning of the section to see thatr $\|\mathscr{F}^\sigma\mathscr{G}^\sigma\|_{Z(\widetilde{Q})}\lesssim 2^{k\gamma(q,\eps)}$ with $\eps$ small.\medskip

Collecting all these estimates, we obtain \eqref{eq:pp-1} and the proof of Proposition \ref{pp_1} is complete. \qed

\subsection{End of the proof }

We are ready to show Theorem \ref{pp:sigma-0}. Let $C_1$ and $C_2$ be the structural constants given by Proposition \ref{pp_1} and Proposition  \ref{pp_2} respectively.  We may take $C_1$ large and  take  $\delta=C_0^{-C_1/100}$.

Let $(\mathscr{F}^\sigma, \mathscr{G}^\sigma) \in\mathfrak{R}^\sigma_R\times\mathfrak{B}^{\sigma,2^k}_R$ such that there exists $\Phi_\sigma\in\mathscr{E}_\sigma$ for which we have $\mathscr{F}^\sigma \in\mathfrak{R}^{\Phi_\sigma}_R$ and $\mathscr{G}^\sigma\in\mathfrak{B}^{\Phi_\sigma,2^k}_R$.
For any spacetime cube $Q_R$, if it satisfies the condition \eqref{eq:crit}, we let  $\zzz_\delta, r_\delta,D_\delta$ be given by Proposition \ref{pp_1}. If $\overline{R}_\delta>2^{C_0k}$, then by  definition of $\mathscr{A}^{\Phi_\sigma}(\overline{R}_\delta,r_\delta, C_0r_\delta)$, \eqref{eq:pp-1},  Proposition \ref{pp_2},  \eqref{eq:en-cr}, we get
\begin{multline*}
\|\mathscr{F}^\sigma  \mathscr{G}^\sigma\|_{Z(Q_R)}\\
\le \bigl(1-C_1(\delta+C_0^{-C_1})^{q}\bigr)^{-2/q} (1-\delta)^{1/q'}\Bigl((1+C_22^{-C_0})\overline{A}^{\sigma}(R)+2^{C_2C_0}2^{k\gamma(q,\eps)}\Bigr)\\
+2^{O(C_0)}2^{k\gamma(q,\eps)}.	
\end{multline*}
Using  $q>1$,  and taking $C_0$ large if necessary (depending only on $q, C_1$), one  has  $\exists\;\delta_\mathbf{o}\in (0,1/10)$  and  $0<C_{\mathbf{o}}<\infty$, depending only on $C_0,C_1,C_2$ and $q$, such that
$$
\|\mathscr{F}^\sigma \mathscr{G}^\sigma\|_{Z(Q_R)}\le (1-\delta_\mathbf{o})\,\overline{A}^\sigma(R)+C_\mathbf{o}\,2^{k\gamma(k,\eps)}.
$$

If $\overline{R}_\delta=2^{C_0k}$, then we have
the trivial estimate $\|\mathscr{F}^\sigma\mathscr{G}^\sigma\|_{Z({Q}_R)}\le \, C_\mathbf{o}2^{k\gamma(q,\eps)}$ by the small scale estimate, in view of the proof in the last subsection.

Thus, we have
\begin{align*}
&\max_{Q_R}
\|\mathscr{F}^\sigma \mathscr{G}^\sigma\|_{Z(Q_R)}\\
\le&
\max\Bigl\{\max_{Q_R: \eqref{eq:crit} \text{holds}}
\|\mathscr{F}^\sigma \mathscr{G}^\sigma\|_{Z(Q_R)}, \max_{Q_R: \eqref{eq:crit}\text{ fails}}
\|\mathscr{F}^\sigma \mathscr{G}^\sigma\|_{Z(Q_R)}\Bigr\}+C_\mathbf{o}\, 2^{k\gamma(q,\eps)}\\
\le&\max\Bigl\{(1-\delta_\mathbf{o})\overline{A}^\sigma(R)+{C}_\mathbf{o}2^{k\gamma(k,\eps)},\;\frac{1}{2}\,\overline{A}^\sigma(R)\Bigr\}+C_\mathbf{o}2^{k\gamma(q,\eps)}\\
\le& (1-\delta_\mathbf{o})\overline{A}^\sigma(R)+2{C}_\mathbf{o}\,2^{k\gamma(q,\eps)}\,.
\end{align*}
Noting the right side is independent of $(\mathscr{F}^\sigma, \mathscr{G}^\sigma) \in\mathfrak{R}^\sigma_R\times \mathfrak{B}^{\sigma,2^k}_R$,
we  get
$$
A^\sigma(R)\le (1-\delta_\mathbf{o})\overline{A}^\sigma(R)+2{C}_\mathbf{o} 2^{k\gamma(q,\eps)}.
$$
Note that  $\overline{A}^\sigma(R)$  is finite for each $\sigma,R$.
Taking suprema, we obtain
$$
\overline{A}^\sigma(R)\le \frac{2{C}_\mathbf{o}} {\delta_\mathbf{o}} \,2^{k\gamma(q,\eps)}.
$$
Since  the right side is uniform with respect to $\sigma\le \sigma_0$,
the proof is thus complete by taking suprema over $\sigma\in (0,\sigma_0]$.\qed

\begin{remark}
	To get Proposition \ref{pp:sigma-1}, the above same proof
	works equally well
	based on the the  wave tables in \cite{TaoMZ}, with modifications in the interpolation argument to meet the mixed-norms.
\end{remark}

\section{Endpoint  bilinear restriction estimates on the sphere concerning a  conjecuture by Foschi and Klainerman }\label{sect:furtherrks}

Let $n\ge 2$ and $\mathbb{S}=\mathbb{S}^{n-1}=\{\xi\in\R^n:|\xi|=1\}$
be the unit sphere in $\R^n$. Define the Stein operator to be
the adjoint of the Fourier restriction operator
$$
Sf(x)=\int_{\mathbb{S}}e^{ix\cdot\xi}f(\xi)\,d\sigma(\xi),
$$
with $d\sigma$ being the surface measure.
It is conjecured in \cite{FoKl} that \footnote{Conjecture 17.2} if
we let $\Omega_1,\Omega_2$ be two disjoint compact subsets of $\mathbb{S}$ such that
$$
\mathsf{dist}(\Omega_1,\Omega_2)>0,\;
\mathsf{dist}(\Omega_1,-\Omega_2)>0\;,
$$
then for every $p\ge\frac{n+2}{n}$, there is 
a finite constant $C=C_{\Omega_1,\Omega_2,p}>0$ such that
\begin{eqnarray}\label{eq:FK107}
\bigl\| Sf\cdot Sg\bigr\|_{L^p(\R^n)}\le C \|f\|_{L^2(\Omega_1)}\|g\|_{L^2(\Omega_2)}
\end{eqnarray}
holds for all $f$ supported in $\Omega_1$, and $g$ supported in $\Omega_2$.\medskip

Nowadays, this question is usually referred as the  bilinear restriction theorems on elliptic type surfaces. In \cite{TaoGFA}, the sharp
result was established by  Tao for all $p>1+\frac{2}{n}$, except the endpoint. The purpose of this section is to invest a possible way of combining the method 
of descent with the argument for Theorem \ref{pp:sigma-0} to get  the endpoint case.
To illustrate the idea, we assume $\Omega_1,\Omega_2$  are contained in a small neighbourhood of the north pole $e_n=(0,\ldots,0,1)$. By modulating the integrand in the $L^p-$norm  on the left side of \eqref{eq:FK107}, it is natural to consider the
Monge function
 $\Psi(\xi')=\sqrt{1-|\xi'|^2}-1$ and
introducing the $\lambda-$dependent operator
$$
S^\lambda_{\Omega} f(x,t)=\int_{-1}^1\int_{\Omega}e^{i(x'\cdot\xi'+ts+x_n\frac{\Psi(\xi')}{\lambda+s})}f(\xi',s)d\xi' ds, 
$$
where $t,s$ are auxiliary variables.

Taylor expand $\Psi(\xi')=\frac{|\xi'|^2}{2}+\mathcal{ E}(\xi')$ for $|\xi'|\ll 1$ and look at the surface
$$
\varSigma^\lambda:=\{(\xi',s,\,\Psi(\xi')/(\lambda+s)):|\xi'|\ll 1, |s|\le 1\} 
$$
Ignoring the error term  $\frac{\mathcal{ E}(\xi')}{\lambda+s}$,
the main term is a paraboloid as in \cite{Y22}, so that 
normal directions at each point of $\varSigma^\lambda$ should be almost contained in a conic surface due to the main term. To justify this rigorously, it remains to 
 control the perturbation term. To this end, one may introduce a similar class of functions including  $\frac{\Psi(\xi')}{\lambda+s}$ as a special case such that the stabiliy result holds  as $\lambda\to+\infty$

 Let $I=[-1,1]$ and $\mathscr{ C}^\lambda(R)$ be the optimal constant $C$ for which 
 $$
 \|S^\lambda_{\Omega_1}f\cdot S^\lambda_{\Omega_2}g\|_{L^p(Q^\lambda_R)}\le C \lambda^{1/p} \|f\|_{L^2(\Omega_1\times I)}\|g\|_{L^2(\Omega_2\times I)}
 $$
 for all $f,g$ and all $\lambda-$stretched $R$-cube $Q^\lambda_R\subset \R^{n+1}_{x,t}$. Here, by stretch we mean multiplying $\lambda$ to the side along $x_n-$direction.
 Construct the wave tables 
  and use bootstrap under the condition $R\le\lambda$ ultimately to get
 $\mathscr{ C}^\lambda(R)\lesssim 1$ for all $R\le \lambda$.

Let $\lambda=R$.
Changing variables, and letting $R\to\infty$ using Lebesgue's dominated convergence theorem and then Fatou's lemma, we obtain after integrating $t$ out

\begin{theorem}
	Under the above conditions on $\Omega_1,\Omega_2$, \eqref{eq:FK107} holds for all $p\ge \frac{n+2}{n}$.
\end{theorem}


\begin{thebibliography}{99}


\bibitem{AI}
 Antoni\'c,  N. and Ivec, I., 
\emph{On the H\"ormander-Mihlin theorem for mixed-norm Lebesgue spaces.}
J. Math. Anal. Appl.  \textbf{433} (2016), no. 1, 176--199.
MR3388786.

\bibitem{BCT}
 Bennett,  J.;   Carbery, A., and Tao, T., 
\emph{On the multilinear restriction and Kakeya conjectures.}
Acta. Math. \textbf{196} (2006), no. 2, 261--302.
MR2275834.

\bibitem{Bo95}
Bourgain, J.
\emph{Estimates for cone multipliers.}
Geometric aspects of functional analysis (Israel, 1992-1994), 41--60,
Oper. Theory Adv. Appl.,  \textbf{77}, Birkh\"auser, Basel, 1995. MR1353448.



\bibitem{CaV}
Carbery, A., and Valdimarsson, S.
\emph{The endpoint multilinear Kakeya theorem via the Borsuk-Ulam theorem.}
J. Funct. Anal. 264 (2013), no. 7,  1643-1663. MR3019726.


\bibitem{DeBook} 
Demeter, C.  \emph{Fourier restriction, decoupling, and applications.} Cambridge Studies in Advanced Mathematics,  \textbf{184}. Cambridge University Press, Cambridge 2020. xvi+331 pp. ISBN:978-1-108-49970-5. MR3971577.


\bibitem{FoKl} 
Foschi, D. and Klainerman, S.
\emph{Bilinear space-time estimates for homogeneous wave equations.}
Ann. Sci. \'{E}cole Norm. Sup. (4) \textbf{33} (2000), no. 2, 211--274. MR 1755116.

\bibitem{GiMe}
Del Nin, Giacomo, and Merlo, Andrea,
\emph{Endpoint Fourier restriction and unrectifiability.}
Proc. Amer. Math. Soc. \textbf{150}(2022), no. 5, 2137-2144. MR 4392348.

\bibitem{guth} 
Guth, L.
\emph{The endpoint case of the Bennett-Carbery-Tao multilinear Kakeya conjecture.}
Acta Math. 205(2010), no. 2, 263-286.
MR2746348.

\bibitem{Hab}
Haberman, B.
\emph{Uniqueness in Calder\'{o}n's problem for conductivities with unbounded gradient.}
Comm. Math. Phys. 340 (2015), no. 2, 639-659. MR3397029.

\bibitem{HKL}
Ham, S.; Kown, Y., and Lee, S.
\emph{Uniqueness in the Calder\'{o}n problem and bilinear restriction estimates.} 
J. Funct. Anal. 281 (2021), no. 8, Paper No.109119,58pp. MR4273826.


\bibitem{KT} 
 Klainerman, S., and  Tataru, D.
\emph{On the optimal local regularity for Yang-Mills equations in $\R^{4+1}$.}
J. Amer. Math. Soc. \textbf{12} (1999), no. 1, 93--116.
MR1626261.



\bibitem{LeeRogersVargas}
Lee, S., Rogers, K. M.,  and Vargas, A. \emph{Sharp null form estimates for the wave equation in $\R^{3+1}$.} Int.  Math. Res. Not., IMRN
 (2008),  Art. ID rnn 096, 18,pp. MR2439536.



\bibitem{LeeVargas}
 Lee, S.  and Vargas, A. \emph{Sharp null form estimates for the wave equation.} Amer. J. Math
\textbf{130} (2008), no. 5, 1279--1326. MR2450209.

\bibitem{FPV}
Ponce-Vanegas, F.
\emph{A bilinear strategy for Calder\'{o}n's problem.}
Rev. Mat. Iberoam. 37 (2021), no. 6,  2119-2160. MR4310288.


%\bibitem{SteinHA}
%Stein, Elias. M. \emph{Harmonic analysis: real-variable methods, orthogonality, and oscillatory integrals.}
%With the assistance of Timothy S. Murphy. Princeton Mathematical Series, 43. Monographs in Harmonic Analysis, III. Princeton University Press, NJ, 1993. xiv+695 pp. ISBN:0-691-03216-5. MR1232192.


\bibitem{TaoMZ}  Tao, T.
\emph{Endpoint bilinear restriction theorems for the cone, and some sharp null form estimates.}
Math. Z. \textbf{238} (2001), no. 2, 215--268. MR1865417.


\bibitem{TaoGFA}  Tao, T.
\emph{A sharp bilinear restrictions estimate for paraboloids. }
Geom. Funct. Anal. \textbf{13} (2003) no. 6, 1359-1384. MR 2033842.


\bibitem{TVV}
Tao, T.;  Vargas, A. and Vega, L. \emph{A bilinear approach to the restriction and Kakeya conjectures.}
J. Amer.  Math. Soc. \textbf{11} (1998), no. 4, 967--1000. MR1625056.

\bibitem{TV-1}
Tao, T. and Vargas, A. \emph{A bilinear approach to cone multipliers I: Restriction estimates.} Geom. Funct. Anal. \textbf{10} (2000), no.1, 185--215. MR1748920.

\bibitem{TV-2}
Tao, T.  and  Vargas, A.  \emph{A bilinear approach to cone multipliers II: Applications.} Geom. Funct. Anal. \textbf{10} (2000), no. 1, 216--258. MR1748921.

\bibitem{Tataru}
Tataru, D., \emph{Null form estimates for second order hyperbolic  operators with rough coefficients.}
Harmonic  analysis at Mount Holyoke, (South Hadley, MA, 2001),  383-409, Contemp. Math.,  320, Amer. Math. Soc., Providence, RI, 2003 . MR1979953.

\bibitem{Temur}  Temur, F.
\emph{An endline bilinear cone restriction estimate for mixed norms.}
Math. Z. \textbf{273} (2013), no. 3-4, 1197--1214. MR3030696.


\bibitem{Wolff99}  Wolff, T.
\emph{A sharp bilinear cone restriction estimate.}
Ann. of Math.  (2) \textbf{153} (2001) no. 3, 661-698.
MR1836285.


\bibitem{Y22}
Yang, J.
\emph{An endline bilinear restriction estimate for  paraboloids.}
\texttt{arXiv:math/2202.13905v2}.


\end{thebibliography}
\end{document}